\documentclass[a4paper,10pt]{article}
\usepackage[T1]{fontenc}
\usepackage[utf8]{inputenc}
\usepackage[english]{babel}

\usepackage{geometry}
\usepackage{lipsum}
\usepackage[nottoc]{tocbibind}
\usepackage[hidelinks]{hyperref}
\usepackage{xcolor}
\usepackage{easyReview} 
	
\usepackage{microtype}

\usepackage{amsmath, amsthm, amssymb, mathtools}
\usepackage{siunitx}
\usepackage{resizegather}
\usepackage{bm, bbm, mathrsfs}

\usepackage{pgfplots, pgfplotstable, booktabs, array}
	\pgfplotsset{compat=1.17}
	\pgfplotstableset{
		every head row/.append style=%
			{before row=\toprule, after row=\midrule},
		every last row/.append style=%
			{after row=\bottomrule}
	}
\usepackage{caption, subcaption}
\usepackage{algorithm,algpseudocode}

\newcommand{\N}{\mathbb{N}}
\newcommand{\Z}{\mathbb{Z}}

\newcommand{\R}{\mathbb{R}}

\newcommand{\abs}[1]{\left\lvert#1\right\rvert}
\newcommand{\absB}[1]{\Big\lvert#1\Big\rvert}
\newcommand{\norm}[1]{\left\lVert#1\right\rVert}

\newcommand{\seminorm}[1]{\left\lvert#1\right\rvert}

\newcommand{\restrict}[1]{|_{#1}}
\newcommand{\deq}{\vcentcolon=}

\newcommand{\dmu}{\,d\mu}
\newcommand{\dmux}{\,d\mu(x)}
\newcommand{\muw}{w}
\newcommand{\dsigma}{\,d\sigma}
\newcommand{\dsigmax}{\,d\sigma(x)}
\newcommand{\sigmav}{v}

\newcommand{\eps}{\varepsilon}

\DeclareMathOperator{\Ker}{Ker}
\DeclareMathOperator{\Img}{Im}

\DeclareMathOperator{\spann}{span} %
\DeclareMathOperator{\diver}{div}
\DeclareMathOperator{\nnz}{nnz}

\DeclareMathOperator{\round}{round}

\theoremstyle{definition}
\newtheorem{lemma}{Lemma}[section]

\newtheorem{proposition}[lemma]{Proposition}
\newtheorem{corollary}[lemma]{Corollary}

\newtheorem{example}[lemma]{Example}
\newtheorem{algotheorem}[lemma]{Algorithm}
\allowdisplaybreaks

\newcommand{\Omegabar}{\overline{\Omega}}
\newcommand{\bbox}{H}
\newcommand{\Acal}{\mathcal{A}}
\newcommand{\Bcal}{\mathcal{B}}

\newcommand{\Dcal}{\mathcal{D}}
\newcommand{\Lcal}{\mathcal{L}}
\newcommand{\Mcal}{\mathcal{M}}
\newcommand{\Ocal}{\mathcal{O}}

\newcommand{\Scal}{\mathcal{S}}
\newcommand{\Ucal}{\mathcal{U}}

\newcommand{\Qcal}{\mathcal{Q}}
\newcommand{\fhat}{	\hat{f}}
\newcommand{\ghat}{	\hat{g}}

\begin{document}

\title{Meshless moment-free quadrature formulas \\
arising from numerical differentiation}
\author{
Oleg Davydov\thanks{Department of Mathematics,
	University of Giessen, Arndtstrasse 2, 35392 Giessen, Germany,
	\url{oleg.davydov@math.uni-giessen.de}}, \hspace*{4pt}
Bruno Degli Esposti\thanks{Department of Mathematics
	and Computer Science, University of Florence,
	Viale Morgagni 67, 50134 Firenze, Italy,
	\url{bruno.degliesposti@unifi.it}}
}
\date{May 31, 2025}
\maketitle

\begin{abstract}
We suggest a method for simultaneously generating high order
quadrature weights for integrals over %
Lipschitz domains and their boundaries that requires
neither meshing nor moment computation.
The weights are computed on pre-defined scattered nodes as a
minimum norm solution of a sparse underdetermined linear system
arising from a discretization of a suitable boundary value
problem by either collocation or meshless finite differences.
The method is easy to implement independently of domain's
representation, since it only requires as inputs the
position of all quadrature nodes and the direction
of outward-pointing normals at each node belonging to the boundary.
Numerical experiments demonstrate the
robustness and high accuracy of the method on a number of smooth
and piecewise smooth domains in 2D and
3D, including some with reentrant corners and edges.
Comparison with quadrature schemes provided by the 
state-of-the-art open source packages Gmsh and MFEM
shows that the new method is competitive in terms
of accuracy for a given number of nodes.

\end{abstract} 

\graphicspath{{./figures/}}

\section{Introduction}
\label{intro}

Efficient high order numerical integration over
piecewise smooth domains $\Omega$ in $\R^2$
and $\R^3$ and on surfaces is a fundamental tool
in numerical analysis and engineering applications,
see for example extensive literature reviews in the recent articles
\cite{divi2020error,gunderman2021high}. In this paper we introduce
a new technique for generating quadrature
formulas on scattered nodes that only requires as inputs
the outward-pointing normals along $\partial\Omega$
at the boundary nodes, in addition to the positions
of the nodes themselves. The method relies on the discretization
of the divergence operator
by numerical differentiation  and the least squares solution
of a sparse underdetermined linear system.

Given a set of nodes
$Y=\{y_1,\ldots,y_N\}\subset\overline{\Omega}$,
the classical way to obtain a weight vector
$\muw = (\muw_i)_{i=1}^N\in\R^N$ of a quadrature formula
\begin{equation} \label{qf}
\int_{\Omega}f(x)\dmu(x)\approx \sum_{i=1}^N\muw_if(y_i)
\end{equation}
for the integral with respect to some measure
$\mu$ over $\Omega$ is to choose a finite
dimensional space
$\Scal = \spann \{ s_1,\dots,s_M \}$ of continuous functions
with good approximation properties, such as
(piecewise) polynomials or radial basis functions, and require
the exactness of the formula for all $f\in\Scal$, which is
equivalent to the exact reproduction of the \emph{moments}
$\int_{\Omega}s_j\dmu$,
\begin{equation} \label{exact}
\int_{\Omega}s_j(x)\dmu(x)=\sum_{i=1}^N\muw_is_j(y_i),\quad
j=1,\ldots,M.
\end{equation}
This design principle based on exactness
(called \emph{moment fitting} by some authors)
leads to a linear system for the weight vector $\muw$,
as soon as all the moments are known. In particular, the book
\cite{stroud1971approximate} presents many quadrature formulas
for special multivariate regions such as cubes, simplices, balls
or spheres, and special node sets, where exactness is enforced
for spaces of polynomials, and the moments are computed by
analytic methods. With the help of the substitution rule these
formulas can be adapted to the smooth images of such elementary
regions. A more complicated domain may be partitioned into mapped
elementary regions leading to a compound quadrature formula
generated by summing up the formulas over this partition. This
approach is used in particular in the finite element method.
To achieve high order accuracy of integration, however,
it requires challenging high-order mesh generation, 
a task that is often prohibitively costly and non-robust
in applications
\cite{turner2016variational,gunderman2021high}.
In the framework of isogeometric analysis,
the elementary regions are defined by spline patches,
and the moment fitting principle
is applied with B-splines
as basis functions
\cite{antolin2019isogeometric,zou2022efficient}.
Instead of meshing into images of elementary regions,
different types of partitioning may be used,
followed by the dimension reduction approach
that replaces volume or surface integrals by repeated
univariate integration, such as the standard iterated
integration on generalized rectangles or sectors
\cite{stroud1971approximate,shampine2008matlab}. Methods of
this type have in particular been developed for domains defined
by a level set function \cite{muller2013highly,saye2015high}. 
Recent versions of the dimension reduction method
\cite{sommariva2009gauss,gunderman2021high,gunderman2021spectral}
rely on the divergence theorem
over arbitrary subregions with known parametrizations
of their piecewise smooth boundaries.
However, a significant drawback of these 
methods is that many quadrature nodes must be placed
outside of $\Omegabar$ in its bounding box, unless
the shape of the integration domain satisfies some restrictive
assumptions (see e.g.\ the concept of a base-line
in \cite{sommariva2009gauss}).

Another line of research is to look for (near-)optimal node
placement to achieve as small as possible number $N$ of nodes satisfying the
exactness condition \eqref{exact} for a given space $\Scal$,
which requires solving nonlinear problems 
\cite{auricchio2012simple,bartovn2017gauss,hashemian2023solving,mousavi2010generalized} 
or infinite-dimensional linear programs \cite{ryu2015extensions}.
High order quadrature formulas with equal weights are studied 
for special domains in $\R^d$ such as a cube or a sphere by
quasi-Monte Carlo methods \cite{brauchart2014qmc,dick2013high}.

When the quadrature nodes in $Y$ are fixed in advance
(for example, they may be supplied by the user, so that
we are not in control of their placement)
and the task is to generate the weight vector $\muw$,
the problem setting is related to scattered data fitting,
because the integral
$\int_{\Omega}f\dmu$ can be approximated by
$\int_{\Omega}s\dmu$ if $s$ is an approximation of $f$, and
this leads to a quadrature formula as soon as $s$ is obtained for
each fixed $Y$ as $s=\Qcal_Y f|_Y$, with $\Qcal_Y:\R^N\to\Scal$
being a linear scattered data fitting operator and 
$f|_Y:=(f(y_i))_{i=1}^{N}$. The most established tool
for scattered data fitting is kernel-based interpolation, in
particular with radial basis functions \cite{Wendland}. The
meshless
quadrature formulas obtained this way provide optimal recovery 
of the integration functional \cite{micchelli1985lectures}, in
particular polyharmonic and thin plate spline RBF give rise to
optimal quadrature formulas for multivariate Sobolev spaces 
\cite{sobolev1961formulas}. However, numerical implementation of
these methods for general domains is problematic because the
corresponding linear system \eqref{exact} is not sparse, and the
moments of the radial basis functions have to be computed. The 
first problem may be tackled by either partitioning $\Omega$ or
using  spaces $\Scal$ with locally supported bases such as
B-splines, but the computation of the moments remains a major
obstacle unless $\Omega$ is a polygonal
domain~\cite{sommariva2021rbf}. A dimension reduction approach for the
moments over curved triangles has been suggested
in~\cite{reeger2018numerical}. When using polynomial spaces $\Scal$
on scattered nodes, a feasible approach is to choose the
polynomial degree such that the  dimension $M$ of $\Scal$ is
significantly smaller than the number of nodes $N$, which
typically makes the system \eqref{exact} underdetermined and
solvable. Then a weight vector $\muw$ satisfying \eqref{exact} 
may be found as a minimum norm solution 
\cite{levin1999stable,fornberg-spherical,glaubitz2021stable,%
glaubitz2023construction,sommariva2023low} or by
subset selection \cite{SoVia09}. These methods require the
moments of the polynomials, in particular \cite{sommariva2023low}
suggests a sophisticated algorithm for computing them on
bivariate domains bounded by B-spline curves.

The method that we suggest does not require any meshing,
and is also \emph{moment-free}, in the sense that
the moment computation problem is completely bypassed.
It relies on the discretization of a pair of linear differential
operators $\Lcal$ and $\Bcal$ satisfying the identity
\begin{equation}\label{intp}
\int_\Omega \Lcal u\dmu=\int_{\partial\Omega}\Bcal u \dsigma.
\end{equation}
Many identities of this form can be obtained from
the generalized Stokes theorem.
For example, a natural choice is $\Lcal=\Delta$, the Laplace
operator (or Laplace-Beltrami operator when $\Omega$
is a Riemannian manifold), and $\Bcal=\partial_\nu$,
the outward normal derivative at the boundary.
This pair, however, may lead to inaccurate
quadrature formulas on domains with nonsmooth boundary,
 as will be demonstrated in the paper.
Instead, we recommend $\Lcal=\diver$, the divergence operator,
and $\Bcal=\gamma_\nu$, the normal trace operator, as a general
purpose choice, which works well in all of our numerical tests
on Lipschitz domains in 2D and 3D with either smooth
or piecewise smooth boundary.

Given two functions $f:\Omega\to\R$ and
$g:\partial\Omega\to\R$, we look for a \emph{combined quadrature}
to approximate the difference of two integrals
\begin{equation} \label{cqf}
\int_{\Omega}f(x)\dmu(x)-\int_{\partial\Omega}g(x)\dsigma(x) 
\approx\sum_{i=1}^{N_Y}\muw_if(y_i)- \sum_{i=1}^{N_Z}\sigmav_ig(z_i),
\end{equation}
using two sets of quadrature nodes
$Y=\{y_1,\ldots,y_{N_Y}\}\subset\overline{\Omega}$ and
$Z=\{z_1,\ldots,z_{N_Z}\}\subset\partial \Omega$ and two weight
vectors $\muw =(\muw_i)_{i=1}^{N_Y}\in\R^{N_Y}$ and
$\sigmav =(\sigmav_i)_{i=1}^{N_Z}\in\R^{N_Z}$.
Clearly, \eqref{cqf} provides in particular a quadrature
formula for $\int_{\Omega}f\dmu$ with nodes in $Y$ and weight vector 
$\muw$ if we set $g\equiv 0$, and a
quadrature formula for $\int_{\partial\Omega}g\dsigma$ with nodes
in $Z$ and weight vector $\sigmav$ if we set $f\equiv 0$.
Assuming that the nodes in $Y$ and $Z$ are provided by the user,
or placed by a node generation algorithm \cite{suchde-point},
we propose a method for the computation of the weight
vectors $\muw$ and $\sigmav$. 

As in the classical approach based on the reproduction of moments
\eqref{exact}, we may employ a finite
dimensional space of functions $\Scal=\spann\{ s_1,\ldots, s_M\}$. 
However, we rely on the simultaneous approximation
of $f$ by $\Lcal s$ and $g$ by $\Bcal s$,
rather than on the direct approximation of $f$ by $s\in\Scal$.
Therefore, we require the exactness of the combined
formula \eqref{cqf} for all
pairs $f,g$ obtained as $f=\Lcal s$ and $g=\Bcal s$ from
the same $s\in\Scal$. By the assumptions on operators $\Lcal$
and $\Bcal$, the \emph{combined moment}
\[
\int_{\Omega}f(x)\dmu(x)
- \int_{\partial\Omega}g(x)\dsigma(x)
= \int_{\Omega} \Lcal s(x)\dmu(x)
- \int_{\partial\Omega} \Bcal s(x)\dsigma(x)
\]
is equal to zero, and so its value is trivially known
in advance, which is the key advantage of our moment-free
quadrature formulas over classical moment-fitting.
Imposing exactness of \eqref{cqf} for the pairs $f=\Lcal s_j$,
$g=\Bcal s_j$, $j=1,\ldots,M$, leads to
the homogeneous linear equations for the weights
$$
\sum_{i=1}^{N_Y}\muw_i\Lcal s_j(y_i)-
\sum_{i=1}^{N_Z}\sigmav_i\Bcal s_j(z_i)=0\quad 
\text{for all } j=1,\ldots,M.
$$
In addition, we require the exactness of \eqref{cqf} for a single
pair
of functions $\hat f,\hat g$ with known difference of the integrals 
$\int_\Omega\hat f\dmu$ and $\int_{\partial\Omega}\hat g\dsigma$, 
such that 
\begin{equation}\label{fghatcond}
\int_\Omega\hat f(x)\dmu(x)-\int_{\partial\Omega}\hat g(x)\dsigma(x)\ne0.
\end{equation}
In particular, we may use $\hat f\equiv 1$, $\hat g\equiv 0$ if
we know the measure of $\Omega$, or $\hat f\equiv 0$, $\hat
g\equiv 1$ if we know the boundary measure of $\partial\Omega$. 
If neither is known, then a purely moment-free approach
is to employ $\hat f\equiv 0$, $\hat g=\partial_\nu\Phi$,
with $\Phi$ being a fundamental solution
of the Laplace(-Beltrami) equation centered at a point in $\Omega$,
see Section~\ref{ssec:on-the-choice-of-fhat-and-ghat}.
In any case, we get the following linear system for
the weight vectors $\muw$ and $\sigmav$,
\begin{align}
\begin{cases}\hspace{-8pt}
\begin{split}\label{wqcond}
\sum_{i=1}^{N_Y} w_i \ell_{ij}
	- \sum_{i=1}^{N_Z} v_i b_{ij}
&= 0,\quad j=1,\ldots,M,\\
\sum_{i=1}^{N_Y}\muw_i\hat f(y_i)- \sum_{i=1}^{N_Z}\sigmav_i\hat g(z_i) 
&= \int_{\Omega}\hat f\dmu-\int_{\partial\Omega}\hat
g\dsigma,
\end{split}
\end{cases}
\end{align}
where $\ell_{ij}$ and $b_{ij}$ are the entries of the collocation matrices
\begin{equation}\label{collocLB}
L=\big(\Lcal s_j(y_i)\big)_{i=1,j=1}^{N_Y,M},\quad B=\big(\Bcal s_j(y_i)\big)_{i=1,j=1}^{N_Z,M}
\end{equation}
for the operators
$\Lcal,\Bcal$ with respect to the sets $Y$, $Z$ and the
space $\Scal$. If $\Omega$ is a closed manifold,
then $\partial\Omega=\emptyset$, and \eqref{wqcond}
simplifies in the sense that everything related to the
integral over $\partial\Omega$ disappears.

Moreover, we may move away from using any function space $\Scal$
and replace $L$ and $B$ by numerical differentiation matrices on
irregular nodes, as those used in the meshless finite difference
method, see for example~\cite{FFprimer15,JTAMB19}, which leads to
a purely meshless approach with full flexibility of arbitrarily
distributed degrees of freedom  associated with an additional
finite set $X=\{x_1,\ldots,x_{N_X}\}$ discretizing either
$\overline{\Omega}$ or some larger set $D\supset\overline{\Omega}$.
In this case  we again compute
$(w,v)$ by solving the linear system~\eqref{wqcond},
with the coefficients $\ell_{ij},b_{ij}$ of matrices $L,B$
obtained as the weights of suitable numerical differentiation formulas
\begin{equation}\label{ndf}
\Lcal u(y_i)\approx\sum_{k=1}^{N_X}\lambda_{ik}^Tu(x_k),
\quad i=1,\ldots, N_Y,
\qquad\Bcal u(z_i)\approx\sum_{k=1}^{N_X}\beta_{ik}^Tu(x_k),
\quad i=1,\ldots,N_Z,
\end{equation}
where
\begin{equation}\label{ndfLB}
L=\big(\ell_{ij}\big)_{i=1,j=1}^{N_Y,M}
=\big(\lambda_{ik}^T\big)_{i=1,k=1}^{N_Y,N_X},
\quad B=\big(b_{ij}\big)_{i=1,j=1}^{N_Z,M}
=\big(\beta_{ik}^T\big)_{i=1,k=1}^{N_Z,N_X}.
\end{equation}
Note that $\lambda_{ij}$ are scalars if $u$ is scalar-valued, but
in general they are vectors, for example when
$\Lcal$ is the divergence operator on a domain
in $\R^d$ or an embedded manifold, and thus $u$
is a vector field.

We choose the functional space $\Scal$ or the set of nodes $X$
such that the system \eqref{wqcond} is underdetermined,
and compute the combined quadrature weight vector $(\muw,\sigmav)$
as its solution with the smallest 2-norm,
\begin{align}
\begin{split}\label{LS}
\|\muw \|_2^2+\|\sigmav \|_2^2&\to\min\\
\text{subject to}\quad 
\begin{pmatrix}
L^T & -B^T \\ 
\hat{f}|_Y^T & -\hat{g}|_Z^T
\end{pmatrix}
\begin{pmatrix}
\muw  \\ 
\sigmav 
\end{pmatrix}
&= \begin{pmatrix}
0 \\
\int_\Omega \hat{f} \dmu -\int_{\partial\Omega} \hat{g} \dsigma
\end{pmatrix},
\end{split}
\end{align}
where 
$$
\hat{f}|_Y=\big(\hat{f}(y_i)\big)_{i=1}^{N_Y},\quad 
\hat{g}|_Z=\big(\hat{g}(z_i)\big)_{i=1}^{N_Z},$$
which is a linearly constrained least squares problem. Efficient
algorithms for this problem are available and can handle very large
node sets $Y,Z$ if the constraints are given by a sparse matrix
\cite{Bjorck96}.
In particular, we use direct methods based on
either the sparse QR decomposition of the system matrix,
or the sparse Cholesky decomposition of the normal equations
of the second kind associated with \eqref{LS}.
The two algorithms are provided by the same open
source library SuiteSparse under the submodules
SuiteSparseQR and CHOLMOD,
respectively~\cite{chen2008algorithm,davis-algorithm,SuiteSparse}.
The sparsity of the matrices
$L$ and $B$ is achieved by using a locally supported basis of
$\Scal$, such as tensor-product B-splines, or by choosing small sets of
influence in the numerical differentiation formulas \eqref{ndf},
with $\lambda_{ik}\ne 0$ and $\beta_{ik}\ne 0$ only for small
distances between $y_i,z_i$ and $x_k$.

This leads to a scheme that uniquely combines the following features.
\begin{itemize}

\item High order quadrature weights are provided for 
user-supplied nodes on %
Lipschitz domains,
with the only additional input of outward-pointing normals
at the boundary nodes.
Therefore the method is not tied to any specific representation
of the integration domain, e.g.\ parametric or implicit.

\item No meshing is needed, which makes this method particularly
well-suited for the meshless workflow in numerical computations
that attracts growing attention as an alternative to mesh-based
methods, because mesh generation remains challenging in more
complicated applications \cite{suchde-point}.

\item The scheme is free from the moment computation problem,
the typical bottleneck of meshless quadrature methods
on complex integration domains, especially in dimensions
higher than two.

\item Simple and easy to implement algorithms rely on a
discretization of a boundary value problem for the divergence
operator and a minimum-norm solution to an underdetermined sparse
linear system, for which  efficient linear algebra routines are
available.

\end{itemize}

We develop the details of both  the meshless finite difference
approach based on the numerical differentiation matrices
\eqref{ndfLB}, and the functional approach based on the
collocation matrices \eqref{collocLB}, and formulate two
practical algorithms for domains in $\R^d$ that use polyharmonic
kernel numerical differentiation (Algorithm \ref{algo:mfd}) and
collocation with tensor-product B-splines (Algorithm
\ref{algo:spline}), respectively.

Extensive numerical experiments demonstrate the effectiveness
of both algorithms, with the largest tests involving about
450k quadrature nodes, and convergence orders beyond $h^7$ for
both interior and boundary quadrature on several 2D and 3D domains,
including some with reentrant corners and edges.
In addition to SuiteSparse, our
MATLAB code relies on the numerical differentiation formulas
provided by mFDlab \cite{mFDlab}, and on several MATLAB
toolboxes (optimization, curve fitting, statistics).

The paper is organized as follows. Section~\ref{sec:qf} presents
the algorithms for the computation of the quadrature weights and
their theoretical background, while
Section~\ref{sec:numerical-tests}  is devoted to numerical
experiments. We start Section~\ref{sec:qf} by an abstract error
analysis that motivated our approach 
(Section~\ref{ssec:error-analysis-in-the-abstract-setting}), and
then describe in Section~\ref{ssec:algorithm} a general framework
for the development of quadrature formulas of the new type,
followed by the details of the options available for each of
its aspect, namely the choice of the functions $\fhat$ and
$\ghat$ satisfying \eqref{fghatcond} in 
Section~\ref{ssec:on-the-choice-of-fhat-and-ghat}, differential
operators $\Lcal$ and $\Bcal$ in 
Section~\ref{ssec:choice-of-operators-Lcal-and-Bcal}, methods for
numerical differentiation by either collocation or meshless finite
differences in  Section~\ref{ssec:on-the-choice-of-L-and-B}, and
the methods for solving the underdetermined linear system 
\eqref{wqcond} in 
Section~\ref{ssec:on-the-choice-of-minimization-norm}. Two
practical, ready to implement algorithms are presented in
Section~\ref{ssec:palgos}.

Section~\ref{ssec:validation-of-parameters} covers
the numerical validation of the parameters suggested in
Algorithms~\ref{algo:mfd} and \ref{algo:spline}, 
Section~\ref{ssec:conv} is devoted to testing the
empirical order of convergence of the methods,
and Section~\ref{ssec:comp} compares the performance of our
scheme on a solid torus and a 3D
domain defined implicitly by a surface of genus 3 to that of the
quadrature schemes provided by the open source packages Gmsh and
MFEM.
Finally, Section~\ref{sec:concl} gives a brief
conclusion and remarks about directions of future work.

\graphicspath{{./figures/}}

\section{Quadrature formulas}\label{sec:qf}

We begin by fixing notation.
The following setting, although abstract, lets us
develop our error analysis in full generality by allowing
integration on manifolds and on nonsmooth domains, and
by allowing functions with discontinuities
or other kinds of singularities.

Let $\Mcal$ be a $C^0$ manifold of dimension $d$,
and $\Omega\subset\Mcal$ a domain (that is, a connected
and open subset) with compact closure $\Omegabar$,
and let $\mu,\sigma$
be finite, strictly positive Borel measures defined
on $\Omegabar$ and $\partial\Omega$, respectively.
We suppose that $\mu(\partial\Omega) = 0$, so that
the space $L^1(\Omega)$ of absolutely integrable functions 
can be identified with $L^1(\Omegabar)$.

For any given set $Y$ of $N_Y$ scattered nodes in $\Omegabar$,
and, in the case of nonempty boundary $\partial\Omega\ne\emptyset$, 
any given set $Z$ of $N_Z$ scattered nodes in 
$\partial\Omega$, we seek quadrature weights
$\muw  \in \R^{N_Y}$ and $\sigmav  \in \R^{N_Z}$ such that
\begin{equation}\label{eq:qf}
\int_{\Omega} f(x) \dmux \approx \sum_{i=1}^{N_Y} \muw_i f(y_i)
\quad \text{and} \quad
\int_{\partial\Omega} g(x) \dsigmax
\approx \sum_{i=1}^{N_Z} \sigmav_i g(z_i)
\end{equation}
for all sufficiently regular integrands 
$f \colon \Omegabar \to \R$ and $g \colon \partial\Omega \to\R$.
To make the pointwise evaluation $f(y_i)$ and $g(z_i)$
a meaningful operation we assume that $f\in L^1_Y(\Omega)$ and
$g\in L^1_Z(\partial\Omega)$, where $L^1_Y(\Omega)$ and
$L^1_Z(\partial\Omega)$ are the subsets of
$L^1(\Omega)$ and $L^1(\partial\Omega)$ whose elements
(equivalence classes of functions) have at least one
representative that is continuous at each point of $Y$ and $Z$,
respectively.

We denote the two quadrature formulas by $(Y,w)$
and $(Z,v)$, and compute them \emph{simultaneously}.
Moreover, they arise as a \emph{combined quadrature}
\[
Q(f,g) \deq
\sum_{i=1}^{N_Y} \muw_i f(y_i)-\sum_{i=1}^{N_Z} \sigmav_i g(z_i)
\]
for the difference\footnote{Note that we could use the sum of the
two integrals instead of their difference, and respectively the
sum of the two quadrature formulas in the definition of the
combined quadrature, which would look more natural in a sense.
The difference, however, matches the usual form of the
compatibility condition~\eqref{eq:compatibility-of-L-and-B}
and emphasizes the "cancellation" nature of the
equation $I(f,g)=0$ that plays a crucial role in the method.}
of the two integrals
\[
I(f,g) \deq
\int_{\Omega} f(x) \dmux-\int_{\partial\Omega} g(x)\dsigmax.
\]
Let $\Lcal \colon \Dcal \to L^1(\Omega)$ and
$\Bcal \colon \Dcal \to L^1(\partial\Omega)$
be operators such that
\begin{equation} \label{eq:compatibility-of-L-and-B}
\int_\Omega \Lcal u \dmu
= \int_{\partial\Omega} \Bcal u \dsigma
\quad \text{for all $u \in \Dcal$,}
\end{equation}
with $\Dcal$ being a common domain of definition for the
two operators. Even though $\Dcal$ can in general be a set
of any nature, with the abstract theory presented in
Sections~\ref{ssec:error-analysis-in-the-abstract-setting}--\ref{ssec:on-the-choice-of-fhat-and-ghat}
applicable in this setting, we have in mind the situation
where $\Dcal$ is a linear space of functions  or vector fields on
$\Omegabar$ (or on a set containing it), $\Lcal$ and $\Bcal$
are both linear differential
operators, measures $\mu,\sigma$ are induced by a Riemannian
metric in $\Mcal$, and identity
\eqref{eq:compatibility-of-L-and-B} holds by the divergence
theorem or another corollary of the generalized Stokes theorem.
When $\Omega=\Mcal$ is a closed manifold with $\partial\Omega=\emptyset$, 
the measure $\sigma$ and the operator $\Bcal$ are
clearly irrelevant, and the right-hand
side in  \eqref{eq:compatibility-of-L-and-B}
is understood to be identically zero.

The underlying idea of the method may be roughly explained as follows.
The identity \eqref{eq:compatibility-of-L-and-B} implies that 
$$
I(f,g)=I(\Lcal u,\Bcal u)=0$$ 
as soon as $f=\Lcal u$ and $g=\Bcal u$ for the same function
$u\in \Dcal$. If $u$ is discretized by a vector $c\in\R^M$,
and matrices $L \in \R^{N_Y \times M}$
and $B \in \R^{N_Z \times M}$ are such that
$$
Lc\approx \Lcal u \restrict{Y}=f\restrict{Y},\quad 
Bc\approx \Bcal u \restrict{Z}=g\restrict{Z},$$
then 
$$
Q(f,g)
\approx \muw^T L c - \sigmav^T B c
= (\muw^T L - \sigmav^T B) \, c,
$$
and hence requiring that the weight vectors $\muw,\sigmav$
satisfy 
$$
\muw^TL-\sigmav^TB=0$$
implies 
\[
Q(f,g)\approx 0=I(f,g)\quad\text{for all $(f,g)$ such that } 
f=\Lcal u, \: g=\Bcal u \text{ for some } u \in \Dcal.
\]
In addition, we require from $\muw,\sigmav$ that 
$$
Q(\fhat,\ghat)=I(\fhat,\ghat)\quad\text{for a single pair $(\fhat,\ghat)$ 
with }I(\fhat,\ghat)\ne0.$$
As we will demonstrate, this combination of
conditions produces accurate quadrature
formulas $(Y,\muw )$ and $(Z,\sigmav )$ for appropriate choices of
the operators $\Lcal,\Bcal$ and numerical
differentiation matrices $L,B$.

We first make these considerations more precise in 
Section~\ref{ssec:error-analysis-in-the-abstract-setting}, then
formulate in \ref{ssec:algorithm} a general framework 
for obtaining quadrature formulas of this type, 
and explore in Sections~\ref{ssec:on-the-choice-of-fhat-and-ghat}--\ref{ssec:on-the-choice-of-minimization-norm}
the options available for each of its steps.
The final subsection \ref{ssec:palgos} presents two
specific algorithms that we recommend for practical
use on piecewise smooth Lipschitz domains in $\R^d$,
and which are extensively tested in Section~\ref{sec:numerical-tests}.

\subsection{Error analysis in an abstract setting}
\label{ssec:error-analysis-in-the-abstract-setting}
Throughout the paper we will use the standard notation for the $p$-norm
of a vector $x \in \R^N$, 
\[
\norm{x}_p=\Big( \sum_{i=1}^N \abs{x_i}^p \Big)^{1/p}
\quad\text{if $1 \leq p < \infty$},
\qquad
\norm{x}_\infty=\max_{i=1,\dots,N} \abs{x_i}.
\]

For any given pair of operators $\Lcal$ and $\Bcal$ and finite sets 
$Y\subset\Omegabar$ and $Z\subset\partial\Omega$, any
triple $(\Psi,L,B)$, where $\Psi \colon \Dcal \to \R^M$,
$L \in \R^{N_Y \times M}$ and $B \in \R^{N_Z \times M}$,
is said to be a \emph{numerical differentiation scheme}
with discretization operator $\Psi$ and
differentiation matrices $L,B$.
We denote by $\eps_L(u)$ and $\eps_B(u)$ the recovery error
with which the matrices $L$ and $B$ approximate the
pointwise values of $\Lcal u$ and $\Bcal u$
using the information about $u\in\Dcal$ contained in $\Psi(u)$:
\begin{align}
\label{eq:epsL}
\varepsilon_L(u)&:=\norm{\Lcal u \restrict{Y} - L\Psi(u)}_\infty=
\max_{i=1,\dots,N_Y}\abs{\Lcal u (y_i) - \bigl(L\Psi(u)\bigr)_i},\\
\label{eq:epsB}
\varepsilon_B(u)&:=\norm{\Bcal u \restrict{Z} - B\Psi(u)}_\infty=
\max_{i=1,\dots,N_Z}\abs{\Bcal u (z_i) - \bigl( B\Psi(u)\bigr)_i}.
\end{align}
The vector $\Psi(u) \in \R^M$ 
may correspond to pointwise values of $u$ over a scattered
set of nodes, coefficients of splines or radial basis functions 
that approximate $u$,
Fourier coefficients, or other kinds of discrete information about $u$.
Tools of approximation theory may then be used to define
matrices $L$ and $B$ that accurately reconstruct
the pointwise values of $\Lcal u$ and $\Bcal u$
from a given vector $\Psi(u)$.

Our approach is motivated by the following proposition that
estimates the combined quadrature error
\[
\delta(f,g) \deq I(f,g)-Q(f,g).
\]

\begin{proposition}[Error bound for combined quadrature]
\label{prop:joint-error-estimate}
Let $(\fhat,\ghat) \in L^1_Y(\Omega) \times L^1_Z(\partial\Omega)$
be a fixed pair of functions such that
\[
I(\fhat,\ghat)%
\neq 0.
\]
Given any pair $(f,g) \in L^1_Y(\Omega) \times L^1_Z(\partial\Omega)$, 
let $\Ucal_{f,g} \subset \Dcal$ be the set of solutions $u$
to the boundary value problem
\begin{equation} \label{eq:bvp-joint-error-estimate}
\begin{cases}
\Lcal u = f - \alpha \fhat & \text{in $\Omega$} \\
\Bcal u = g - \alpha \ghat & \text{on $\partial\Omega$},
\end{cases}
\end{equation}
where the coefficient $\alpha \in \R$ is chosen so that 
\eqref{eq:bvp-joint-error-estimate} satisfies
the compatibility condition \eqref{eq:compatibility-of-L-and-B}:
\[
\alpha=I(f,g)/I(\fhat,\ghat).
\]
Then for any pair of quadrature formulas \eqref{eq:qf}, %
the following estimate of the combined error holds:

\begin{equation} \label{eq:joint-error-estimate}
\abs{\delta(f,g)} \leq \abs{\alpha} \hat{\varepsilon}
+ \inf_{u \in \Ucal_{f,g}} \Bigl\{
	\norm{\muw }_1 \varepsilon_L(u)
	+ \norm{\sigmav }_1 \varepsilon_B(u)
	+ \norm{\Psi(u)}_\infty \norm{L^T\muw -B^T\sigmav }_1
\Bigr\},
\end{equation}
where $\hat\eps:=|\delta(\fhat,\ghat)|$.
\end{proposition}

\begin{proof}
For all $u \in \Ucal_{f,g}$, by the linearity of $\delta$,
\[
\abs{\delta(f,g)}
= \abs{\delta(f - \alpha \fhat, g - \alpha \ghat)
	+ \delta(\alpha\fhat,\alpha\ghat)}
= \abs{\delta(\Lcal u, \Bcal u) + \alpha \delta(\fhat,\ghat)}
\leq \abs{\delta(\Lcal u, \Bcal u)} + \abs{\alpha} \hat{\varepsilon}.
\]
Moreover, by the compatibility condition
\eqref{eq:compatibility-of-L-and-B} and by the definitions
of $\varepsilon_L(u)$ and $\varepsilon_B(u)$,
\begin{align*}
\abs{\delta(\Lcal u, \Bcal u)}
&= \absB{
	\sum_{i=1}^{N_Y} \muw_i (\Lcal u)(y_i)
	- \sum_{i=1}^{N_Z} \sigmav_i (\Bcal u)(z_i)} \\
&= \absB{
	\sum_{i=1}^{N_Y} \muw_i (\Lcal u)(y_i)
	\pm \sum_{i=1}^{N_Y} \muw_i \bigl( L\Psi(u) \bigr)_i
	\pm \sum_{i=1}^{N_Z} \sigmav_i \bigl( B\Psi(u) \bigr)_i
	- \sum_{i=1}^{N_Z} \sigmav_i (\Bcal u)(z_i)} \\
&\leq \norm{\muw }_1 \varepsilon_L(u)
	+ \abs{\muw ^T L \Psi(u) - \sigmav ^T B \Psi(u)}
	+ \norm{v}_1 \varepsilon_B(u) \\
&\leq \norm{\muw }_1 \varepsilon_L(u)
	+ \norm{\sigmav }_1 \varepsilon_B(u)
	+ \norm{\Psi(u)}_\infty \norm{L^T\muw -B^T\sigmav }_1.
\end{align*}
Taking the infimum over $u \in \Ucal_{f,g}$ completes the proof.
\end{proof}

By choosing either $f\equiv0$ or $g\equiv0$, we 
obtain separate error bounds for both quadrature formulas
\eqref{eq:qf}.

\begin{corollary}[Separate error bounds]
\label{cor:split-error-estimates}
By the previous proposition, the special cases
\[
\delta(f,0)
= \int_\Omega f \dmu - \sum_{i=1}^{N_Y} \muw_i f(y_i),
\qquad \delta(0,g)
= -\int_{\partial\Omega} g \dsigma + \sum_{i=1}^{N_Z} \sigmav_i g(z_i),
\]
immediately imply the following error bounds for
quadrature formulas $(Y,\muw )$ and $(Z,\sigmav )$,
\begin{align}
\label{eq:split-error-estimate-f}
\absB{\int_\Omega f \dmu - \sum_{i=1}^{N_Y} \muw_i f(y_i)}
&\leq \abs{\alpha} \hat{\varepsilon}
+ \inf_{u \in \Ucal_f} \Bigl\{
	\norm{\muw }_1 \varepsilon_L(u)
	+ \norm{\sigmav }_1 \varepsilon_B(u)
	+ \norm{\Psi(u)}_\infty \norm{L^T\muw -B^T\sigmav }_1
\Bigr\}, \\
\label{eq:split-error-estimate-g}
\absB{\int_{\partial\Omega} g \dsigma - \sum_{i=1}^{N_Z} \sigmav_i g(z_i)}
&\leq \abs{\beta} \hat{\varepsilon}
+ \inf_{u \in \Ucal_g} \Bigl\{
	\norm{\muw }_1 \varepsilon_L(u)
	+ \norm{\sigmav }_1 \varepsilon_B(u)
	+ \norm{\Psi(u)}_\infty \norm{L^T\muw -B^T\sigmav }_1
\Bigr\},
\end{align}
with $\Ucal_f$ and $\Ucal_g$ being the sets of solutions to the
boundary value problems
\begin{equation} \label{eq:bvp-split-error-estimates}
\begin{cases}
\Lcal u = f - \alpha \fhat & \text{in $\Omega$} \\
\Bcal u =   - \alpha \ghat & \text{on $\partial\Omega$},
\end{cases}
\qquad\text{and}\qquad
\begin{cases}
\Lcal u =   - \beta \fhat & \text{in $\Omega$} \\
\Bcal u = g - \beta \ghat & \text{on $\partial\Omega$},
\end{cases}
\end{equation}
respectively, where  $\alpha=I(f,0)/I(\fhat,\ghat)$ and
$\beta=I(0,g)/I(\fhat,\ghat)$
are again determined
by the compatibility condition
\eqref{eq:compatibility-of-L-and-B}.
\end{corollary}

The case of a closed manifold simplifies to just one quadrature
formula.

\begin{corollary}[Manifold without boundary]\label{rem:noboundary}
If $\partial\Omega = \emptyset$, then the argumentation of
Proposition \ref{prop:joint-error-estimate}
leads to the following error estimate for the quadrature formula $(Y,\muw )$,
\begin{equation*}
\absB{\int_\Omega f \dmu - \sum_{i=1}^{N_Y} \muw_i f(y_i)} 
\leq \abs{\alpha} \hat{\varepsilon}
+ \inf_{u \in \Ucal_f} \Bigl\{
	\norm{\muw }_1 \varepsilon_L(u)
	+ \norm{\Psi(u)}_\infty \norm{L^T\muw }_1
\Bigr\},
\end{equation*}
where $\Ucal_f$ is the set of solutions to the
equation $\Lcal u = f - \alpha \fhat$ in $\Omega$, and
$\alpha=\int_\Omega f \dmu/\int_\Omega \fhat \dmu$, with
$\fhat \in L_Y^1(\Omega)$ such that $\int_\Omega \fhat \dmu\ne0$.
\end{corollary}

\subsection{A general framework}
\label{ssec:algorithm}

It is clear from Proposition \ref{prop:joint-error-estimate} that 
the error $|\delta(f,g)|$ of the combined quadrature is small if 
$$
\abs{\alpha} \hat\eps,\; \eps_L(u),\; \eps_B(u)\text{ and }
\norm{L^T\muw -B^T\sigmav }_1\text{ are small}$$
and 
$$
\norm{\muw }_1,\; \norm{\sigmav }_1 \text{ and }
\norm{\Psi(u)}_\infty\text{ are bounded}.
$$
Moreover, if we choose the weight vectors $\muw$ and $\sigmav$
such that $L^T\muw -B^T\sigmav =0$, then the size of
$\norm{\Psi(u)}_\infty$ is irrelevant. We may
also require that $\hat\eps=|\delta(\fhat,\ghat)|=0$,
which is just another linear equation for $\muw$ and $\sigmav$
to satisfy. 

Assuming that $\eps_L(u)$ and $\eps_B(u)$
are small thanks to (a) the smoothness of $u$ enforced for smooth
$f,g,\fhat,\ghat$ by
the appropriate choice of the operators $\Lcal,\Bcal$, and (b)
the choice of the numerical differentiation scheme, we may use any
remaining degrees of freedom in  $\muw$ and $\sigmav$
in order to minimize the stability constants
$\norm{\muw }_1$ and $\norm{\sigmav }_1$. 
We may minimize the joint 1-norm $\norm{\muw }_1 + \norm{\sigmav }_1$,
or some other norm $\norm{(\muw ,\sigmav )}_{\sharp}$ of the
combined weight vector $(\muw ,\sigmav )$
that may be preferable for computational or other reasons.
This considerations lead us to the following framework
for generating accurate quadrature formulas. %

\begin{algotheorem} \label{algo:abstract-algorithm}
Given $\Omega$, $\mu$, $\sigma$, $Y$, $Z$ as introduced in
Section~\ref{ssec:error-analysis-in-the-abstract-setting}, 
compute quadrature weights $\muw ^* \in \R^{N_Y}$
and $\sigmav ^* \in \R^{N_Z}$ as follows:

\begin{enumerate}

\item Choose auxiliary functions $\fhat \in L^1_Y(\Omega)$
and $\ghat \in L^1_Z(\partial\Omega)$ such that
\[
\int_\Omega \fhat \dmu
- \int_{\partial\Omega} \ghat \dsigma \neq 0. 
\]

\item Choose operators $\Lcal$ and $\Bcal$ such that
\begin{equation*}
\int_\Omega \Lcal u \dmu
= \int_{\partial\Omega} \Bcal u \dsigma
\quad \text{for all $u \in \Dcal$,}
\end{equation*}
where the set $\Dcal$ contains at least one solution 
of the boundary value problem
\eqref{eq:bvp-joint-error-estimate} for each pair of functions $f \in L^1_Y(\Omega)$
and $g \in L^1_Z(\partial\Omega)$
whose integrals we wish to approximate using
formulas $(Y,\muw ^*)$ and $(Z,\sigmav ^*)$.

\item Choose a numerical differentiation scheme $(\Psi,L,B)$
for the linear reconstruction of the pointwise values
$\Lcal u \restrict{Y}$ and $\Bcal u \restrict{Z}$
from the vector $\Psi(u)$ in the form of matrix-vector
products $L\Psi(u)$ and $B\Psi(u)$.

\item Assemble the non-homogeneous linear system
with unknowns $(\muw ,\sigmav )$
\begin{equation} \label{eq:linear-system-fhat-ghat}
L^T \muw  - B^T \sigmav  = 0, \qquad
\sum_{i=1}^{N_Y} \muw_i \fhat(y_i)
- \sum_{i=1}^{N_Z} \sigmav_i \ghat(z_i)
= \int_\Omega \fhat \dmu
- \int_{\partial\Omega} \ghat \dsigma,
\end{equation}
which can be written more compactly as %
\begin{equation} \label{eq:ls}
A x = b, \quad\text{ with}\quad A = \begin{pmatrix}
L^T & -B^T \\ 
\fhat|_Y^T & -\ghat|_Z^T
\end{pmatrix}, \quad
x = \begin{pmatrix}
\muw  \\ 
\sigmav 
\end{pmatrix}, \quad
b = \begin{pmatrix}
0 \\[0.5ex]
\int_\Omega \fhat \dmu - \int_{\partial\Omega} \ghat \dsigma
\end{pmatrix}.
\end{equation}

\item Compute a solution $(\muw ^*,\sigmav ^*)$ of
\eqref{eq:linear-system-fhat-ghat} that minimizes
some norm $\norm{(\muw ,\sigmav )}_{\sharp} \colon
\R^{N_Y+N_Z} \to \R_+$ defined on the space of
all possible weights $(\muw ,\sigmav )$.
\end{enumerate}
\end{algotheorem}

The algorithm terminates successfully as long as the linear system
\eqref{eq:linear-system-fhat-ghat} is consistent,
in which case Corollary~\ref{cor:split-error-estimates}
implies that the quadrature formulas $(Y,\muw ^*)$
and $(Z,\sigmav ^*)$ satisfy the error bounds
\begin{align}
\label{eq:fba}
\absB{\int_\Omega f \dmu - \sum_{i=1}^{N_Y} \muw_i^* f(y_i)}
&\leq \inf_{u \in \Ucal_f} \Bigl\{
	\norm{\muw ^*}_1 \varepsilon_L(u)
	+ \norm{\sigmav ^*}_1 \varepsilon_B(u)
\Bigr\}, \\
\label{eq:gba}
\absB{\int_{\partial\Omega} g \dsigma - \sum_{i=1}^{N_Z} \sigmav_i^* g(z_i)}
&\leq \inf_{u \in \Ucal_g} \Bigl\{
	\norm{\muw ^*}_1 \varepsilon_L(u)
	+ \norm{\sigmav ^*}_1 \varepsilon_B(u)
\Bigr\}.
\end{align}

The consistency of \eqref{eq:linear-system-fhat-ghat} may be
characterized in terms of a certain
discrete incompatibility condition as follows.
Recall that $\fhat$ and $\ghat$ satisfy $I(\fhat,\ghat) \neq 0$
and hence in view of \eqref{eq:compatibility-of-L-and-B} 
are incompatible as the right hand sides of the boundary value
problem for the operators $\Lcal$ and $\Bcal$, that is, the problem
\begin{equation} \label{eq:incompatibility}
\begin{cases}
\Lcal u = \hat{f}, \\
\Bcal u = \hat{g},
\end{cases}
\end{equation}
is not solvable for any $u\in\Dcal$.

\begin{proposition}\label{prop:discrete-compatibility-conditions}
The linear system $Ax=b$ given by \eqref{eq:ls} is
consistent if and only if the following \emph{discrete
incompatibility condition} holds: the linear system
\begin{equation} \label{eq:discrete-compatibility-conditions}
\begin{cases}
L c = \hat{f}|_Y, \\
B c = \hat{g}|_Z,
\end{cases}
\end{equation}
is inconsistent, i.e.\ has no solution $c \in \R^M$.
Moreover, if the linear system \eqref{eq:ls} is
consistent, then its solution $(\muw^*,\sigmav^*)$  
that minimizes the norm $\norm{(\muw ,\sigmav )}_{\sharp}$
satisfies 
\begin{equation}\label{dualstab}
\norm{(\muw^*,\sigmav^*)}_{\sharp} 
=|I(\fhat,\ghat)|\Big(\inf_{c\in\R^M}
\Big\|\begin{pmatrix}	Lc - \fhat|_Y\\ 	
Bc -\ghat|_Z\end{pmatrix}\Big\|_\sharp'\Big)^{-1},
\end{equation}
where $\|\cdot\|_\sharp'$ denotes the norm on $\R^{N_Y+N_Z}$
dual to $\|\cdot\|_\sharp$.
\end{proposition}

\begin{proof}
The characterization essentially follows
from the identity $\Img(A)^\perp = \Ker(A^T)$
and the usual assumptions on $(\fhat,\ghat)$.
We can reason as follows:
\begin{align*}
b \in \Img(A)
&\;\;\Leftrightarrow\;\;
	\begin{pmatrix}
	0 \\ 
	I(\hat{f},\hat{g}) 
	\end{pmatrix} \in \Ker(A^T)^\perp \\
&\;\;\Leftrightarrow\;\;
	\Big(\begin{matrix}
	0 & 
	I(\hat{f},\hat{g})
	\end{matrix}\Big)\!
	\begin{pmatrix}
	c \\
	\lambda
	\end{pmatrix}
	= 0 \quad \text{for all
	$\begin{pmatrix}
	c \\
	\lambda
	\end{pmatrix} \in \Ker(A^T)$} \\
&\;\;\Leftrightarrow\;\;
	\lambda	= 0 \quad \text{for all
	$\begin{pmatrix}
	c \\
	\lambda
	\end{pmatrix} \in \Ker(A^T)$}. %
\end{align*}
The last statement in the chain is equivalent
to saying that $\lambda = 0$ for all $c\in\R^M$ and $\lambda\in\R$
that satisfy
\[
\begin{cases}
L c + \lambda \hat{f}|_Y = 0, \\
-B c - \lambda \hat{g}|_Z = 0,
\end{cases}
\]
which is in turn equivalent to
\eqref{eq:discrete-compatibility-conditions}
being inconsistent.

To show \eqref{dualstab} we assume that \eqref{eq:ls} is
consistent. Then it follows from 
\cite[Lemma~8]{davydov-minimal} that 
\begin{align*}
\inf\big\{\norm{x}_{\sharp}: Ax=b\big\} &=\sup\Big\{b^T\!	
\begin{pmatrix}	c \\	\lambda	\end{pmatrix}: 
c\in\R^M,\;\lambda\in\R\;\text{ with }
\Big\|A^T\!\begin{pmatrix}	c \\ \lambda\end{pmatrix}\Big\|_\sharp'
\le1\Big\},\\
&=\sup\Big\{\lambda I(\fhat,\ghat): c\in\R^M,\;\lambda\in\R\;\text{ with }
\Big\|\begin{pmatrix}	Lc +\lambda \fhat|_Y\\ 	Bc +\lambda\ghat|_Z\end{pmatrix}
\Big\|_\sharp'\le1\Big\}\\
&=|I(\fhat,\ghat)|\,\sup\Big\{\lambda>0 :
\exists\,\tilde c\in\R^M\;\text{ with }
\Big\|\begin{pmatrix}	L\tilde c -\fhat|_Y\\ 	B\tilde c -\ghat|_Z\end{pmatrix}
\Big\|_\sharp'\le\frac1\lambda\Big\},
\end{align*}
and \eqref{dualstab} follows. \qedhere
\end{proof}

Since \eqref{eq:discrete-compatibility-conditions} is a
discretization of \eqref{eq:incompatibility}, we expect that the
discrete incompatibility condition holds whenever the
discretization obtained via the choice of $Y,Z$ and the numerical 
differentiation scheme is sufficiently accurate for the linear
system to inherit this important feature of the continuous
problem.

Even though we cannot guarantee the consistency of 
\eqref{eq:linear-system-fhat-ghat} in general, in particular for
arbitrary user-supplied quadrature nodes $Y,Z$, a simple
strategy of choosing the size $M$ of the discretization
vector $\Psi(u)$ such that the linear system
\eqref{eq:linear-system-fhat-ghat}
is underdetermined with significantly fewer rows than columns
makes the system solvable in all of our numerical experiments,
leaving enough room for enforcing reasonable size of
the stability constants $\norm{\muw }_1$ and $\norm{\sigmav }_1$
through minimization, see Section~\ref{sec:numerical-tests}
for more details.

For completeness, we note that
Algorithm~\ref{algo:abstract-algorithm}
is not the only way to derive numerical
integration schemes based on the error analysis
in Proposition~\ref{prop:joint-error-estimate}.
For example, if a sufficiently accurate quadrature
formula $(Z,v)$ for the integral over  $\partial\Omega \neq \emptyset$
is known and fixed in advance, then one could seek weights $w^*$
by minimizing a norm $\norm{\cdot}_{\sharp} \colon \R^{N_Y} \to \R_+$
over the solution space of the non-homogeneous
equation $L^T w = B^T v$. Alternatively, without fixing $v$,
one might relax the constraint $L^T w - B^T v = 0$
by instead minimizing the expression
\[
\lambda_1 \norm{w}_2^2
+ \lambda_2 \norm{v}_2^2
+ \lambda_3 \norm{L^T w - B^T v}_2^2,
\]
where the coefficients
$\lambda_1, \lambda_2, \lambda_3 \in \R_+$
are chosen to balance the terms appropriately.
The accuracy and stability analysis of quadrature
formulas generated by these alternative methods
lies beyond the scope of this paper.
Among the variants explored in our preliminary numerical
experiments, however, Algorithm~\ref{algo:abstract-algorithm}
appeared to achieve the best trade-off
between accuracy, generality, and computational efficiency.

In the following subsections we illustrate and compare
several options for the choices to be made at each step of
Algorithm~\ref{algo:abstract-algorithm}.

\subsection{Auxiliary functions \texorpdfstring{$\hat{f}$}{fhat}
and \texorpdfstring{$\hat{g}$}{ghat}}
\label{ssec:on-the-choice-of-fhat-and-ghat}

The non-homogeneous constraint
\begin{equation} \label{eq:non-homogeneous-constraint}
\sum_{i=1}^{N_Y} \muw_i \fhat(y_i)
- \sum_{i=1}^{N_Z} \sigmav_i \ghat(z_i)
= \int_\Omega \fhat \dmu
- \int_{\partial\Omega} \ghat \dsigma \neq 0
\end{equation}
plays a fundamental role in the definition of linear system
\eqref{eq:linear-system-fhat-ghat}. It
rules out the trivial solution $(\muw ^*,\sigmav ^*) = (0,0)$,
and ensures that the combined quadrature formula is accurate
beyond the pairs of integrands $(f,g)$ with $I(f,g)=0$.
Thus, the method requires
the availability of a single nonvanishing
combined moment $I(\fhat,\ghat)$.

In practice, choices for $\fhat$ and $\ghat$ need not
be complicated: in particular, both 
\begin{equation}\label{eq:fghat01}
(\fhat,\ghat) \equiv (1,0)\quad\text{or}\quad(\fhat,\ghat) \equiv (0,1)
\end{equation}
are usually effective,
as long as at least one of the measures
\[
\abs{\Omega} \deq \mu(\Omega)=\int_\Omega 1 \dmu, \qquad
\abs{\partial\Omega} \deq \sigma(\partial\Omega) =\int_{\partial\Omega} 1 \dsigma
\]
is either known in advance, or can be accurately approximated in
a straightforward way. In particular, this can be easily done
using standard quadrature formulas when the boundary
$\partial\Omega$ is defined by explicit parametric patches over
simple regions in the parameter domain. 

The following approach does not even require knowing the measures
of either $\Omega$ or $\partial\Omega$ and is therefore
entirely moment-free, which is an advantage for instance for
domains defined by implicit surfaces as used in the
level set method, or by trimmed multipatch surfaces of 
Computer Aided Design.

Assume that $\Omega$ is a domain in $\R^d$ and
$\mu$ and $\sigma$ are the standard
Lebesgue and hypersurface measures.
Let $\Phi(x,x_0)$ be the fundamental solution of the
Laplace equation centered at a point $x_0 \in \Omega$:
\[
\Phi(x,x_0)
= \frac{1}{2\pi} \log\bigl(\norm{x-x_0}\bigr)
\quad\text{if $d = 2$,} \qquad
\Phi(x,x_0)
= -\frac{\norm{x-x_0}^{2-d}}{d(d-2)\omega_d}
\quad\text{if $d = 1$ or $d > 2$,}
\]
with $\omega_d$ being the measure of
the $d$-dimensional unit ball. Let $\nu$ be
the outward-pointing unit normal field along $\partial\Omega$,
and let $\partial_\nu$ be the normal derivative operator
on $\partial\Omega$. By Green's third identity,
\[
\int_{\partial\Omega} \partial_\nu \Phi(x,x_0) \dsigmax = 1,
\]
and so the choice 
\begin{equation}\label{eq:fghatfs}
\fhat(x)=0\quad\text{and}\quad\ghat(x)= \partial_\nu
\Phi(x,x_0)=\frac{\nu(x)^T(x-x_0)}{d\omega_d\norm{x-x_0}^d}
\end{equation}
satisfies
\[
\int_\Omega \fhat \dmu
- \int_{\partial\Omega} \ghat \dsigma = -1
\]
independently of the shape of the domain and the position
of $x_0 \in \Omega$.
Observe that the singularity of $\Phi(x,x_0)$ at the interior
point $x_0$ of $\Omega$ plays no role because 
$\ghat(x)= \partial_\nu\Phi(x,x_0)$ is only evaluated at the
boundary $\partial\Omega$. In practice, care must be taken to avoid
points $x_0$ too close to the boundary,
or else the smoothness of the solution $u$ of
\eqref{eq:bvp-joint-error-estimate} may be affected by the
near-singularity of $\ghat$, and as a consequence
the quantities $\varepsilon_L(u)$, $\varepsilon_B(u)$
in the quadrature errors estimates may blow up.

Closed-form expressions for the fundamental solutions
of the Laplace equation are known on some manifolds,
too, see for example the list in \cite{beltran2019discrete}.
This approach can be useful, for example, to compute
a quadrature formula on a subdomain of the $d$-dimensional
sphere with unknown $\abs{\Omega}$ and
$\abs{\partial\Omega}$.

\subsection{Operators \texorpdfstring{$\Lcal$}{Lcal} and \texorpdfstring{$\Bcal$}{Bcal}}
\label{ssec:choice-of-operators-Lcal-and-Bcal}

The error bounds in
Section~\ref{ssec:error-analysis-in-the-abstract-setting}
depend on the accurate reconstruction of pointwise
values of $\Lcal u$ and $\Bcal u$ from a given vector
$\Psi(u)$, with $u$ being a solution to a
boundary value problem of the form
\begin{equation} \label{eq:bvp-generic}
\begin{cases}
\Lcal u = f &\text{in $\Omega$}, \\
\Bcal u = g &\text{on $\partial\Omega$},
\end{cases} \qquad\text{with}\quad
\int_{\Omega} f \dmu = \int_{\partial\Omega} g \dsigma.
\end{equation}
No matter which numerical differentiation method is used
to define matrices $L$ and $B$, the accuracy of the
reconstruction, and thus the quantities $\eps_L(u)$ and
$\eps_B(u)$ in \eqref{eq:joint-error-estimate}, depend crucially
on the smoothness of~$u$.
On the contrary, it is natural in numerical integration to
investigate how the accuracy of the quadrature formulas depends 
on the smoothness of the integrands $f$ and $g$.
Therefore, we are led to consider operators
$\Lcal$ and $\Bcal$ for which the boundary value problem
\eqref{eq:bvp-generic} admits solutions of sufficiently high
smoothness, depending on the regularity of $f$ and $g$.

More rigorously, the regularity of functions
in $\Omega$ and on its boundary will be measured
using Sobolev norms.
Let us fix some notation for Sobolev spaces
that will be used in this and the next subsections.
The Sobolev space $W_p^k(\Omega)$ with integer order $k \geq 0$
and exponent $1 \leq p \leq \infty$ on a domain
$\Omega \subset \R^d$ is equipped with the Sobolev
norm and seminorm
\[
\norm{f}_{W_p^k(\Omega)}
= \sum_{\abs{\alpha} \leq k} \norm{\partial^\alpha f}_{L_p(\Omega)},
\qquad
\seminorm{f}_{W_p^k(\Omega)}
= \sum_{\abs{\alpha} = k} \norm{\partial^\alpha f}_{L_p(\Omega)},
\]
where $\norm{\cdot}_{L_p(\Omega)}$ denotes the usual $L_p$-norm
on $\Omega$, and $\partial^\alpha$ is the generalized partial
derivative defined by the multi-index $\alpha\in\Z^d_+$.
To properly state some results involving trace operators,
Sobolev spaces of integer order are not sufficient,
therefore Sobolev-Slobodeckij spaces $W_p^s(\Omega)$ of intermediate regularity
are considered. They are defined for fractional order
$s \in \R_+ \setminus \Z$ and exponent $1 \leq p < \infty$. 
These spaces are also generalized to the case when $\Omega$
is a domain in a $C^{\infty}$
Riemannian manifold $\Mcal$. For $p = 2$, the spaces
$W_2^s(\Omega)$ are Hilbert spaces and are denoted $H^s(\Omega)$.
Sobolev-Slobodeckij spaces up to order $k+1$ can also be
defined on the boundary of a domain, as long as the boundary
is of class $C^{k,1}$, which means $k$ times continuously
differentiable with 
Lipschitz continuous derivatives of order $k$. 
For more information, see e.g.~\cite{grisvard-elliptic}.

\subsubsection{The elliptic approach}
\label{sssec:neumann}

Relying on classical results in the theory of elliptic partial
differential equations, see e.g.\ \cite[p.~409]{taylor-PDE},
we consider the choice of the Laplace-Beltrami
operator as $\Lcal$ %
and the normal derivative trace operator as
$\Bcal$. %

Let $\Omega$ be a domain in a $C^\infty$ Riemannian
manifold $\Mcal$, with compact closure $\Omegabar$ and (possibly empty)
smooth boundary $\partial\Omega$, let $\mu$ and
$\sigma$ be the measures on $\Omega$ and $\partial\Omega$
induced by the Riemannian metric of $\Mcal$,
and let $\nu$ be the outward-pointing unit
normal field along the boundary.
Then it follows from the divergence theorem that 
\begin{equation} \label{eq:divlapbe}
\int_\Omega \Delta u \dmu = \int_{\partial\Omega}
\partial_\nu u \dsigma\quad\text{for all }u \in H^2(\Omega),
\end{equation}
where $\Delta$ is the Laplace-Beltrami operator
on the manifold, and $\partial_\nu$ is the normal derivative
trace operator. In other words, identity
\eqref{eq:compatibility-of-L-and-B} is satisfied for
\begin{equation} \label{eq:choice-Lcal-Bcal-laplacian} 
\Lcal =\Delta, 
\quad \Bcal = \partial_\nu, 
\end{equation}
with $\Dcal = H^2(\Omega)$. 

\begin{proposition}[Well-posedness and regularity of Neumann problem
on smooth manifolds]
\label{prop:well-posedness-neumann}
Under the above assumptions, for all $f \in H^s(\Omega)$ and 
$g \in H^{s+1/2}(\partial\Omega)$, the Neumann boundary value problem
\begin{equation} \label{eq:bvp-laplacian}
\begin{cases}
\Delta u = f &\text{in $\Omega$} \\
\partial_\nu u = g &\text{on $\partial\Omega$}
\end{cases}
\quad \text{with constraint } \; \int_\Omega u \dmu = 0
\end{equation}
has a unique solution $u \in H^{s+2}(\Omega)$
that depends continuously on $f$ and $g$, if and only if
\begin{equation} \label{eq:compatibility-conditions}
\int_{\Omega} f \dmu = \int_{\partial\Omega} g \dsigma,
\end{equation}
a constraint known as \emph{compatibility condition}
for the Neumann boundary value problem.
The continuous dependence of the solution on $f$ and $g$
amounts to the existence of a constant $C > 0$
independent of $f$ and $g$ such that
\[
\norm{u}_{H^{s+2}(\Omega)}
\leq C \Bigl( \norm{f}_{H^{s}(\Omega)}
+ \norm{g}_{H^{s+1/2}(\partial\Omega)} \Bigr).
\]
\end{proposition}
As a direct consequence of this result, the choice
\eqref{eq:choice-Lcal-Bcal-laplacian}
is fully satisfactory when $\Omega$ is either a closed manifold with 
$\partial\Omega=\emptyset$, or a bounded domain $\Omega$ in a smooth
Riemannian manifold $\Mcal$, with smooth boundary
$\partial\Omega$. Indeed, in these cases a solution $u$
to the boundary value problem \eqref{eq:bvp-joint-error-estimate}
is guaranteed to exist and even has a higher regularity
than the corresponding data $f-\alpha\fhat,g-\alpha\ghat$ in the right-hand side.

However, when $\partial\Omega$ is not smooth, as is
often the case in practical applications, 
comparable elliptic regularity results do not hold anymore,
and so the solution $u$ to the Neumann boundary value
problem~\eqref{eq:bvp-joint-error-estimate} may not be
regular enough for the numerical differentiation scheme
\[
 \Delta u \restrict{Y}\approx L\Psi(u) , \quad
\partial_\nu u \restrict{Z} \approx B\Psi(u)
\]
to be sufficiently accurate even for highly smooth $f$ and $g$,
which indicates that the choice
\eqref{eq:choice-Lcal-Bcal-laplacian} may be problematic
because the numerical
differentiation errors $\varepsilon_L(u)$ and $\varepsilon_B(u)$
featuring in the estimate \eqref{eq:joint-error-estimate}
are presumably affected by the loss of smoothness of $u$.

Indeed, let $\Omega$ be a bounded Lipschitz domain in $\R^d$,
and $\mu,\sigma$ the standard Lebesgue and
hypersurface measures. Then $\Delta$ is
the usual Laplace operator. The Lipschitz regularity of $\Omega$
is sufficient to define an outward-pointing
unit normal field $\nu$ almost everywhere on $\partial\Omega$,
and \eqref{eq:divlapbe} still holds.
Moreover, for all $f \in L^2(\Omega)$ and
all $g \in L^2(\partial\Omega)$, the Neumann boundary problem
\eqref{eq:bvp-laplacian} has a unique
solution $u \in H^{3/2}(\Omega)$ if and only if
the compatibility condition \eqref{eq:compatibility-conditions}
is satisfied, as implied by the results in \cite{jerison1981neumann}.
However, standard examples in \cite{grisvard-elliptic}
for the corresponding Dirichlet problem on piecewise
smooth domains $\Omega\subset\R^2$ with reentrant corners can be
modified to apply to the Neumann problem, and generate for any
$s>3/2$ a solution $u \notin H^{s}(\Omega)$, even when $f$ is
infinitely differentiable and $g=0$.
This means that the regularity of $u$ is not only limited
by the regularity of $f$ and $g$, but also by the regularity
of $\partial\Omega$.

The following example, adapted from
\cite{grisvard-elliptic}, shows that this may happen
even on elementary geometries, such as a square.

\begin{example} \label{ex:elliptic-regularity-fails}
Let $\Omega = (0,1)^2$, and let $\Gamma_1,\ldots,\Gamma_4$ be the
left, bottom, right, and top sides of $\partial\Omega$,
respectively. For all $i = 1,\dots,4$ and $s \geq 0$,
let $\gamma_i$ be the trace operator from $H^{s+1/2}(\Omega)$
to $H^s(\Gamma_i)$.
Functions in $H^s(\Gamma_i)$ are parametrized
by the restrictions of Cartesian coordinates $(x,y)$
to~$\partial\Omega$, that is, $x \in [0,1]$ for $i = 2,4$
and $y \in [0,1]$ for $i = 1,3$. The following function,
defined using polar coordinates $(r,\theta)$
centered at the origin, is harmonic in $\Omega$:
\[
v(r,\theta) = r^2 \bigl(
\log(r) \cos(2\theta) - \theta \sin(2\theta)
\bigr).
\]
Let $\varphi \colon [0,1] \to \R$ be a smooth function
such that $\varphi(r) \equiv 1$ for all $r \in [0,\frac{1}{3}]$,
and $\varphi(r) \equiv 0$ for all $r \in [\frac{2}{3},1]$.
Then, the function $u(r,\theta) \deq \varphi(r) \, v(r,\theta)$
is the unique solution (up to a constant) of
\[
\begin{cases}
\Delta u = \Delta (\varphi v)
	&\text{in $\Omega$} \\
\partial_\nu u = 0
	&\text{on $\partial\Omega \setminus \Gamma_1$} \\
\partial_\nu u = \gamma_1(-\partial_x u) = \pi y \varphi(y)
	&\text{on $\Gamma_1$}.
\end{cases}
\]
The right hand side functions $f,g$ of this Neumann problem 
are smooth in the sense
that they are restrictions to $\Omega$ and $\partial\Omega$
of suitable functions that belong to $C^\infty(\R^2)$.
Nevertheless, it is clear that $u \notin H^3(\Omega)$,
because $u$ and $v$ coincide for $r < 1/3$ and
\[
\frac{\partial^3 v(x,y)}{\partial x^3}
= \frac{2x}{x^2+y^2}
= \frac{2\cos(\theta)}{r}
\notin L^2(\Omega).
\]
\end{example}

In principle, the estimate \eqref{eq:joint-error-estimate} that
relates the error of our quadrature to the numerical
differentiation errors $\varepsilon_L(u)$ and $\varepsilon_B(u)$
is just an upper bound, with no guarantee of being tight.
Moreover, on Lipschitz domains in $\R^2$ with piecewise
smooth boundary, it can be proved that
$u$ still has regularity $H^{s+2}$ away from corners,
so one could hope that a few inaccurate approximations
of $\Delta u(y_i)$ and $\partial_\nu u(z_i)$
around corners would not affect the overall
accuracy of the quadrature formulas. However, our numerical
experiments in Section~\ref{sec:numerical-tests}
demonstrate that a catastrophic loss of accuracy
in the quadrature formulas may occur on a Lipschitz domain
with piecewise smooth boundary, whereas we do not register 
this behavior for the alternative approach suggested
below based on the divergence operator.

\subsubsection{The divergence approach}
\label{sssec:div}

The most basic instance of the identity
\eqref{eq:compatibility-of-L-and-B} is clearly the divergence
theorem, which suggests the choice 
\begin{equation} \label{eq:choice-div}
\Lcal = \diver,
\quad \Bcal = \gamma_\nu,
\end{equation}
where $\diver$ is the divergence operator on the manifold
and $\gamma_\nu$ is the normal trace operator.
This leads us to the boundary value problem
\begin{equation} \label{eq:bvp-divergence}
\begin{cases}
\diver F  = f &\text{in $\Omega$}, \\
\gamma_\nu(F) = g &\text{on $\partial\Omega$}
\end{cases}\qquad\text{with}\quad
\int_{\Omega} f \dmu = \int_{\partial\Omega} g \dsigma,
\end{equation}
where the measures $\mu$ and $\sigma$ are determined by the
Riemannian metric, the compatibility condition on the right must
be satisfied due to the divergence theorem, and the solution $F$ 
is sought in an  appropriate space $\Dcal$ of vector fields on
$\Omega$. In the case of the manifold $\Mcal=\R^d$, vector fields
may be identified with vector valued functions
$F:\overline{\Omega}\to\R^d$, and the normal trace is just the
projection $\gamma_\nu(F)=\nu^T\!F$ of $F$ to the
outward-pointing unit normal $\nu$ on the boundary $\partial\Omega$
for sufficiently smooth $F$.

Any solution $u$ of the Neumann problem \eqref{eq:bvp-laplacian}
gives rise to a solution $F=\nabla u$ of
\eqref{eq:bvp-divergence}. Hence, no smooth solutions are ever
lost by considering operators $\diver$ and $\gamma_\nu$
instead of $\Delta$ and $\partial_\nu$. However,
unlike \eqref{eq:bvp-laplacian}, the boundary value problem
\eqref{eq:bvp-divergence} is underdetermined,
with highly non-unique solutions $F$.
Recall that the error bound \eqref{eq:joint-error-estimate}
involves the infimum over all solutions of
\eqref{eq:bvp-joint-error-estimate}, so the non-uniqueness is an
advantage as \eqref{eq:joint-error-estimate} may rely on the
solution $F$ with the smallest numerical differentiation errors
$\eps_L(F)$ and $\eps_B(F)$.
In particular, one may hope that on a Lipschitz domain
with piecewise smooth boundary there still exists a solution $F$
of \eqref{eq:bvp-divergence} that inherits a high order of
smoothness from smooth functions $f,g$ in the right hand side,
even when the same is not true for \eqref{eq:bvp-laplacian}.

Although we have not found results of this type in the
literature, considerable attention has been paid to the corresponding
Dirichlet problem, where the full trace of $F$ on the boundary is
prescribed instead of the normal trace $\nu^T\!F$ in
\eqref{eq:bvp-divergence}. The Dirichlet problem, however, requires 
additional compatibility conditions between $f$  and the boundary
data. %
In particular, if
$F$ vanishes on the boundary of a polygon $\Omega\subset\R^2$,
then its divergence $f$ must vanish at the vertices of $\Omega$,
which is a quite unnatural condition to impose on integrands $f$.
Nevertheless, several results on the Dirichlet problem for domains in
$\Mcal=\R^d$ imply that a solution
$F$ of \eqref{eq:bvp-divergence} belongs to $H^{s+1}(\Omega)^d$
when $f\in H^s(\Omega)$ and $g$ satisfies
certain smoothness and compatibility assumptions. In particular, this follows from
\cite[Lemma II.2.3.1]{sohr2001navier} for Lipschitz domains in
$\R^d$, with $s\in\N$, when $g=0$ and $f$ vanishes on the
boundary   $\partial\Omega$ to order $s-1$. Furthermore,
\cite{arnold1988regular} gives the same regularity result for
polygons in $\R^2$ when $g=0$ and $f$ vanishes at all vertices of
the polygon, and also for non-homogeneous $g$, with smoothness
and compatibility conditions that are difficult to recast to our
setting, in which $g$ is the normal projection of the Dirichlet
boundary data studied in \cite{arnold1988regular}. These results
are generalized in \cite{mitrea2008poisson} to 
arbitrary bounded Lipschitz domains in smooth
manifolds, see in particular \cite[Corollary 1.4]{mitrea2008poisson}
for the case $\Mcal=\R^d$.

Our numerical results in Section~\ref{sec:numerical-tests}
for the divergence-based quadrature do not indicate any
loss of convergence order  on non-smooth domains, which
suggests that the smoothness of $F$ is not lost.
Moreover, the following elementary construction produces a smooth
solution $F$ of \eqref{eq:bvp-divergence} on $\Omega=(0,1)^2$,
whenever $f$ is smooth in $\Omega$ and $g$ is smooth on each side
of the square, in contrast to
Example~\ref{ex:elliptic-regularity-fails} for the
elliptic Neumann problem.

\subsubsection*{Smooth solution of divergence boundary
problem on the square} 

As before, let $\Gamma_1,\ldots,\Gamma_4$
denote the left, bottom, right, and top sides of $\partial\Omega$,
and let $g_i$ for $i=1,\ldots,4$ be the univariate
functions representing $g|_{\Gamma_i}$
in the natural Cartesian parametrizations. 
We set $\tilde g_i \deq g_i-\alpha_i$,
with $\alpha_i=\int_0^1g_i(t)\,dt$, and
consider the vector fields
$$
G_1(x,y)=
\begin{pmatrix}(x-1)\tilde g_1(y)\\[6pt] -\int_0^y\tilde g_1(t)\,dt
\end{pmatrix},\quad
G_2(x,y)=
\begin{pmatrix} -\int_0^x\tilde g_2(t)\,dt\\[6pt](y-1)\tilde g_2(x)
\end{pmatrix},$$
$$
G_3(x,y)=
\begin{pmatrix}x\tilde g_3(y)\\[6pt] -\int_0^y\tilde g_3(t)\,dt
\end{pmatrix},\quad
G_4(x,y)=
\begin{pmatrix} -\int_0^x\tilde g_4(t)\,dt\\[6pt]y\tilde g_4(x)
\end{pmatrix}
$$
that satisfy
$$
\diver G_i=0,\qquad \nu^T\!G_i|_{\partial\Omega}=
\begin{cases}\tilde g_i & \text{in } \Gamma_i,\\
0 & \text{otherwise}. 
\end{cases}$$
By construction, it follows that
$$
\tilde G(x,y) \deq G_1(x,y)+G_2(x,y)+G_3(x,y)+G_4(x,y)+
\begin{pmatrix}\alpha_1(x-1)+\alpha_3x\\[6pt]
\alpha_2(y-1)+\alpha_4y
\end{pmatrix}$$
is a solution of the problem
$$
\diver \tilde G=\alpha_1+\alpha_2+\alpha_3+\alpha_4,\quad
\nu^T\!\tilde G|_{\partial\Omega}=g.$$
Moreover, let 
$\alpha=\int_\Omega f(x,y)\,dx\,dy$,
 $h(y)=\int_0^1f(\xi,y)\,d\xi-\alpha$ and $\tilde f(x,y)=f(x,y)-\alpha-h(y)$. 
Then the vector field
$$
\tilde F(x,y)=\begin{pmatrix}\int_0^x\tilde f(\xi,y)\,d\xi\\[6pt] 
\int_0^yh(\zeta)\,d\zeta
\end{pmatrix}$$
satisfies
$$
\diver \tilde F=f-\alpha,\qquad \nu^T\!\tilde
F|_{\partial\Omega}=0.$$
Since $\alpha=\alpha_1+\alpha_2+\alpha_3+\alpha_4$ by the
compatibility condition, we conclude that $F=\tilde F+\tilde G$
solves \eqref{eq:bvp-divergence}. It is easy to see that for any
$r\ge0$ we have $F\in C^r(\overline\Omega)^2$ as soon as $f\in
C^r(\overline\Omega)$ and $g_i\in C^r[0,1]$
for all $i=1,\ldots,4$, which even allows $g$
to be discontinuous at the corners of $\partial\Omega$.

\subsection{Numerical differentiation schemes}
\label{ssec:on-the-choice-of-L-and-B}
The choice of the numerical differentiation scheme $(\Psi,L,B)$
is crucial for the performance of the method, because
the factors $\eps_L(u)$ and  $\eps_B(u)$ in the error bounds
\eqref{eq:fba} and \eqref{eq:gba}, as well as the solvability of
the linear system \eqref{eq:ls} with reasonable stability
constants $\norm{\muw }_1$ and $\norm{\sigmav }_1$
depend on this choice. In particular, the discrete 
incompatibility condition \eqref{eq:discrete-compatibility-conditions}
must be satisfied. Moreover, for an efficient numerical
implementation it is important that the differentiation matrices 
\[
L = \bigl( \ell_{ij} \bigr)_{i=1,j=1}^{N_Y,M}
\quad\text{and}\quad
B = \bigl( b_{ij} \bigr)_{i=1,j=1}^{N_Z,M}
\]
are sparse.

We describe two approaches that work well in our numerical tests.
The first is based on meshless numerical differentiation formulas,
as used in the generalized finite difference methods,
and the second employs tensor-product splines.

In both cases we assume that $\Lcal$ and $\Bcal$ are  linear
differential operators of integer order, well-defined for
functions in a space  $\Dcal$  of functions on the closure of a
domain $D$  such that $\Omega\subset D$. 
The domains $\Omega$ and $D$
need not have the same dimension; indeed, when $\Mcal$
is embedded into an ambient space $\Acal$ such as
a surface $\Mcal$ in $\Acal=\R^3$, one may
prefer to choose a domain $D$ in $\Acal$ rather than in $\Mcal$,
leading to an \emph{immersed} approach. 

For the sake of notational simplicity,
but without loss of generality, we assume in this section that
the functions in $\Dcal$ are scalar-valued.

\subsubsection{Meshless finite difference formulas}
\label{sssec:mfd}
Assuming that $\Dcal\subset C(\overline{D})$, we choose a finite set 
$X = \{x_1,\ldots,x_M\} \subset D$, define the
discretization operator $\Psi \colon \Dcal \to \R^M$
by pointwise evaluation over $X$,
\begin{equation} \label{eq:Psi-u}
\Psi(u) = \bigl( u(x_1), \dots, u(x_M) \bigr)^T,
\end{equation}
and approximate $\Lcal u (y_i)$ and $\Bcal u (z_i)$ at the
quadrature nodes $y_i \in Y$ and $z_i \in Z$
by appropriate numerical differentiation formulas
\begin{equation} \label{eq:numerical-differentiation-formulas}
\Lcal u (y_i) \approx
\sum_{j=1}^M \ell_{i j} \, u(x_j), \qquad
\Bcal u (z_i) \approx
\sum_{j=1}^M b_{i j} \, u(x_j).
\end{equation}
Since integer order differentiation is a local operation, sufficiently accurate
formulas can be found with the sums in
\eqref{eq:numerical-differentiation-formulas}
restricted to small sets of indices
$S_{L,i}$ and $S_{B,i}$ corresponding to subsets of nodes
in $X$ close to $y_i$ and $z_i$, known as
\emph{sets of influence}. This in turn implies that
$\ell_{ij}=0$ for $j\notin S_{L,i}$ and $b_{ij}=0$ for 
$j\notin S_{B,i}$, and so the matrices $L$ and $B$ are sparse.

When $X$ is a gridded set and $Y\cup Z\subset X$, formulas of
this type can be obtained by  classical finite differences as
those used in the finite difference method for partial
differential equations. For irregular sets $X,Y,Z$, as typically
needed on more complicated domains $\Omega$, numerical
differentiation formulas are in the core of the meshless
finite difference methods, such as RBF-FD or GFDM, see for example
\cite{FFprimer15,bayona2017role,JTAMB19,suchde2019meshfree,D23}  
and references therein. Therefore, we generally refer to
\eqref{eq:numerical-differentiation-formulas}
as \emph{meshless finite difference formulas}.

A basic approach to obtaining accurate numerical differentiation
weights $\ell_{ij}$, $j\in S_{L,i}$, and $b_{ij}$,  $j\in
S_{B,i}$, is to impose exactness of formulas 
\eqref{eq:numerical-differentiation-formulas} over a local basis of
functions that can provide a good approximation of
$u\in\Dcal$ in a neighborhood of $y_i$ and $z_i$, respectively. 
In the case when $\Omega$ is a domain in $\R^d$, a natural local
approximation tool is given by the spaces $\Pi^d_q$ of multivariate
polynomials of total degree $<q$. The sets of influence
are usually chosen somewhat larger than the minimum needed to admit a 
numerical differentiation formula exact for all $u\in \Pi^d_q$,
and the extra degrees of freedom are used either to minimize a
(semi-)norm of the weight vectors $(\ell_{ij})_{j\in S_{L,i}}$
and $(b_{ij})_{j\in S_{B,i}}$, or to enhance the local
approximation space by adding radial basis functions. We refer to
\cite{davydov-minimal,davydov-optimal,D21} for the theory and
error bounds for these methods. Polynomial type methods are also
available on certain manifolds, such as the $d$-dimensional
sphere, where spherical harmonics may be employed. Otherwise,
numerical differentiation formulas on arbitrary manifolds may be
generated with the help of positive definite kernels. Suitable local error
bounds for the kernel-based numerical differentiation may be found 
in \cite{davydov-kernel} for domains in $\R^d$ and in
\cite[Section 4.4]{D23} for reproducing kernels of Sobolev
spaces on manifolds.

We now discuss in more detail the numerical differentiation formulas
employed in our numerical tests. They are generated by the
\emph{polyharmonic radial basis kernel} $K(x,y)=\|x-y\|^{2m-1}$,
$x,y\in\R^d$, with a polynomial term in $\Pi^d_q$, $q\ge m$, which 
is a standard approach in RBF-FD since \cite{bayona2017role}. We
use $m=q$. Thus, the weights $\ell_{ij}$, $j\in S_{L,i}$, are
obtained by requiring the exactness  of the first formula in
\eqref{eq:numerical-differentiation-formulas} for all functions
of the form
$$
u(x)=\sum_{j\in S_{L,i}}c_j\|x-x_j\|^{2q-1}+\tilde p(x),\quad
c_j\in\R,\;\tilde p\in \Pi^d_q,\quad\text{with }
\sum_{j\in S_{L,i}}c_jp(x_j)=0\text{ for all }p\in \Pi^d_q,$$
and similarly for the weights $b_{ij}$, $j\in S_{B,i}$. In most
cases the square linear system  that arises from
these conditions is regular. However, as shown in \cite{D21},
numerical differentiation weights are uniquely determined and can
be computed by a null space method also for certain `deficient'
sets of influence, such as the five point stencil for the
Laplacian on gridded nodes.

Let $\Omega$ be a domain in $\R^d$, and $\Omega\subset D\subset\R^d$. 
Denote by $k_L$ and $k_B$
the orders of the operators $\Lcal$ and $\Bcal$, respectively.
Suppose we use polyharmonic numerical differentiation formulas
with $q=q_L>k_L$ for the operator $\Lcal$, and $q=q_B>k_B$ for $\Bcal$.
Under appropriate assumptions on the sets of influence
$X_{L,i} \deq \{x_j:j\in S_{L,i}\}$, and
$X_{B,i} \deq \{x_j:j\in S_{B,i}\}$, such
as quasi-uniformity and boundedness of the
polynomial norming constants, the errors \eqref{eq:epsL} and
\eqref{eq:epsB} of the polyharmonic numerical differentiation
 can be estimated as 
\begin{align*}
\eps_L(u)=\norm{\Lcal u \restrict{Y} - L\Psi(u)}_\infty
=\max_{1\le i\le N_Y}\Big|\Lcal u (y_i) -
\sum_{j\in S_{L,i}} \ell_{i j} \, u(x_j)\Big|\le 
C_{L}h_{L}^{q_L-k_L}|u|_{W^{q_L}_\infty(D)},\\
\eps_B(u)=\norm{\Bcal u \restrict{Z} - B\Psi(u)}_\infty
=\max_{1\le i\le N_Z}\Big|\Bcal u (z_i) -
\sum_{j\in S_{B,i}} b_{i j} \, u(x_j)\Big|\le 
C_{B}h_{B}^{q_B-k_B}|u|_{W^{q_B}_\infty(D)},
\end{align*}
where $h_L$ denotes the maximum diameter of the sets
$\{y_i\}\cup X_{L,i}$ for $i=1,\ldots,N_Y$,
$h_B$ is the maximum diameter of the sets
$\{z_i\}\cup X_{B,i}$ for $i=1,\ldots,N_Z$,
and $C_L,C_B$ are some positive constants
independent of $h_L$, $h_B$, and $u$.
Note that a simple method for the selection of the sets of
influence that ensures in practice that the errors behave in
accordance with these estimates is to compose $X_{L,i}$ of
$2\dim\Pi^d_{q_L}$ closest nodes of $y_i$ in $X$, and $X_{B,i}$
of $2\dim\Pi^d_{q_B}$ closest nodes of $z_i$, as recommended in
\cite{bayona2017role} for RBF-FD.
Further methods that try to optimize the selection can be
found in \cite{DOT23} and references therein.
In any case, the number of nodes in $X_{L,i}$ and $X_{B,i}$
remains bounded if $q_L$ and $q_B$ are fixed, which in turn
implies that the matrices $L$ and $B$ are sparse and that
the diameters $h_{L}$ and $h_{B}$ shrink as the density
of the nodes increases.

Assuming that $k_B=k_L-1$, as in both the elliptic and the divergence
settings of Section~\ref{ssec:choice-of-operators-Lcal-and-Bcal},
we get the same approximation order for  $\eps_L(u)$ and
$\eps_B(u)$ by choosing $q_B=q_L-1$. 
Hence, by Corollary~\ref{cor:split-error-estimates} 
we obtain the following estimates for any
quadrature formulas satisfying
\eqref{eq:linear-system-fhat-ghat}, 
\begin{align}
\label{eq:fbamfd}
\absB{\int_\Omega f \dmu - \sum_{i=1}^{N_Y} \muw_i f(y_i)}
&\leq  C\,
	\big(\norm{\muw}_1+ \norm{\sigmav}_1 \big)h^{q-k}\inf_{u \in \Ucal_f}
	\max\big\{|u|_{W^{q-1}_\infty(D)},|u|_{W^{q}_\infty(D)}\big\}, \\
\label{eq:gbamfd}
\absB{\int_{\partial\Omega} g \dsigma - \sum_{i=1}^{N_Z} \sigmav_i g(z_i)}
&\leq  C\,
	\big(\norm{\muw}_1+ \norm{\sigmav}_1 \big)h^{q-k}\inf_{u \in \Ucal_g}
	\max\big\{|u|_{W^{q-1}_\infty(D)},|u|_{W^{q}_\infty(D)}\big\},
\end{align}
where 
$$
C=\max\{C_L,C_B\},\quad h=\max\{h_L,h_B\},\quad q=q_L=q_B+1,\quad
k=k_L=k_B+1.$$

These estimates show that, as soon as the stability constants
$\norm{\muw ^*}_1$ and $\norm{\sigmav ^*}_1$ for the weights
computed by Algorithm~\ref{algo:abstract-algorithm} remain bounded
for $h \to 0$ and the sets $\Ucal_f$ and $\Ucal_g$
contain functions $u\in W^q_\infty(D)$,
the errors of the two quadrature formulas behave as
$\Ocal(h^{q-k})$.
In order to relate the convergence order to the smoothness of the
integrands $f$ and $g$, we need regularity results for the
boundary value problems \eqref{eq:bvp-split-error-estimates},
such as those discussed in
Section~\ref{ssec:choice-of-operators-Lcal-and-Bcal}. 

For example, in the setting of a smooth domain
$\Omega\subset\R^d$ and operators $\Lcal=\Delta$,
$\Bcal=\partial_\nu$ (with $k=2$), assuming that $\fhat$
and $\ghat$ are infinitely differentiable, we may use
Proposition~\ref{prop:well-posedness-neumann} to infer that 
there exists $u_f\in U_f$ that belongs to $H^{s+k}(\Omega)$
as soon as $f\in H^{s}(\Omega)$.
By the Sobolev embedding theorem,  $H^{s+k}(\Omega)\subset
W^q_\infty(\Omega)$ if $s+k>q+d/2$, and by the Stein extension
theorem $u_f$ may be extended to a function in $W^q_\infty(D)$.
Therefore, the estimate 
\begin{equation}\label{eq:fb}
\absB{\int_\Omega f \dmu - \sum_{i=1}^{N_Y} \muw_i^* f(y_i)}
\le C_1\big(\norm{\muw ^*}_1+ \norm{\sigmav ^*}_1 \big)h^{q-k}
	\|f\|_{H^{s}(\Omega)}
\end{equation}
holds for all functions $f\in H^{s}(\Omega)$ with $s>q-k+d/2$,
where $C_1$ is independent of $f$.
Similarly, we derive from
Proposition~\ref{prop:well-posedness-neumann}  that  there exists
$u_g\in U_g$ that belongs to $H^{s+k-1/2}(\Omega)$ as soon as $g\in
H^{s}(\partial\Omega)$. By the same arguments with the Sobolev
and Stein theorems, $u_g$ may be extended to a  function in
$W^q_\infty(D)$ if $s+k-1/2>q+d/2$, and we obtain the estimate
\begin{equation}\label{eq:gb}
\absB{\int_{\partial\Omega} g \dsigma - \sum_{i=1}^{N_Z} \sigmav_i^* g(z_i)}
\le C_2\big(\norm{\muw ^*}_1+ \norm{\sigmav ^*}_1 \big)h^{q-k}
	\|g\|_{H^{s}(\partial\Omega)}
\end{equation}
for all functions $g\in H^{s}(\partial\Omega)$ with $s>q-k+(d+1)/2$,
where $C_2$ is independent of $g$.
Note that both estimates \eqref{eq:fb} and \eqref{eq:gb} are
supposed to be suboptimal as they possibly
require a higher smoothness of $f$ and $g$ than strictly needed
for the $\mathcal{O}(h^{q-k})$ convergence order of the
quadrature, since the
natural assumption for this should be $f\in H^{q-k}(\Omega)$ and $g\in
H^{q-k}(\partial\Omega)$, respectively, see 
\cite[Chapter~4]{sobolev1997theory}. 

The estimates \eqref{eq:fb} and \eqref{eq:gb} with $k=1$ are obtained 
in the same way for the operators $\Lcal=\diver$, $\Bcal=\gamma_\nu$ if 
$\Omega\subset\R^d$ is smooth, because the regularity of the solution of 
\eqref{eq:bvp-split-error-estimates} follows from 
Proposition~\ref{prop:well-posedness-neumann} by making
use of the gradient of the solution to the corresponding
Neumann problem \eqref{eq:bvp-laplacian}.

\subsubsection*{Choosing discretization nodes}
In practical applications, quadrature formulas often
need to be defined on fixed, user-supplied nodes,
so in Algorithm~\ref{algo:abstract-algorithm} we do not assume
to be in control of the placement of quadrature
nodes $Y$ and $Z$.
The set $X\subset \overline{D}$ of discretization nodes
that determines the operator $\Psi$, however, can be
chosen freely, and so it must be computed as part of the algorithm.

First of all, we have to choose $D$, which may coincide with
$\Omega$, but it may also be a larger domain in $\Mcal$ or in
the ambient space $\Acal$, for example a bounding box around
$\Omega$. While exploring these options numerically, we have
seen that using $D=\Omega$ and irregular $X$ was consistently
a better choice than a bounding box $D$ with a Cartesian grid $X$,
see the numerical results in Section~\ref{sssec:choice-of-nodes}.

Even though a discretization set $X$ fitting $\Omegabar$ may be
obtained via mesh generation, for example by taking $X$
as the vertices of a triangulation of $\Omega$,
we only need meshless nodes for our purposes.
The nodes need not be connected into grids or networks,
which simplifies node generation algorithms,
especially for complicated 3D domains or surfaces.
For a survey on the topic of meshless
node generation, see \cite{suchde-point}.

The generation of $X$ should take into account multiple
criteria that in part contradict one another.
On the one hand, for any fixed quadrature nodes $Y$ and $Z$,
we want the set $X$ to be as large as possible, in order
to improve the errors $\eps_L(u)$ and $\eps_B(u)$
of the numerical differentiation thanks to
a higher density of $X$ in $\Omegabar$.
On the other hand, a smaller set $X$ produces
more stable quadrature formulas. Indeed, let us consider
two nested sets of nodes $X \subset \tilde X$, 
and corresponding matrices $L,B,\tilde L,\tilde B$ so that
$L,B$ are submatrices of $\tilde L,\tilde B$. Then,
the solution $(\tilde\muw^*,\tilde\sigmav^*)$ to
\eqref{eq:linear-system-fhat-ghat} with $\tilde L,\tilde B$
that minimizes the 1-norm $\norm{\muw}_1 + \norm{\sigmav}_1$
satisfies not just
$\tilde L^T \tilde\muw^* - \tilde B^T \tilde\sigmav^* = 0$,
but also $L^T \tilde\muw^* - B^T \tilde\sigmav^* = 0$,
and so $\norm{\muw^*}_1 + \norm{\sigmav^*}_1
\le\norm{\tilde\muw^*}_1 + \norm{\tilde\sigmav^*}_1$.
Moreover, as soon as $X$ is large enough so that $M+1>N_Y +
N_Z$, the linear system \eqref{eq:linear-system-fhat-ghat}
is overdetermined, i.e.\ it has more equations than unknowns,
and so it may not admit solutions anymore.
In general, $X$ should not be unnecessarily irregular,
although a higher density of the nodes may be
advantageous near the boundary of $\Omega$, especially in the
vicinity of its corners or fine features. In the numerical tests we 
only consider quasi-uniform nodes.
For more details on the distribution of discretization nodes $X$
and how the size $N_X$ should be chosen, see
Section~\ref{sssec:choice-of-nodes}.

\subsubsection{Numerical differentiation by tensor-product splines}
\label{sssec:splines}

Let $\Scal$ be an $M$-dimensional linear subspace of $\Dcal$
spanned by a basis $\{s_1,\dots,s_M\}\subset \Dcal$.
Any map 
\[
\Psi \colon \Dcal \to \R^M, \quad
\Psi(u) = \bigl( c_1(u), \dots, c_M(u) \bigr)^T,
\]
can be seen as a discretization operator that gives rise to 
the numerical differentiation  scheme 
\begin{align*}
\Lcal u (y_i)&\approx \Lcal \Big(
	\sum_{j=1}^M c_j(u) s_j \Big) (y_i)
	=\sum_{j=1}^M c_j(u)\Lcal s_j(y_i)
	=(L\Psi(u))_i,\quad i=1,\ldots,N_Y\\
\Bcal u (z_i)&\approx \Bcal \Big(
	\sum_{j=1}^M c_j(u) s_j \Big) (z_i)
	=\sum_{j=1}^M c_j(u)\Bcal s_j(z_i)
	=(B\Psi(u))_i,\quad i=1,\ldots,N_Z,
\end{align*}
given by the matrices $L \in \R^{N_Y \times M}$ and
$B \in \R^{N_Z \times M}$ with entries
\begin{equation}\label{LBS}
\ell_{i j} = \Lcal s_j(y_i)
\quad \text{and} \quad
b_{ij } = \Bcal s_j(z_i).
\end{equation}

Since these matrices do not depend on $\Psi$, the weights
$\muw ^*,\sigmav ^*$ computed by Algorithm~\ref{algo:abstract-algorithm} do
not depend on it either, and hence the choice of $\Psi$ is 
irrelevant for the implementation of our quadrature in this
setting. However, numerical differentiation errors
$$
\eps_L(u)=\max_{i=1,\dots,N_Y} \absB{\Lcal u (y_i) 
- \Lcal \Big(\sum_{j=1}^M c_j(u) s_j \Big)(y_i)},
\quad\eps_B(u)=\max_{i=1,\dots,N_Z} \absB{\Bcal u (z_i) 
- \Bcal \Big(\sum_{j=1}^M c_j(u) s_j \Big)(z_i)}$$  
of \eqref{eq:epsL}--\eqref{eq:epsB} do depend on $\Psi$.
Therefore we take the infimum of \eqref{eq:split-error-estimate-f} and
\eqref{eq:split-error-estimate-g} over all possible $\Psi$, which is equivalent to
the infimum over all $s=\sum_{j=1}^M c_j(u) s_j\in\Scal$,
and implies the estimates
\begin{align}
\label{eq:fbaS}
\absB{\int_\Omega f \dmu - \sum_{i=1}^{N_Y} \muw_i f(y_i)}
&\leq \inf_{u \in \Ucal_f} \inf_{s\in\Scal}\Bigl\{
	\norm{\muw}_1 \|\Lcal (u-s)|_Y\|_\infty
	+ \norm{\sigmav}_1 \|\Bcal (u-s)|_Z\|_\infty
\Bigr\}, \\
\label{eq:gbaS}
\absB{\int_{\partial\Omega} g \dsigma - \sum_{i=1}^{N_Z} \sigmav_i g(z_i)}
&\leq \inf_{u \in \Ucal_g} \inf_{s\in\Scal}\Bigl\{
	\norm{\muw}_1  \|\Lcal (u-s)|_Y\|_\infty
	+ \norm{\sigmav}_1\|\Bcal (u-s)|_Z\|_\infty
\Bigr\},
\end{align}
for any quadrature weights satisfying
\eqref{eq:linear-system-fhat-ghat}.

Note that the weights $\muw^*$ and $\sigmav^*$ do not depend on the 
choice of the basis  $\{s_1,\dots,s_M\}$ of a given space
$\Scal$. Indeed, if we choose another basis and let $V$
be the change-of-basis matrix, then the new differentiation
matrices are given by $\tilde{L} = L V$
and $\tilde{B} = B V$, hence the condition
$\tilde L^T\muw -\tilde B^T\sigmav
= V^{T} (L^T\muw-B^T\sigmav )=0$
is equivalent to $L^T\muw -B^T\sigmav =0$ because $V$ is invertible.
Nevertheless, the matrices $L$ and $B$ still depend on the
choice of the basis, and their properties strongly influence the
efficiency and stability of the computation of the weights. In
particular, in order to obtain sparse  matrices $L$ and $B$, we need
locally supported basis functions $s_j$, such that only a few of
them do not vanish in the vicinity of each point $y_i$ or $z_i$.

Although the space $\Scal$ can be chosen freely,
and in principle it would be interesting to compare
different approaches, in this work we only consider the case
of unfitted tensor-product spline approximation.
Alternative candidates for $\Scal$  that also possess locally
supported bases include finite element spaces, mapped
tensor-product splines used in the isogeometric analysis 
\cite{cottrell2009isogeometric} and various spline 
constructions on triangulations, see e.g.\
\cite{lai2007spline,schumaker2022approximation}.

Suppose that the manifold $\Mcal$ of dimension $d$ is embedded in
$\R^n$, with $n\ge d$,  and let $\bbox$ be the $n$-dimensional
bounding box around $\Omegabar$ such that the lengths of its
sides are integer multiples of a parameter $h_{\Scal} > 0$:
\[
\Omegabar \subset \bbox = [a_1, b_1] \times \dots \times [a_n, b_n],
\qquad b_i - a_i = N_i \, h_{\Scal}
\quad \text{for all $i = 1,\dots,n$}.
\]
For each dimension, let $T_i$ be the uniform knot vector
\begin{equation} \label{eq:knotv}
T_i = \{ a_i, \dots, a_i,
a_i + h_{\Scal}, \dots, b_i - h_{\Scal},
b_i, \dots, b_i \},
\end{equation}
where the first and last knots are repeated $q$ times,
and let $\Scal_i$ be the univariate spline space
of order $q$ (degree $q-1$) defined by the knot vector $T_i$.
We define $\Scal$ to be the restriction of the
tensor-product spline space
$\Scal_1 \otimes \dots \otimes \Scal_n$
to  $\Omegabar$, and a natural basis
for $\Scal$ is given by the restriction to $\Omegabar$
of those tensor-product B-splines $s_1,\ldots,s_M$,
whose supports have non-empty intersection with $\Omega$.
In the case $d<n$ this is known as the \emph{ambient},
or \emph{immersed} approach \cite{lehmann-ambient}.
Clearly, an upper bound for $M = \dim\Scal$ is given by
\[
\prod_{i=1}^n \dim\Scal_i
= \prod_{i=1}^n \bigl( N_i + q - 1 \bigr),
\]
although $M$ will often only be a 
fraction of this upper bound, for example when
$n>d$, or when $n=d$ but $\abs{\Omega} \ll \abs{\bbox}$.
We denote by $D$ the interior of the union of the supports of all
B-splines in $\Scal$. Then $\overline D$ may also be considered
as the domain of definition of the splines in $\Scal$.

B-splines $s_j$ have local supports contained in cubes with side
length $qh_\Scal$, and so the matrices $L$ and $B$ are sparse as
each of their rows has at most $q^n$ nonzeros. Note that the
assembly of $L$ and $B$ does not require testing
whether the intersection of the support of $s_j$
with $\Omega$ is non-empty: we can simply discard
any basis function such that $s_j|_{Y\cup Z} \equiv 0$,
as done in Algorithm~\ref{algo:spline}.
 
Assume that $n=d$, such that  $\Omega\subset D$ are domains in $\R^d$.
As before, we denote by $k_L$ and $k_B$
the orders of the operators $\Lcal$ and $\Bcal$, respectively,
with $\max\{k_L,k_B\}< q$.
Then it follows by \cite[Theorem 5.1]{deboor-spline} that
for any $u\in W^{q}_\infty(D)$ there exists $s\in\Scal$
such that simultaneously
$$
\|\Lcal (u-s)|_Y\|_\infty\le C_{L}h_{\Scal}^{q-k_L}|u|_{W^{q}_\infty(D)}
\quad\text{and}\quad 
\|\Bcal (u-s)|_Z\|_\infty\le
C_{B}h_{\Scal}^{q-k_B}|u|_{W^{q}_\infty(D)},$$
where $C_L,C_B$ are some positive constants independent of $h_L,h_B$ and $u$.
Therefore, we obtain from \eqref{eq:fbaS} and \eqref{eq:gbaS}
\begin{align}
\label{eq:fbaBS}
\absB{\int_\Omega f \dmu - \sum_{i=1}^{N_Y} \muw_i f(y_i)}
&\leq \inf_{u \in \Ucal_f} \Bigl\{
	C_L\norm{\muw}_1 h_{\Scal}^{q-k_L}
	+ C_B\norm{\sigmav}_1h_{\Scal}^{q-k_B} 
\Bigr\}|u|_{W^{q}_\infty(D)}, \\
\label{eq:gbaBS}
\absB{\int_{\partial\Omega} g \dsigma - \sum_{i=1}^{N_Z} \sigmav_i g(z_i)}
&\leq \inf_{u \in \Ucal_g} \Bigl\{
	C_L\norm{\muw}_1 h_{\Scal}^{q-k_L}
	+ C_B\norm{\sigmav}_1h_{\Scal}^{q-k_B} 
\Bigr\}|u|_{W^{q}_\infty(D)},
\end{align}
as long as $\Ucal_f$ and $\Ucal_g$ contain functions extensible
to $D$ with a finite seminorm $|u|_{W^{q}_\infty(D)}$, and the
weight vectors $w,v$ satisfy \eqref{eq:linear-system-fhat-ghat}.

Similar to Section~\ref{sssec:mfd}, for a smooth domain
$\Omega\subset\R^d$ and operators $\Lcal=\Delta$,
$\Bcal=\partial_\nu$ with $k_L=2$, $k_B=1$, or operators 
$\Lcal=\diver$, $\Bcal=\gamma_\nu$ with $k_L=1$, $k_B=0$,
assuming that $\fhat$
and $\ghat$ are infinitely differentiable, we may use
Proposition~\ref{prop:well-posedness-neumann}, the Sobolev
embedding theorem, and the Stein extension
theorem, to obtain the $\Ocal(h_{\Scal}^{q-k_L})$
estimates 
\begin{equation}\label{eq:fbBS}
\absB{\int_\Omega f \dmu - \sum_{i=1}^{N_Y} \muw_i^* f(y_i)}
\le C_1\big(\norm{\muw ^*}_1+ \norm{\sigmav ^*}_1 \big)h_{\Scal}^{q-k_L}
	\|f\|_{H^{s}(\Omega)}
\end{equation}
for $f\in H^{s}(\Omega)$ with $s>q-k_L+d/2$, and
\begin{equation}\label{eq:gbBS}
\absB{\int_{\partial\Omega} g \dsigma - \sum_{i=1}^{N_Z} \sigmav_i^* g(z_i)}
\le C_2\big(\norm{\muw ^*}_1+ \norm{\sigmav ^*}_1 \big)h_{\Scal}^{q-k_L}
	\|g\|_{H^{s}(\partial\Omega)}
\end{equation}
for $g\in H^{s}(\partial\Omega)$ with $s>q-k_L+(d+1)/2$,
where the constants $C_1$ and $C_2$ are independent of $f$ and $g$.

We refer to \cite{lehmann-ambient} for approximation 
results by ambient tensor-product splines when $\Omega$ is a
compact closed hypersurface. %

Elements of the basis of $\Scal$ obtained by
restricting B-splines to $\Omega$
may have arbitrarily small support, and this is a known
cause of instability in some numerical schemes
of immersed type, see for example 
\cite{chu2022stabilization,davydov2014two,de2023stability}.
In the context of Algorithm~\ref{algo:abstract-algorithm},
however, we did not encounter situations where
instability of $w$ and $v$ would likely be caused by small 
cut elements, and so we did not
investigate the use of any special stabilization techniques,
see the remarks at the end of Section~\ref{sec:numerical-tests}
for more information.

\subsection{On the choice of minimization norm}
\label{ssec:on-the-choice-of-minimization-norm}

As we have seen in Section~\ref{ssec:on-the-choice-of-L-and-B},
Algorithm~\ref{algo:abstract-algorithm} leads to quadrature
formulas of high convergence order assuming sufficient regularity
of the boundary value problem \eqref{eq:bvp-joint-error-estimate},
the use of an appropriate numerical differentiation scheme,
and stability of the quadrature formulas, that is,
$\norm{\muw }_1 = \Ocal(1)$ and $\norm{\sigmav }_1 = \Ocal(1)$.
This suggests to employ the combined 1-norm
\[
\|(\muw,\sigmav)\|_\sharp
= \|(\muw,\sigmav)\|_1 = \norm{\muw}_1+\norm{\sigmav}_1
\]
in Step 5 of the algorithm.

However, the main reason why the 1-norm appears in the
estimates is that in the error analysis it was convenient
to use the 1-norms
of the quadrature weight vectors $\muw$, $\sigmav$,
and the $\infty$-norm of the error of numerical differentiation.
In fact, H\"older's inequality with any pair of conjugated 
exponents $p,p' \in [1,\infty]$ such that $1/p + 1/p' = 1$,
leads to the estimates 
\begin{align*}
\absB{\int_\Omega f \dmu
	- \sum_{i=1}^{N_Y} \muw_i f(y_i)}
&\leq \inf_{u \in \Ucal_f} %
\Bigl\{
	\norm{\muw}_p \norm{\Lcal u \restrict{Y} - L\Psi(u)}_{p'}
	+ \norm{\sigmav}_p \norm{\Bcal u \restrict{Z} - B\Psi(u)}_{p'}
\Bigr\}, \\
\absB{\int_{\partial\Omega} g \dsigma
	- \sum_{i=1}^{N_Z} \sigmav_i g(z_i)}
&\leq \inf_{u \in \Ucal_g} %
\Bigl\{
	\norm{\muw}_p \norm{\Lcal u \restrict{Y} - L\Psi(u)}_{p'}
	+ \norm{\sigmav}_p \norm{\Bcal u \restrict{Z} - B\Psi(u)}_{p'}
\Bigr\},
\end{align*}
where the pair $p=1$ and $p'=\infty$ as in \eqref{eq:fba}
and \eqref{eq:gba} is not necessarily optimal.

Nevertheless, there are still good reasons why one may prefer
the minimization of the 1-norm of the quadrature weights.
First of all, the 1-norm is the standard way
to assess stability of a quadrature formula: any
perturbation of size $\varepsilon > 0$ in the values
$f(y_i)$ and $g(z_i)$ is potentially amplified by
$\norm{\muw}_1$ and $\norm{\sigmav}_1$ when computing the sums
$\sum_{i=1}^{N_Y} \muw_i f(y_i)$ and 
$\sum_{i=1}^{N_Z} \sigmav_i g(z_i)$.

Second, minimizing the 1-norm can be a way to obtain
\emph{positive formulas} with non-negative weights
$\muw_i\ge0$ or $\sigmav_i\ge0$, which is a highly
desirable property in many applications, although
far from necessary for achieving stability of numerical
integration in general. If we know the measures
\[
\abs{\Omega} = \int_\Omega 1 \dmu, \qquad
\abs{\partial\Omega} = \int_{\partial\Omega} 1 \dsigma,
\]
of the domain and its boundary,
then we can choose $(\fhat,\ghat) \equiv (1,-1)$
as auxiliary functions, so that
\begin{equation}\label{comb_exact}
\sum_{i=1}^{N_Y} \muw_i +\sum_{i=1}^{N_Z} \sigmav_i = \abs{\Omega} 
+ \abs{\partial\Omega}
\end{equation}
holds for all quadrature formulas $(Y,\muw)$ and $(Z,\sigmav)$
satisfying \eqref{eq:linear-system-fhat-ghat}. Hence,
\[
\abs{\Omega} + \abs{\partial\Omega}
\leq \norm{\muw}_1 + \norm{\sigmav}_1,
\]
and equality is attained if and only if the weights
$\muw$ and $\sigmav$ are non-negative. If a pair of non-negative
weight vectors $(\muw,\sigmav)$ satisfying system
\eqref{eq:linear-system-fhat-ghat} exists, then it will be found
as $(\muw^*,\sigmav^*)$ by 1-norm minimization, and the combined
1-norm $\|(\muw^*,\sigmav^*)\|_1$ will be equal to
$\abs{\Omega} + \abs{\partial\Omega}$,
implying excellent stability.
Alternatively, knowing $\abs{\Omega}$ and $\abs{\partial\Omega}$ we 
may ensure that both formulas $(Y,\muw)$ and $(Z,\sigmav)$ 
are exact for constants by requiring two separate conditions 
\begin{equation}\label{sep_exact}
\sum_{i=1}^{N_Y} \muw_i = \abs{\Omega}\quad\text{and}\quad 
\sum_{i=1}^{N_Z} \sigmav_i = \abs{\partial\Omega}
\end{equation}
instead of \eqref{comb_exact}. Note that when imposing one or both
conditions in \eqref{sep_exact} there is no guarantee that 1-norm 
minimization delivers a positive formula for either $w$ or $v$, 
even if they exist, because we minimize the sum
$\|w\|_1+\|v\|_1$ rather than $\|w\|_1$ or $\|v\|_1$ separately.

When using more than one auxiliary condition, for example by
imposing exactness of the combined quadrature formula for two
pairs of functions $(\fhat_1,\ghat_1)$ and
$(\fhat_2,\ghat_2)$ as in \eqref{sep_exact}, the error analysis
in Proposition~\ref{prop:joint-error-estimate}
can be modified to take this into account.
The boundary value problem \eqref{eq:bvp-joint-error-estimate}
is replaced by
\begin{equation} \label{eq:bvp-two-pairs-fhat-ghat}
\begin{cases}
\Lcal u = f -\alpha_1 \fhat_1 -\alpha_2 \fhat_2
	& \text{in $\Omega$}, \\
\Bcal u = g -\alpha_1 \ghat_1 -\alpha_2 \ghat_2
	& \text{on $\partial\Omega$},
\end{cases}
\end{equation}
the term $\abs{\alpha} \hat{\varepsilon}$ in
inequality \eqref{eq:joint-error-estimate} becomes
\[
\abs{\alpha_1} \hat{\varepsilon}_1
+ \abs{\alpha_2} \hat{\varepsilon}_2,
\quad \text{with $\hat{\varepsilon}_1 = |\delta(\fhat_1,\ghat_1)|$
and $\hat{\varepsilon}_2 = |\delta(\fhat_2,\ghat_2)|$},
\]
and by taking the infimum over all pairs of coefficients 
$\alpha_1,\alpha_2$ satisfying the  compatibility condition
$$
\int_\Omega f -\alpha_1 \fhat_1 -\alpha_2 \fhat_2\dmu
= \int_{\partial\Omega} g -\alpha_1 \ghat_1 -\alpha_2
\ghat_2\dsigma$$
and over all solutions $u$ of \eqref{eq:bvp-two-pairs-fhat-ghat} for
these pairs $(\alpha_1,\alpha_2)$, we may improve the estimate 
\eqref{eq:joint-error-estimate} compared to a single
pair of auxiliary functions $(\fhat_1,\ghat_1)$ or
$(\fhat_2,\ghat_2)$. However, in our numerical experiments in 
Section~\ref{sssec:choice-fghat}
 we did not observe a significant improvement
of the accuracy of the quadrature formulas
when using more than one pair $(\fhat,\ghat)$.

From a computational point of view, any 1-norm minimization
problem can be reformulated as a linear program
with twice as many unknowns, so any general purpose
linear programming solver can be used for 1-norm
minimization. When the simplex algorithm is used,
\emph{sparse} interior quadrature formulas $\muw$
with no more than $m$ nonzero weights can be obtained,
with $m$ being the total number of rows in the
linear system \eqref{eq:linear-system-fhat-ghat}.

The choice $p = 2$ is very attractive from the computational
point of view: efficient linear algebra routines
based on e.g.\ QR decomposition are available
to find a solution of system
\eqref{eq:linear-system-fhat-ghat}
with minimal combined 2-norm
\[
\norm{(\muw,\sigmav)}_2
\deq \sqrt{\norm{\muw}_2^2 + \norm{\sigmav}_2^2}.
\]
Moreover, the strict convexity
of the 2-norm guarantees uniqueness of the optimal
solution $(\muw^*,\sigmav^*)$, a property that may not
hold for the 1-norm.
For a numerical comparison of 1-norm
and 2-norm minimization, we refer to the experiments
in Section~\ref{sssec:choice-norm}.

More generally, one can consider a weighted 2-norm
\[
\norm{(\muw,\sigmav)}_{2,\Lambda}
\deq \sqrt{\lambda_1 \norm{\muw}_2^2 + \lambda_2 \norm{\sigmav}_2^2}
\]
to balance the asymptotic sizes of
$\norm{w}_2$ and $\norm{v}_2$ as the spacing
parameters $h$ and $h_{\Scal}$
introduced in Section~\ref{ssec:on-the-choice-of-L-and-B}
approach zero, under the assumptions of quasi-uniform nodes
and the expectation that 
\[
\norm{w}_\infty = \Ocal(N_Y^{-1}),
\quad \norm{v}_\infty = \Ocal(N_Z^{-1}),
\]
where $N_Z\ll N_Y$ as $Z\subset\partial\Omega$.
Although this idea is theoretically sound because it is easy
to see that the unweighted 2-norm is unbalanced,
numerical experiments do not show significant improvements in the
accuracy and stability of our quadrature formulas
after switching to a suitably balanced weighted 2-norm.
For this reason, and to keep complexity to a minimum,
we have decided not to investigate the use of weighted
2-norms in this paper.

\subsection{Practical choices for effective algorithms}
\label{ssec:palgos}

After discussing in
Sections~\ref{ssec:on-the-choice-of-fhat-and-ghat}--\ref{ssec:on-the-choice-of-minimization-norm} 
various options available for the realization of the framework
outlined in Algorithm~\ref{algo:abstract-algorithm}, we now
describe two particular settings 
tested in the numerical experiments of
Section~\ref{sec:numerical-tests} and recommended for practical
applications. 
The evidence gathered in our experiments suggests
that the choices presented below in Algorithms~\ref{algo:mfd} and
\ref{algo:spline} are essentially optimal among
the ones that will be compared in
Section~\ref{ssec:validation-of-parameters}, and therefore
represent a good starting point for applications
and future research.
The first algorithm is based on meshless finite difference
formulas, whereas the second one is based on collocation
of tensor-product spline spaces defined on a bounding
box around $\Omega$.
In this way we demonstrate that both the functional and
meshless finite difference approaches have effective
realizations. 

We only consider 
domains $\Omega$ in $\R^d$, even if the methods are applicable
to surfaces and other manifolds.
We assume that $|\partial\Omega|$
is known, and choose
\[
(\fhat,\ghat) \equiv (0,1),
\qquad \Lcal = \diver,
\qquad \Bcal = \gamma_\nu,
\qquad \norm{(\muw,\sigmav)}_\sharp = \norm{(\muw,\sigmav)}_2
= \sqrt{\norm{\muw}_2^2 + \norm{\sigmav}_2^2}.
\]
For simplicity, we assume that the quadrature nodes in $Y$ and $Z$, 
although irregular in general, are not intentionally generated
with density varying over the domain or its boundary. Therefore
we characterize their density by a single
\emph{spacing parameter} $h > 0$, meaning that $Y$ and $Z$ are 
produced by some node generation method with a uniform target
spacing depending on $h$. We assume that the \emph{packing
spacing} of the node sets, defined as 
\begin{equation}\label{packsp}
h_{ps}(Y) \deq \left( \frac{\abs{\Omega}}{N_Y} \right)^{1/d}, \quad
h_{ps}(Z) \deq \left( \frac{\abs{\partial\Omega}}{N_Z}\right)^{1/(d-1)},
\end{equation}
is  proportional to the spacing parameter $h$ of the node
generation method. Note that $h_{ps}(Y)$ is the step size 
of a uniform Cartesian grid with $N_Y$ nodes in a
$d$-dimensional cube of measure $|\Omega|$, and similarly for $h_{ps}(Z)$.
For some methods we may expect that $h_{ps}(Y) \approx h$ and $h_{ps}(Z)
\approx h$, but in other cases the packing spacing may be larger
or smaller than the spacing parameter of the method by some factor.
We expect the algorithms proposed below to perform
well if the nodes are of sufficiently high quality,
see the discussion of node quality measures
and meshless node generation algorithms in
\cite{suchde-point}. Node sets obtained by some of these 
algorithms are tested in Section~\ref{sssec:choice-of-nodes},
which leads to a recommendation to use advancing front method
for node generation unless the nodes are already
provided by the user.

We assume that, in addition to supplying the quadrature nodes
in $Y$ and $Z$, the user provides sufficiently accurate
outward-pointing unit normals at the boundary
nodes of $Z \subset \partial\Omega$,
and, in the case of Algorithm~\ref{algo:mfd}, is in position to
generate an additional node set $X$ in $\Omegabar$ targeting
a prescribed spacing parameter $h_X > 0$.

\begin{algotheorem} \label{algo:mfd}
Approach based on meshless finite difference formulas (MFD).\\[0.5em]
\textbf{Input:} Sets of quadrature nodes $Y\subset\Omegabar$
and $Z\subset \partial\Omega$ with spacing parameter $h>0$.\\
\textbf{Input:} Outward-pointing unit normals $\nu|_Z$
at the nodes in $Z$.\\
\textbf{Input:} Measure $\abs{\partial\Omega}$ of the boundary.\\
\textbf{Input:} Polynomial order $q \geq 2$.\\[0.5em]
\textbf{Output:} Quadrature weights $w$ and $v$.
\begin{algorithmic}[1]
\vspace*{0.5em}
\State Set the spacing parameter $h_X=1.6\,h$.
\State Generate a node set $X$ of $N_X$ nodes in $\Omegabar$ according to
	the spacing parameter $h_X$.
\State Generate a k-d tree to facilitate nearest neighbors
	searches on $X$.
\State Set \[ n_L = 2 \binom{q-1+d}{d}
	\qquad \text{and} \qquad n_B = 2 \binom{q-2+d}{d}. \]
\For{$i=1,\dots,N_Y$}
	\State Query the k-d tree and store in $S_{L,i}$
	the indices of the $n_L$ nodes in $X$ closest to $y_i$.
\EndFor
\For{$i=1,\dots,N_Z$}
	\State Query the k-d tree and store in $S_{B,i}$
	the indices of the $n_B$ nodes in $X$ closest to $z_i$.
\EndFor
\For{$k = 1, \dots, d$}
	\State \parbox[t]
	{\dimexpr\textwidth-\leftmargin-\labelsep-\labelwidth}
	{Assemble the sparse matrix $L_k \in \R^{N_Y \times N_X}$
	whose nonzero elements $\ell_{kij}$ are differentiation
	weights for the discretization of operator $\partial_{x_k}$
	on nodes $X$ and $Y$:
	\[
	\partial_{x_k} u(y_i)
	\approx \sum_{j \in S_{L,i}} \ell_{kij} u(x_j).
	\]
	The weights $\ell_{kij}$ are generated using
	polyharmonic radial basis
	kernels with power $2q-1$ and polynomial term in $\Pi_q^d$,
	as explained in Section \ref{sssec:mfd}.}
	\State Assemble the diagonal matrix
	$D_k \in \R^{N_Z \times N_Z}$ whose elements
	are the $k$-th components of $\nu|_Z$.
\EndFor
\State \parbox[t]
	{\dimexpr\textwidth-\leftmargin-\labelsep-\labelwidth}
	{Assemble the sparse matrix $\tilde B \in \R^{N_Z \times N_X}$
	whose nonzero elements $b_{ij}$ are approximation weights
	for the pointwise values at the nodes in $Z$ from the nodes of $X$:
	\[
	u(z_i) \approx \sum_{j \in S_{B,i}} b_{ij} u(x_j).
	\]
	The weights are generated using polyharmonic radial basis
	kernels with power $2q-3$ and polynomial term in $\Pi_{q-1}^d$,
	as explained in Section \ref{sssec:mfd}.}
\State Concatenate $L_1,\dots,L_d$ horizontally
	into $L \in \R^{N_Y \times dN_X}$.
\State Concatenate the products $D_1 \tilde B,\dots,D_d \tilde B$
	horizontally into $B \in \R^{N_Z \times dN_X}$.
\State Assemble the sparse matrix
	\[ A = \begin{pmatrix}
	L^T & -B^T \\ 
	0 & \mathbbm{1}^T
	\end{pmatrix}, \]
	with $\mathbbm{1}$ being the all-ones column vector
	of dimension $N_Z$.
\State Assemble the right hand side
	$b = (0, \abs{\partial\Omega})^T$.
\State Compute the solution to $Ax = b$ with smallest 2-norm.
\State Split $x$ into subvectors $\muw \in \R^{N_Y}$
	and $\sigmav \in \R^{N_Z}$.
\end{algorithmic}
\end{algotheorem}

Note that in order to increase the performance of Algorithm~\ref{algo:mfd},
one should assemble the sparse matrices $L^T$ and $B^T$ directly,
instead of precomputing matrices $L_k$ and $D_k$ for each
spatial dimension separately.
Moreover, the index sets $S_{L,i}$ and $S_{B,i}$ can be
determined on the fly during the assembly procedure.
Nevertheless, we have chosen to present Algorithm~\ref{algo:mfd}
in the form above in order to enhance readability, and to clearly
show how $L^T$ and $B^T$ can be put together in the divergence
case, which has not been detailed out in
Section~\ref{ssec:on-the-choice-of-L-and-B}.
Indeed, we have assumed in that section that
functions in $\Dcal$ are scalar valued, but in the case
of $\Lcal = \diver$ and $\Bcal = \gamma_\nu$
the functions in $\Dcal$ are actually vector valued.

In the scalar case, the discretization operator
$\Psi \colon \Dcal \to \R^{N_X}$ is defined by pointwise
evaluation over $X$, see \eqref{eq:Psi-u}, and so the elements
in the vector $\Psi(u)$ follow the order of the nodes in $X$.
In the vector valued case, the discretization operator
$\Psi(F) \in \R^{dN_X}$ is defined by pointwise
evaluation over $X$ of all $d$ components of the vector field
$F \in \Dcal$, and so one must choose whether to order
the elements of $\Psi(F)$ so that all $d$ components for a
fixed node are contiguous, as in
\[
\Psi(F) = \bigl(
F_1(x_1), \dots, F_d(x_1),
\;\; \dots, \;\;
F_1(x_{N_X}), \dots, F_d(x_{N_X}) \bigr)^T,
\]
or so that all $N_X$ pointwise values for a fixed component
are contiguous, as in
\[
\Psi(F) = \bigl(
F_1(x_1), \dots, F_1(x_{N_X}),
\;\; \dots, \;\;
F_d(x_1), \dots, F_d(x_{N_X}) \bigr)^T.
\]
In Algorithm~\ref{algo:mfd} we have chosen the latter ordering,
and although the discretization operator $\Psi$ is not
used in the algorithm, the chosen ordering for $\Psi(F)$ clearly
determines the order of the columns of $L$ and $B$,
and hence their assembly.
In a high performance implementation, the order
that provides the fastest memory access should be preferred.
In any case, the final values of $\muw$ and $\sigmav$
do not depend on the chosen ordering for $\Psi(F)$,
because changing the order amounts to permuting the rows
of the linear system $Ax=b$.
Similar considerations about ordering also apply to the
next algorithm, based on collocation of vector fields
whose components are uniform tensor-product splines defined
on a bounding box around $\Omega$. A basis for this space
of vector fields is defined component-by-component using
the same scalar basis of B-splines for each component.
Once again, the ordering of basis elements in $\Psi(F)$
is completely arbitrary, and the one that provides the
fastest memory access should be preferred in practice.

\begin{algotheorem} \label{algo:spline}
Approach based on collocation of a tensor-product
spline space (BSP).\\[0.5em]
\textbf{Input:} Sets of quadrature nodes $Y\subset\Omegabar$
and $Z\subset \partial\Omega$ with spacing parameter $h>0$.\\
\textbf{Input:} Outward-pointing unit normals $\nu|_Z$
at the nodes in $Z$.\\
\textbf{Input:} Measure $\abs{\partial\Omega}$ of the boundary.\\
\textbf{Input:} Polynomial order $q \geq 2$.\\[0.5em]
\textbf{Output:} Quadrature weights $w$ and $v$.
\begin{algorithmic}[1]
\vspace*{0.5ex}
\State Set the knot spacing parameter $h_{\Scal} = 4 \, h$.
\State Compute $\bbox = [a_1,b_1] \times \dots \times [a_d,b_d]$,
	a bounding box around $\Omegabar$ such that the
	length of its sides are multiples of $h_{\Scal}$.
\For{$i = 1, \dots, d$}
	\State Compute the knot vector $T_i$
		on $[a_i,b_i]$ with uniform spacing $h_{\Scal}$,
		as defined in \eqref{eq:knotv}.
\EndFor
\State Let $\{s_1,\dots,s_{N_{\Scal}}\}$ be the B-spline basis
	of the tensor-product space
	$\Scal_1 \otimes \dots \otimes \Scal_d$,
	where each $\Scal_i$ is the univariate spline space
	with knot vector $T_i$ and degree $q-1$.
\For{$k = 1, \dots, d$}
	\State \parbox[t]
	{\dimexpr\textwidth-\leftmargin-\labelsep-\labelwidth}
	{Assemble the sparse collocation matrix
	$L_k \in \R^{N_Y \times N_{\Scal}}$ whose nonzero
	elements $\ell_{kij}$
	are given by the evaluation of the $k$-th partial derivative
	of $B$-splines at the nodes of $Y$:
	\[
	\ell_{kij} = \partial_{x_k} s_j(y_i)
	\]
	For each node $y_i$, only $q^d$ B-splines need to be evaluated,
	i.e.\ the ones whose support contains $y_i$.}
	\State Assemble the diagonal matrix
	$D_k \in \R^{N_Z \times N_Z}$ whose elements
	are the $k$-th components of $\nu|_Z$.
\EndFor
\State \parbox[t]
	{\dimexpr\textwidth-\leftmargin-\labelsep-\labelwidth}
	{Assemble the sparse collocation matrix
	$\tilde B \in \R^{N_Z \times N_{\Scal}}$
	whose nonzero elements $b_{ij}$
	are given by evaluation of $B$-splines at the nodes of $Z$:
	\[
	b_{ij} = s_j(z_i)
	\]
	For each node $z_i$, only $q^d$ B-splines need to be evaluated,
	i.e.\ the ones whose support contains $z_i$.}
\State Concatenate $L_1,\dots,L_d$ horizontally
	into $L \in \R^{N_Y \times dN_{\Scal}}$.
\State Concatenate the products $D_1 \tilde B,\dots,D_d \tilde B$
	horizontally into $B \in \R^{N_Z \times dN_{\Scal}}$.
\State Assemble the sparse matrix
	\[ A = \begin{pmatrix}
	L^T & -B^T \\ 
	0 & \mathbbm{1}^T
	\end{pmatrix}, \]
	with $\mathbbm{1}$ being the all-ones column vector
	of dimension $N_Z$. Any row whose elements are all zeros
	is not included in $A$. Such rows come from B-splines
	whose support is disjoint from $Y \cap Z$.
\State Assemble the right hand side
$b = (0, \abs{\partial\Omega})^T$.
\State Compute the solution to $Ax = b$ with smallest 2-norm.
\State Split $x$ into subvectors $\muw \in \R^{N_Y}$
and $\sigmav \in \R^{N_Z}$.
\end{algorithmic}
\end{algotheorem}

Once again, we have written Algorithm~\ref{algo:spline}
in a way that enhances readability rather than performance.

We conclude this section with three remarks.

First, we note that the set $X$ of Algorithm~\ref{algo:mfd} may
also be generated as a Cartesian grid over the bounding box $\bbox$
as defined in Algorithm~\ref{algo:spline}. This version of MFD
delivered acceptable results in our numerical experiments,
although it is inferior in accuracy and stability to
the approach described in the algorithm.
Generating a node set $X\subset\Omegabar$ targeting a prescribed 
spacing parameter $h_X$ is not a restrictive hypothesis in practice, 
because, whenever node generation is deemed expensive due to the
complexity of~$\Omegabar$, one may obtain $X$ by thinning 
the set $Y \cup Z$. 

Second, Algorithms~\ref{algo:mfd} and \ref{algo:spline} can be 
adapted to the setting where the quadrature nodes $Y$ and $Z$ 
are generated with locally varying density according to a
\emph{spacing function} $h(x) \colon \Omegabar \to \R_+$. In this
case, the set $X$ of Algorithm~\ref{algo:mfd} may be generated by
using a scaled spacing function $h_X(x) = c\, h(x)$ for some
$c>1$, either from scratch, or by a suitable
subsampling algorithm, see e.g.\ \cite{lawrence2024node}. 
In the case of Algorithm~\ref{algo:spline}, the spacing function
$h(x)$ may be used  to guide local refinement of the
tensor-product spline space.

Third, the choice of the coefficients in $h_X = 1.6 \, h$ and 
$h_{\Scal} = 4 \, h$ is justified by the numerical experiments of
Section~\ref{sec:numerical-tests}.
For now, we point out that these choices lead
by design to underdetermined systems $Ax=b$.
To see why, let us consider the case of Algorithm~\ref{algo:mfd}.
Assuming that $N_Y \gg N_Z$, the matrix $A$
has approximately $N_Y$ columns and $dN_X$ rows.
By the definitions of $h$ and $h_X$, assuming that 
$\big( \frac{\abs{\Omega}}{N_X} \big)^{1/d}\approx h_X $, we have
\[
\frac{d N_X}{N_Y}
\approx \frac{d \abs{\Omega} h_X^{-d}}{\abs{\Omega} h^{-d}}
= d (1.6)^{-d},
\]
and so the ratio of rows to columns is always smaller
than 1 for all $d \in \N$, with a peak of about 0.8
for $d=2$. A similar argument shows that the linear system
$Ax=b$ assembled in Algorithm~\ref{algo:spline} is
also underdetermined for small enough $h$.

\graphicspath{{./figures/}}

\pgfplotstableset{
	search path = ./tables,
	col sep = comma,
	string replace = {NaN}{},
	empty cells with = {-},
	set thousands separator = {},
	every column/.append style = %
		{zerofill, sci e, precision = 2}
}

\pgfplotstableset{
	columns/Nsamples/.style={
		column name=$N_{\text{seeds}}$, int detect
	},
	columns/hX/.style={
		column name=$h_X$, sci
	},
	columns/hY/.style={
		column name=$h$, sci
	},
	columns/hZ/.style={
		column name=$h_Z$, sci
	},
	columns/NX/.style={
		column name=$N_X$, int detect
	},
	columns/NY/.style={
		column name=$N_Y$, int detect
	},
	columns/NZ/.style={
		column name=$N_Z$, int detect
	},
	columns/NrowsA/.style={
		column name=$m$, int detect
	},
	columns/NrowsA_mfd/.style={
		column name=$m_{\text{MFD}}$, int detect
	},
	columns/NrowsA_bsp/.style={
		column name=$m_{\text{BSP}}$, int detect
	},
	columns/sizeratioA/.style={
		column name=$m/N$, fixed
	},
	columns/sizeratioA_mfd/.style={
		column name=$m_{\text{MFD}}/N$, fixed
	},
	columns/sizeratioA_bsp/.style={
		column name=$m_{\text{BSP}}/N$, fixed
	},
	columns/nnzA/.style={
		column name=$nnz(A)$, int detect
	},
	columns/nnzA_mfd/.style={
		column name=$nnz_{\text{MFD}}$, sci
	},
	columns/nnzA_bsp/.style={
		column name=$nnz_{\text{BSP}}$, sci
	},
	columns/residual/.style={
		column name=$\norm{Ax-b}$, sci
	},
	columns/tnodegen/.style={
		column name=$t_{\text{nodes}}$, sci
	},
	columns/tmatassembly/.style={
		column name=$t_{\text{mat}}$, sci
	},
	columns/tsolve/.style={
		column name=$t_{\text{solve}}$, sci
	},
	columns/sreintrungeavg/.style={
		column name=$e_{\text{avg}}(f_1)$, sci
	},
	columns/sreintrungestd/.style={
		column name=$e_{\text{std}}(f_1)$, sci
	},
	columns/sreintrungerms/.style={
		column name=$e_{\text{rms}}(f_1)$, sci
	},
	columns/sreintrungemax/.style={
		column name=$e_{\text{max}}(f_1)$, sci
	},
	columns/srebndrungeavg/.style={
		column name=$e_{\text{avg}}(g_1)$, sci
	},
	columns/srebndrungestd/.style={
		column name=$e_{\text{std}}(g_1)$, sci
	},
	columns/srebndrungerms/.style={
		column name=$e_{\text{rms}}(g_1)$, sci
	},
	columns/srebndrungemax/.style={
		column name=$e_{\text{max}}(g_1)$, sci
	},
	columns/sreintfrankeavg/.style={
		column name=$e_{\text{avg}}(f_2)$, sci
	},
	columns/sreintfrankestd/.style={
		column name=$e_{\text{std}}(f_2)$, sci
	},
	columns/sreintfrankerms/.style={
		column name=$e_{\text{rms}}(f_2)$, sci
	},
	columns/sreintfrankemax/.style={
		column name=$e_{\text{max}}(f_2)$, sci
	},
	columns/srebndfrankeavg/.style={
		column name=$e_{\text{avg}}(g_2)$, sci
	},
	columns/srebndfrankestd/.style={
		column name=$e_{\text{std}}(g_2)$, sci
	},
	columns/srebndfrankerms/.style={
		column name=$e_{\text{rms}}(g_2)$, sci
	},	
	columns/srebndfrankemax/.style={
		column name=$e_{\text{max}}(g_2)$, sci
	},	
	columns/sreintexpsumavg/.style={
		column name=$e_{\text{avg}}(f_3)$, sci
	},
	columns/sreintexpsumstd/.style={
		column name=$e_{\text{std}}(f_3)$, sci
	},
	columns/sreintexpsumrms/.style={
		column name=$e_{\text{rms}}(f_3)$, sci
	},
	columns/sreintexpsummax/.style={
		column name=$e_{\text{max}}(f_3)$, sci
	},
	columns/srebndexpsumavg/.style={
		column name=$e_{\text{avg}}(g_3)$, sci
	},
	columns/srebndexpsumstd/.style={
		column name=$e_{\text{std}}(g_3)$, sci
	},
	columns/srebndexpsumrms/.style={
		column name=$e_{\text{rms}}(g_3)$, sci
	},
	columns/srebndexpsummax/.style={
		column name=$e_{\text{max}}(g_3)$, sci
	},
	columns/stabw/.style={
		column name=$K_{\muw}$, fixed, dec sep align ={c}, precision = 2
	},
	columns/stabv/.style={
		column name=$K_{\sigmav}$, fixed
	},
	columns/stabwmax/.style={
		column name=$\max K_{\muw}$, fixed,
		dec sep align ={c}, precision = 2
	},
	columns/stabvmax/.style={
		column name=$\max K_{\sigmav}$, fixed,
		dec sep align ={c}, precision = 2
	},
	columns/minw/.style={
		column name=$\min{\muw}$, sci
	},
	columns/minv/.style={
		column name=$\min{\sigmav}$, sci
	},
	columns/nnzw/.style={
		column name=$\nnz(\muw)$, int detect
	},
	columns/nnzv/.style={
		column name=$\nnz(\sigmav)$, int detect
	},
	columns/nnzwv/.style={
		column name=$\nnz(\muw)+\nnz(\sigmav)$, int detect
	}
}

\section{Numerical tests}
\label{sec:numerical-tests}

In this section we numerically evaluate the accuracy and 
stability of the quadrature formulas $(Y,w)$ and $(Z,v)$
produced by Algorithms~\ref{algo:mfd} and \ref{algo:spline}
on several 2D and 3D domains (denoted by
$\Omega_1,\dots,\Omega_7$, see Figure~\ref{fig:domains}).
Let $\Omega$ be a bounded Lipschitz domain in $\R^2$ or $\R^3$
with piecewise smooth boundary.
Quadrature errors of $(Y,w)$ and $(Z,v)$ are computed
for two test functions $f_1, f_2 \in C^\infty(\Omegabar)$, and,
respectively, their restrictions to the boundary
$g_1 = f_1\restrict{\partial\Omega}$ and 
$g_2 = f_2\restrict{\partial\Omega}$.
The first test function $f_1$ is a multidimensional
generalization of Runge's function centered at
a domain-dependent point $x_R$ in the interior of $\Omega$:
\[
f_1(x,x_R) = \frac{1}{1+25\norm{x-x_R}_2^2}.
\]
The second test function $f_2$ is a scaled and translated
version of the Franke function \cite{franke1979critical} in 2D,
and its 3D generalization by Renka defined in \cite{renka1988multivariate}.
The Franke and Renka functions are meant to be evaluated over $[0,1]^d$,
but our domains are centered at the origin. For this reason,
we compose these functions with the affine mapping
\[
(x_1,\dots,x_d) \mapsto
\left( \frac{x_1+1}{2}, \dots, \frac{x_d+1}{2} \right),\quad
d=2,3,
\]
between $[0,1]^d$ and $[-1,1]^d$.

The accuracy of the quadrature formulas is assessed
by computing the relative errors
\[
e(f_1) \deq \abs{\delta(f_1,0)} / \abs{I(f_1,0)}\;,
\;\ldots\;,\; %
e(g_2) \deq \abs{\delta(0,g_2)} / \abs{I(0,g_2)}.
\]
Note that the denominators are large enough in all tests, which
makes relative errors meaningful.
Relative errors were preferred to absolute errors to make
results comparable across different test functions and domains.
As will be clear %
from the convergence
plots of Section~\ref{ssec:conv} that go below $10^{-15}$ in some
instances, we need very accurate reference values of the integrals
\[
I(f_1,0), \quad I(f_2,0), \quad I(0,g_1), \quad I(0,g_2),
\]
in order to reliably compute at least one significant digit
of the true relative errors. For every test domain,
except for the deco-tetrahedron $\Omega_7$,
parametrizations of all smooth pieces of $\partial\Omega$
are known in closed form, with the parametric domain given by
either an interval for 2D domains, or a rectangle for 3D domains.
The integrals of $f_1$ and $f_2$ over $\Omega$ were
computed by choosing vector fields $F_1$ and $F_2$
whose divergence is $f_1$ and $f_2$, then using
the divergence theorem to turn the integrals over $\Omega$
into integrals over $\partial\Omega$, and finally integrating
over the parametric domain of each smooth boundary
piece with MATLAB's adaptive quadrature routines
\texttt{integral} and \texttt{integral2}.
To meet the required accuracy target, absolute and
relative tolerances were set to $4 \cdot 10^{-16}$.

Note that even though the test functions $f_1$ and $f_2$
are both infinitely differentiable and of a simple shape,
we can distinguish their smoothness because the partial
derivatives of $f_1$ grow as the factorial of their order,
whereas those of $f_2$, as en entire function, grow much 
slower as an exponent of the order.
This makes $f_1$, that actually stems from the famous Runge
example for polynomial interpolation, a more difficult test
function for high order methods than $f_2$.

The stability of the quadrature formulas is assessed
by computing the \emph{normalized stability constants}
\[
K_{\muw} \deq \frac{1}{\abs{\Omega}}
	\sum_{i=1}^{N_Y} \abs{\muw_i},
\qquad
K_{\sigmav} \deq \frac{1}{\abs{\partial\Omega}}
	\sum_{i=1}^{N_Z} \abs{\sigmav_i},
\]
that measure the sensitivity of the quadrature to the maximum
absolute error in the function values, such that formulas with 
smaller $K_{\muw}$ and $K_{\sigmav}$ are more stable. Again, the
normalization helps to compare the stability across all domains. 
For any quadrature formula exact on
constants we have $K_{\muw} \ge 1$ or $K_{\sigmav} \ge 1$, and the
equality holds if and only if the formula is positive. Without
exactness for constants, but with the relative quadrature error
for the constant function less than some $\eps\in(0,1)$, we get 
$K_{\muw}> 1-\eps$ and $K_{\sigmav} > 1-\eps$. Therefore, the best
stability is attained when $K_{\muw}$ and $K_{\sigmav}$ are close
to one. 
Note that Algorithms~\ref{algo:mfd} and \ref{algo:spline} may produce
quadrature weights such that $K_{\muw} < 1$, because exactness
for constants is only enforced for $(Z,\sigmav)$ by the choice
$(\fhat,\ghat) \equiv (0,1)$. If $|\Omega|$ is known, in addition
to $|\partial\Omega|$, then one may also enforce exactness of
$(Y,\muw)$ for constants by adding one more row to the linear
system \eqref{eq:linear-system-fhat-ghat}, but this does not
provide any tangible benefit, as will be shown
in Section~\ref{sssec:choice-fghat}.

In what follows we first demonstrate in
Section~\ref{ssec:validation-of-parameters} 
the effectiveness of the settings suggested in 
Algorithms~\ref{algo:mfd} and \ref{algo:spline},
which serves as a kind of a tuning step for the free
parameters in the algorithms.
We do the testing mostly for Algorithm~\ref{algo:mfd},
because the results for Algorithm~\ref{algo:spline}
are very similar. The second batch of numerical experiments in
Section~\ref{ssec:conv} provides comprehensive tests of 
the numerical convergence order of the method on domains 
$\Omega_2,\Omega_3,\Omega_5,\Omega_6$. Surprisingly, 
this order is higher than predicted in
Section~\ref{ssec:on-the-choice-of-L-and-B} as it matches or
exceeds  the order $q$ of the numerical differentiation scheme.
Finally, in Section~\ref{ssec:comp} we compare the errors of the
interior integral for the weights $w$ produced by Algorithm~\ref{algo:spline}
on domains $\Omega_6$ and $\Omega_7$ to those by the numerical integration schemes
provided by state-of-the-art open source packaged Gmsh and MFEM.

In all tests the weights of the meshless finite difference
formulas are computed using the open source library mFDlab
\cite{mFDlab}, and those for the tensor-product spline
differentiation by MATLAB's Curve Fitting Toolbox.

\subsection{Validation of recommended parameters}
\label{ssec:validation-of-parameters}

\subsubsection{Choice of quadrature nodes}
\label{sssec:choice-of-nodes}

In the first test we demonstrate the robust performance of the
MFD method described in Algorithm~\ref{algo:mfd} with respect to
the choice of nodes in the sets $X$, $Y$, $Z$.
These experiments are performed on 
an elliptical domain $\Omega_1$ in $\R^2$
with semi-axes of length $1$ and $3/4$ centered at the origin:
\[
\Omega_1 = \left\{
(x_1,x_2) \in \R^2 \mathrel{\Big|}
x_1^2 + \frac{x_2^2}{(3/4)^2} < 1
\right\}.
\]

A domain in $\R^d$ with piecewise smooth boundary may be
discretized by scattered nodes using a wide range of techniques,
see the survey \cite{suchde-point}.
We consider here two important cases: a meshless advancing front
node generation algorithm, which we have released as the
open-source library NodeGenLib \cite{nodegenlib},
and a rejection sampling algorithm.

Advancing front methods for node generation,
as introduced in \cite{lohner1998advancing},
work by generating a set of boundary nodes $Z$ from
a parametric description of $\partial\Omega$,
possibly consisting of multiple patches,
and then placing interior nodes $Y_{int}$ in $\Omega$
by advancing the nodal front inside the domain starting from $Z$.
Node sets $Z$ and $Y_{int}$ are disjoint, and their nodes
are spaced according to locally varying \emph{spacing functions}
$h_Z \colon \partial\Omega \to \R_+$ and
$h_Y \colon \Omega \to \R_+$,
although we restrict ourselves to constant spacing
\[
h_Z \equiv h_Y \equiv h \in \R_+
\]
for all numerical experiments, which implies that the nodes
are quasi-uniform. Outward-pointing unit normals
$\nu \restrict{Z}$ are also computed and stored,
because they are required to assemble matrix $B$
in Algorithms~\ref{algo:mfd} and \ref{algo:spline}.
We do not provide here a description of our own
implementation of the advancing front method,
because it is still in early development, and is mostly motivated
by future research topics, such as automatic refinement
of quadrature nodes, or integration over CAD geometries.
For reference, however, we remark that the set $Z$ is generated
by NodeGenLib release 0.3 (the version used in our numerical
experiments) in essentially
the same way as in \cite[Algorithm~2.1]{duh2021fast},
and the set $Y_{int}$ as in
\cite[Algorithm~4.1]{slak2019generation}.

Rejection sampling is based on either a uniform Cartesian grid,
or the quasi-random Halton sequence,
or pseudo-random uniformly distributed samples, and
works in the same way regardless
of the node distribution: samples are first generated
in a bounding box~$\bbox$ containing $\Omega$, and then
are either accepted, projected to the boundary,
or discarded. More precisely, a set $Y_{\bbox}$ of
$N_{\bbox} = \round \bigl( \abs{\bbox} h^{-d} \bigr)$ nodes is 
initially generated in the bounding box $\bbox$, 
and then only nodes inside $\Omega$ are kept by using a
level set function $\varphi$ such that
\[
\Omega = \{ x \in \bbox \colon \varphi(x) < 0 \},
\qquad
Y_{int} = \{ y \in Y_{\bbox} \colon \varphi(y) < 0 \}.
\]
Nodes in a first-order approximation of a tubular
neighborhood of $\partial\Omega$ of width $2h$ are
projected to the zero set of $\varphi$
along the direction parallel to the gradient of $\varphi$,
and then a thinning operation is applied to the projected
points to ensure that no pair of nodes in $Z$
has distance smaller than $h$.

By the uniform density of the initial nodes used in the
rejection sampling process, the packing spacing $h_{ps}(Y)$
of the generated internal nodes, see \eqref{packsp}, ends up being very
close to the spacing parameter $h$. The same, however, is not true
for the advancing front method: for a fixed value of~$h$,
the two approaches produce a significantly different
number of internal nodes (about 15\% more for rejection
sampling in 2D, compared to advancing front), and this
is due to different definitions of $h$ in two
node placement techniques.
For the sake of notational simplicity, we have decided
to use the same variable $h$ in all cases. However,
whenever necessary, we tweak $h$ in our numerical
experiments to ensure that the sets $Y$ and $Z$ have
approximately the same size across all node generation methods.

Regardless of how the sets $Y_{int}$ and $Z$ are generated,
we produce \emph{open} quadrature formulas, for which
$Y \cap \partial\Omega = \emptyset$, as well as \emph{closed}
quadrature formulas, for which
$Y \cap \partial\Omega \neq \emptyset$.
In all of our numerical experiments, the former are obtained by
taking $Y = Y_{int}$, whereas the latter are obtained by
taking $Y = Y_{int} \cup Z$.
Note that the accuracy and stability of $\sigmav$
are also affected by the choice of $Y$, because
the weight vectors $\muw$ and $\sigmav$ are computed
simultaneously.

Unless stated otherwise, the set $X$ of discretization nodes
used to define the numerical differentiation scheme is generated
in the same way as the closed version of~$Y$,
but with a larger spacing $h_X = 1.6 \, h$.
Alternatively, the set $X$ is chosen as an unfitted
uniform Cartesian grid over a bounding box $\bbox$ around $\Omega$,
large enough to ensure that no set of influence $S_{L,i}$ or $S_{B,i}$
reaches the boundary of $\bbox$. The step size of the grid
is slightly larger than $1.6 \, h$, so that the number of
rows in the matrix $A$ stays approximately the same.

Node generation may be a stochastic process,
which applies in particular to the advancing front method
and rejection-sampling of random nodes,
or a deterministic process, whose starting conditions
are arbitrary and may therefore be randomized,
like rejection-sampling of Halton nodes.
Even in the case of Cartesian grids, some randomness
can be included in the node generation process
by a random shift of the entire grid in the range $[0,h)$
along each Cartesian dimension. 
In either case, node sets $X$, $Y$, $Z$ can be
understood as functions of a \emph{seed}, an integer
variable $i_s$ that initializes the pseudo-random number
generators.
A naive comparison of quadrature errors based on a
single choice of seed can lead to misleading results,
because all errors are subject to random noise, and 
so any particular node distribution may overperform
or underperform compared to the others, sometimes
by an order of magnitude. Numerical quadrature is much more
affected by this than, for example, the data fitting problem, where
the maximum or the root mean square error over an entire domain
is computed, which significantly reduces the influence of the
noise in the pointwise errors.

Therefore we use average errors in all parameter validation tests.
They are computed as follows.
Let $e(\cdot,i_s)$ be the relative quadrature
error for a given seed $i_s \in \N$.
We denote by $e_{rms}(\cdot)$ %
the root mean square values (RMS) of
$e(\cdot,i_s)$ for $i_s = 1, \dots,n_s$:
\begin{equation} \label{eq:err-rms}
e_{rms}(\cdot)
= \bigg(\frac{1}{n_s} \sum_{i_s = 1}^{n_s} e(\cdot,i_s)^2
\bigg)^{1/2},
\end{equation}
where we take $n_s=64$.
The stability constants $K_{\muw}$
and $K_{\sigmav}$ are obtained by averaging the values
\[
K_{\muw}(i_s)
= \frac{1}{\abs{\Omega}}
	\sum_{j=1}^{N_Y(i_s)} \abs{\muw_j(i_s)}
\quad \text{and} \quad
K_{\sigmav}(i_s)
= \frac{1}{\abs{\partial\Omega}}
	\sum_{j=1}^{N_Z(i_s)} \abs{\sigmav_j(i_s)}
\]
over all seeds $i_s = 1, \dots, n_s$.

Table~\ref{tab:choice-XYZ} reports the RMS
quadrature errors and average stability
constants for open and closed quadrature formulas
computed by Algorithm \ref{algo:mfd}
on $\Omega_1$ and $\partial\Omega_1$
using the four node generation methods described above.
Numerical differentiation is performed by meshless finite
difference formulas of polynomial order $q=5$, and
the Runge function $f_1$ is centered at $x_R = (0,0)$.
In all rows of Table~\ref{tab:choice-XYZ}, care was taken
to tweak the value of $h$ so that $N_Y \approx 2500$.
The values of $N_Y$ and $N_Z$ are also averages over 64 seeds,
but have been rounded to the nearest integer
to enhance readability.

We observe that nodes generated by the advancing
front algorithm lead to the most accurate and stable
quadrature formulas, although Cartesian nodes,
Halton nodes, and even random nodes remain competitive.
This suggests that our quadrature formulas are robust with
respect to the exact placement of the input
nodes in $Y$ and $Z$. On average, closed quadrature
formulas are more accurate and stable than
their open counterparts, and so they will be used in all
subsequent experiments. Note that the errors for the worst of 64
seeds are less than three times higher than the RMS errors in the
table in all cases except for the open quadrature on random nodes,
where they may get to almost 8 times higher. The worst stability
constant $K_w$ is 1.71 for the closed quadrature on advancing
front nodes and is significantly higher than the average only for
random nodes, with 8.29 for the closed and 233.21 for the open
formulas. The largest constant $K_v$ for the open formulas on
random nodes is 12.72 and never exceeds 1.07 in other cases.

Table~\ref{tab:choice-XYZ-unfitted} reports the errors
for the same numerical experiments as in Table~\ref{tab:choice-XYZ},
except that the set $X$ is chosen as an unfitted
uniform Cartesian grid with step size $2h$
when $Y$ and $Z$ are generated by the advancing front method,
and step size $1.8 \, h$ when rejection sampling is used.
Smaller step sizes, such as $1.6 \, h$, may lead
to overconstrained linear systems, because the
sets of influence $S_{L,i}$ and $S_{B,i}$ extend
outside of $\Omegabar$, and this increases significantly
the number of constraints; we have found $2h$ and $1.8 \, h$
to roughly equalize the size of $A$ across the two tables.
Errors are up to two orders of
magnitude larger, and the resulting quadrature formulas are
much less stable. This indicates a clear advantage of the
fitted approach, where $X$ is chosen in the same way as the
closed version of $Y$.

Summarizing the results of this section, in
those applications where the nodes are not provided by
the user  we recommend to generate the node sets by the advancing
front algorithm, and include the boundary nodes in both sets $X$
and $Y$.

\begin{table}[htbp!]
\caption{Comparison of RMS quadrature errors and average
stability constants for different node distributions.
MFD algorithm on ellipse $\Omega_1$ with polynomial
order $q=5$ and $n_s=64$ distinct seeds. The set $X$
is chosen in the same way as the closed version of $Y$,
but with a larger spacing $h_X = 1.6 \, h$.}
\label{tab:choice-XYZ}
\centering
\begin{small}
\begin{filecontents*}{test-choice-XYZ-fitted.csv}
Nsamples,hX,hY,hZ,NX,NY,NZ,NrowsA,sizeratioA,nnzA,residual,tnodegen,tmatassembly,tsolve,sreintoneavg,sreintonestd,sreintonerms,sreintonemax,srebndoneavg,srebndonestd,srebndonerms,srebndonemax,sreintrungeavg,sreintrungestd,sreintrungerms,sreintrungemax,srebndrungeavg,srebndrungestd,srebndrungerms,srebndrungemax,sreintfrankeavg,sreintfrankestd,sreintfrankerms,sreintfrankemax,srebndfrankeavg,srebndfrankestd,srebndfrankerms,srebndfrankemax,sreintexpsumavg,sreintexpsumstd,sreintexpsumrms,sreintexpsummax,srebndexpsumavg,srebndexpsumstd,srebndexpsumrms,srebndexpsummax,stabw,stabv,stabwmax,stabvmax,minw,minv,nnzw,nnzv,nnzwv
64,0.046632,0.029145,0.029145,1000,2507,189,2002,0.742581602373887,158114,2.91653648619747e-14,0.173221524739583,0.164697671875,0.0804608046875,-4.15288440910016e-11,3.10024006858224e-10,3.10383206125136e-10,9.6594227856902e-10,-9.29225382406999e-17,6.67215047226642e-16,6.6847187197408e-16,1.60730876956886e-15,1.27041845814844e-06,6.92062289259161e-06,6.98288082195894e-06,1.59495313521996e-05,-6.04167504181371e-10,3.06844517773205e-09,3.10374922430021e-09,8.14278149769092e-09,-1.41994210755183e-07,3.08767848439588e-07,3.37654098693331e-07,1.01456044004592e-06,2.31167699724501e-08,1.93623397760912e-07,1.93490629358199e-07,5.13749983326325e-07,3.27778834801397e-08,4.03582634528914e-09,3.30215541938697e-08,4.08205179349232e-08,2.27264587721623e-10,8.57465886748757e-10,8.8057293492601e-10,2.631553890204e-09,1.41305808612233,1.00155749376965,1.66876715798368,1.01595979255051,-0.00867595404016057,-0.0258349607188473,2507,189,2696
64,0.045008,0.02813,0.02813,1072,2492,196,2145,0.797991071428571,157544,3.04233791871121e-14,0.184272946614583,0.153625666666667,0.084592314453125,6.16826497019481e-11,2.75192311219825e-10,2.79914743361881e-10,6.20238319255699e-10,5.0228399049027e-17,8.60037150440857e-16,8.54768707762395e-16,1.92877052348264e-15,2.13344912655136e-06,8.04308189615604e-06,8.2602646500605e-06,2.36853287788633e-05,-2.59384024950729e-09,4.73548129602635e-09,5.36678703778746e-09,1.17782440115878e-08,1.25784182583036e-06,5.86161951546651e-07,1.38577897357469e-06,2.6917982446514e-06,3.32844776788249e-08,3.14878602830885e-07,3.14177012404147e-07,7.10753195945042e-07,4.35007802615557e-08,4.4095719342212e-09,4.37202286182801e-08,5.22777124906374e-08,3.28403921102502e-10,1.03016595421231e-09,1.0735497771562e-09,2.47209116884451e-09,2.3215431872223,1.00107127837671,2.71235887024306,1.00832353915268,-0.0161596626784889,-0.0174186997007906,2492,196,2688
64,0.05104,0.0319,0.0319,1019,2475,181,2040,0.768072289156627,155937,1.19715508700611e-13,0.0495946009114583,0.170989963541667,0.108265263020833,-1.77140206680314e-10,2.28146803319225e-10,2.87430548788987e-10,8.83636874028008e-10,1.20548157717665e-16,6.63019699406369e-16,6.68773733299908e-16,1.76803964652575e-15,-1.71772129759996e-06,1.11371021634714e-05,1.11824667888601e-05,3.39638200339031e-05,9.37489739850497e-09,4.27384486825173e-09,1.02892686254799e-08,2.03056239845321e-08,-2.21644114132092e-08,4.02395439171623e-07,3.99854120838474e-07,1.14111509001895e-06,-1.92908488690848e-07,2.01206274382273e-07,2.77605994007673e-07,5.54340020890235e-07,6.84875548009059e-09,2.3738038970319e-09,7.24239950243859e-09,1.37838597279282e-08,-2.16327670321008e-10,1.21236646418163e-09,1.22215546166989e-09,3.45891296770877e-09,1.52983004685802,1.00340216349775,2.33316705461838,1.01864064107908,-0.0145081616190456,-0.04138691590635,2475,181,2656
64,0.049184,0.03074,0.03074,1099,2473,188,2199,0.826381059751973,156100,5.094727157481e-14,0.0335320364583333,0.160934095052083,0.112992774088542,-4.87483616027913e-11,3.13613126808099e-10,3.14948928272343e-10,9.66741233905004e-10,1.50685197147081e-17,7.08883099766273e-16,7.03484547558019e-16,2.08950140043952e-15,2.13513553095892e-06,1.53210749562096e-05,1.53501270970639e-05,3.13022796160135e-05,4.385970757719e-09,2.6836982113021e-09,5.13092980248378e-09,1.07641006831552e-08,4.67457592154836e-08,4.50531407090849e-07,4.4943539575258e-07,1.24492272326269e-06,-1.06648995396518e-07,1.90580453247147e-07,2.17088468540711e-07,6.05694911486409e-07,3.83235663417701e-09,2.66179600130136e-09,4.65418196153266e-09,9.94285995515806e-09,-1.04042759445766e-11,9.6551348749518e-10,9.57997214970338e-10,2.20067448508525e-09,2.12099521318744,1.00530931697182,2.69087671927268,1.0479399734363,-0.0155369478255033,-0.0378044700842975,2473,188,2661
64,0.05104,0.0319,0.0319,1018,2499,183,2038,0.759880686055183,156653,1.28268671587434e-13,0.0412267083333333,0.175817272135417,0.101916341796875,2.35019911324007e-10,4.68531406451249e-10,5.20889634489104e-10,1.49986835405566e-09,-1.07991057955408e-16,8.29543512696858e-16,8.30091731340126e-16,2.41096315435329e-15,-2.95847091472472e-06,9.19192951293843e-06,9.58769719856815e-06,3.14118744006413e-05,-1.70779952682343e-09,4.93351964778237e-09,5.18419606468306e-09,1.55930989053744e-08,-1.35010824655009e-07,4.10446877776739e-07,4.29024790669346e-07,9.38315042148524e-07,4.93694174891181e-09,4.11417402521971e-07,4.08220403863958e-07,1.11247735903167e-06,3.75910335630809e-08,4.52265090974327e-09,3.78578997838136e-08,4.85536338206284e-08,6.01248495631298e-11,1.81748261953519e-09,1.80422972396605e-09,4.57395045164371e-09,1.96542570854947,1.00321594487828,2.24764209517999,1.02235241669679,-0.0147208500109547,-0.049800123889435,2499,183,2682
64,0.049184,0.03074,0.03074,1093,2494,191,2188,0.814897579143389,156605,7.76313022065742e-14,0.0401910501302083,0.204232576822917,0.132752794921875,2.84487222587753e-10,4.13131985609169e-10,4.9894305588174e-10,1.13219981394746e-09,-1.00456798098054e-16,8.24309478557903e-16,8.23990724556493e-16,1.92877052348264e-15,-1.06409898092952e-06,9.80240927727352e-06,9.78356639209501e-06,2.2837773697198e-05,-2.63389012370453e-09,4.46618541602005e-09,5.15485406329393e-09,1.41791304173682e-08,-3.59347437542812e-09,3.82679491568397e-07,3.79695042343637e-07,1.12886019626749e-06,-5.48262298995571e-08,3.4024273925969e-07,3.41997378759897e-07,9.49394157317314e-07,3.30545147259625e-08,3.65859959130369e-09,3.32532276422825e-08,3.93131551764423e-08,5.18186639333636e-12,1.53874010594569e-09,1.52668016394541e-09,3.31600512479294e-09,2.45790553871754,1.00401510833523,2.99684546530374,1.02993486727547,-0.0213333631699491,-0.0481250628561359,2494,191,2684
64,0.05104,0.0319,0.0319,1014,2488,175,2029,0.761922643634998,156467,4.78426513693759e-13,0.0362682194010417,0.1678225546875,0.0957070377604167,1.79786498854285e-10,1.08141081866532e-09,1.08788781483505e-09,3.06922697461046e-09,-1.13013897860311e-16,8.71313422470697e-16,8.71835376180277e-16,2.25023227739641e-15,1.95672181249775e-07,6.17336280717992e-05,6.12497478318427e-05,0.000391689594659849,-7.58127946974133e-10,1.51990005520768e-08,1.50988360640206e-08,3.1368419042866e-08,2.00835230316817e-07,1.67798026322344e-06,1.67688953898131e-06,4.41541105834757e-06,-5.56177530767216e-08,1.21194962845335e-06,1.20372957683535e-06,2.7396487088421e-06,2.52764080796821e-08,1.08545171535918e-08,2.7475014141389e-08,4.64026849278928e-08,2.78128619360521e-10,5.05753367098731e-09,5.02556822710721e-09,1.64185956046485e-08,4.98660187527716,1.00397572846574,7.08655671361599,1.03191083604524,-0.0797746109768118,-0.0500726809105858,2488,175,2663
64,0.049184,0.03074,0.03074,1091,2494,182,2183,0.81576980568012,157060,5.80705777928245e-13,0.03467853125,0.160702430338542,0.090543798828125,-7.46831345678918e-11,1.95954899924982e-09,1.94561364912658e-09,1.03575915740637e-08,1.10502477907859e-16,8.2250416922176e-16,8.23500691373362e-16,1.92877052348264e-15,2.55098676249024e-07,6.98033657107269e-05,6.92563496962306e-05,0.000338696467077312,6.1830926603041e-10,2.37221316335165e-08,2.35441931363553e-08,7.95991424483134e-08,7.05767216353094e-07,1.78189681708812e-06,1.90358919248866e-06,5.9591189698164e-06,-1.2166656775212e-07,1.68759261412587e-06,1.67877100045612e-06,4.62755357346572e-06,2.34180931111247e-08,1.48102320193014e-08,2.76463889772619e-08,5.93412837069442e-08,9.93492979326495e-10,7.21805302697035e-09,7.23002425837472e-09,3.19902427010553e-08,9.99698524079345,1.10936463565817,25.3044672014128,2.18269549334709,-0.517509166250109,-0.470867454399239,2494,182,2675
\end{filecontents*}
\pgfplotstabletypeset[
	create on use/method/.style={
        create col/set list={
        	Adv.~front closed, Adv.~front open,
        	Cartesian grid closed, Cartesian grid open,
        	Halton closed, Halton open,
        	Random closed, Random open
        }
    },
	columns/method/.style={
		column name = {Node distribution},
		column type = {l},
		string type
	},
	columns/stabv/.append style={
		precision = 3
	},
	columns={
		method,hY,NY,NZ,sreintrungerms,srebndrungerms,
		sreintfrankerms,srebndfrankerms,stabw,stabv
	}]
{test-choice-XYZ-fitted.csv}
\end{small}
\end{table}

\begin{table}[htbp!]
\caption{Comparison of RMS quadrature errors and average
stability constants for different node distributions.
MFD algorithm on ellipse $\Omega_1$ with polynomial
order $q=5$ and $n_s=64$ distinct seeds. The set $X$
is chosen as an unfitted uniform Cartesian grid
with step size about 0.057 over the bounding box
$\bbox = [-2,2] \times [-2,2]$ around $\Omega_1$.}
\label{tab:choice-XYZ-unfitted}
\centering
\begin{small}
\begin{filecontents*}{test-choice-XYZ-unfitted.csv}
Nsamples,hX,hY,hZ,NX,NY,NZ,NrowsA,sizeratioA,nnzA,residual,tnodegen,tmatassembly,tsolve,sreintoneavg,sreintonestd,sreintonerms,sreintonemax,srebndoneavg,srebndonestd,srebndonerms,srebndonemax,sreintrungeavg,sreintrungestd,sreintrungerms,sreintrungemax,srebndrungeavg,srebndrungestd,srebndrungerms,srebndrungemax,sreintfrankeavg,sreintfrankestd,sreintfrankerms,sreintfrankemax,srebndfrankeavg,srebndfrankestd,srebndfrankerms,srebndfrankemax,sreintexpsumavg,sreintexpsumstd,sreintexpsumrms,sreintexpsummax,srebndexpsumavg,srebndexpsumstd,srebndexpsumrms,srebndexpsummax,stabw,stabv,stabwmax,stabvmax,minw,minv,nnzw,nnzv,nnzwv
64,0.05829,0.029145,0.029145,5041,2507,189,2056,0.762611275964392,158152,3.98243771949001e-14,0.157951335286458,0.167761041015625,0.060275515625,-2.1138857557309e-10,1.77520720227335e-09,1.77392382343185e-09,4.07058376361134e-09,2.76256194769649e-17,7.8457634169731e-16,7.78912758427338e-16,2.25023227739641e-15,-5.1950679333779e-05,0.000167091790506767,0.000173730523367781,0.000600053937152852,2.10500538724566e-09,2.0001605454083e-08,1.99560583525921e-08,7.54441216829174e-08,8.22920878635362e-07,4.00719988085467e-06,4.06004296313914e-06,1.02184948258685e-05,4.10054888048751e-08,1.56481753463364e-06,1.55308568641934e-06,5.93519303088006e-06,-7.31853195527829e-08,1.79779179239672e-08,7.53275941768525e-08,1.24994049509198e-07,-2.6769190676981e-10,6.96516922872727e-09,6.91572242191635e-09,1.75344801086565e-08,6.71508753375478,1.188066170045,11.0622793876106,1.38257088150786,-0.158029701648953,-0.149951787433718,2481,189,2670
64,0.05626,0.02813,0.02813,5329,2492,196,2044,0.760416666666667,157582,1.06984686833429e-13,0.166308789713542,0.162710572916667,0.0586947389322917,-2.65245693120124e-09,6.0035760319051e-09,6.52037439392464e-09,1.9174268581484e-08,-1.10502477907859e-16,7.14160106283324e-16,7.17123643285039e-16,1.60730876956886e-15,-0.000194046045815918,0.0003674463184412,0.000412990347805684,0.00139584848423954,-4.47065664526527e-08,6.11226591241053e-08,7.53411034120785e-08,2.07865041092587e-07,4.44856697799847e-06,8.77401420740851e-06,9.77600168839735e-06,2.69875438269161e-05,-1.20230606777438e-07,4.10466932077349e-06,4.07424972041677e-06,9.73173938281646e-06,-4.70708328913338e-08,3.83540432418931e-08,6.05285970198735e-08,1.75994003535133e-07,4.6658185022985e-09,1.85472565758794e-08,1.89840874796884e-08,5.50295050209822e-08,22.0116879011872,2.41548514449957,32.1024080998862,3.17087172318012,-0.378139365244678,-0.348768987167,2478,195,2673
64,0.05742,0.0319,0.0319,5184,2475,181,2106,0.792921686746988,155937,3.80886352105651e-14,0.0343768274739583,0.169468567057292,0.0617799348958333,-9.78347216350692e-10,1.38075506541733e-09,1.683406963664e-09,4.15238254113109e-09,3.21461753913773e-16,1.03929042187069e-15,1.08008578068757e-15,3.37534841609461e-15,3.15523789504319e-06,8.74897033110979e-05,8.6860825224841e-05,0.000248961162058073,1.02250030528048e-08,2.97566995710526e-08,3.12438238060016e-08,1.32878383555007e-07,-1.81438006321066e-07,2.51120852091223e-06,2.49811010705298e-06,8.46828406072325e-06,-7.28723176301842e-08,1.56059762523912e-06,1.55007134763784e-06,6.20922305946717e-06,-7.5910466461172e-08,1.40012880508184e-08,7.71710562444661e-08,1.18835676558472e-07,-8.95979646793832e-10,5.00202841456354e-09,5.04302744227174e-09,2.08250542810687e-08,6.20980328809473,1.11455740443928,18.2522263338539,1.55779791937361,-0.386671435023679,-0.297965449581211,2454,181,2635
64,0.055332,0.03074,0.03074,5476,2473,188,2162,0.812476512589252,156100,1.76544510331774e-13,0.0351822727864583,0.185969743489583,0.063733763671875,-1.26560288797185e-09,8.55258682609066e-09,8.57936910577979e-09,2.14279870114913e-08,-7.28311786210891e-17,1.08027856179077e-15,1.07427731474288e-15,3.05388666218084e-15,-2.96049757085544e-05,0.000559337849790499,0.00055573992620016,0.00182499908604652,2.94809289284681e-08,9.7699057751175e-08,1.01316774408251e-07,3.79208153903441e-07,5.13662770716763e-06,1.52211185142791e-05,1.59514060435082e-05,6.35908288910214e-05,-2.09151703480013e-06,1.06199321761454e-05,1.07422143523464e-05,4.26627152294903e-05,-5.60803395479216e-08,6.17259577238849e-08,8.30395425338222e-08,2.07073177711964e-07,-2.13336035374296e-09,1.91870741587796e-08,1.91557509397982e-08,5.28939053515616e-08,35.0329006829029,3.44753517691348,62.4462708344863,5.558996764207,-0.829199303781819,-1.57934058372906,2427,187,2614
64,0.05742,0.0319,0.0319,5184,2499,183,2105,0.784862043251305,157468,3.53655658105533e-14,0.03867807421875,0.185193939453125,0.0690932916666667,-3.16213535306365e-10,1.70073826621261e-09,1.71677202073021e-09,4.37109821353311e-09,-1.15525317812762e-16,7.77179254110773e-16,7.79689728953112e-16,1.92877052348264e-15,-1.09419224380915e-05,8.26122189071245e-05,8.26913974486646e-05,0.000224094884760883,-1.47842855089594e-09,1.24423897300982e-08,1.24330148829216e-08,3.07651814230295e-08,-2.73946265669214e-07,2.7903514097138e-06,2.78198676186677e-06,6.67258923958272e-06,-1.05056166385414e-07,9.34991538373333e-07,9.33587947939222e-07,3.26863367452034e-06,-7.67463639784288e-08,1.18249618878758e-08,7.7637937041933e-08,1.04151033654233e-07,-6.86451612201199e-10,4.62106283616432e-09,4.63592254433259e-09,1.40331486425019e-08,5.68432793602757,1.08968401427846,8.07139714905174,1.25148279419833,-0.119783764535275,-0.193629037700883,2482,183,2665
64,0.055332,0.03074,0.03074,5476,2494,191,2162,0.805214152700186,157440,1.36719023453269e-13,0.04051081640625,0.1895079140625,0.0688595123697917,-6.973525225742e-10,6.18762477752688e-09,6.17857355986405e-09,1.86741387502073e-08,-1.33105257479921e-16,6.73180722936741e-16,6.81034858386736e-16,1.92877052348264e-15,-0.000145446567929712,0.000380585998224731,0.000404644525568799,0.00103181888307111,-2.09646033412718e-08,6.66886645014659e-08,6.94075088547197e-08,1.87827337815644e-07,6.25353163403202e-08,1.26962825423976e-05,1.25968575431429e-05,3.47129677820964e-05,-5.00758863850092e-07,4.75631035135128e-06,4.74550010226754e-06,1.35730306800904e-05,-5.00243080858346e-08,4.21057784366943e-08,6.51738173517479e-08,1.38141284183284e-07,-1.06204702187046e-09,1.8509161188398e-08,1.83946741721987e-08,5.48237698146218e-08,27.1337361896229,2.74568597702668,39.9693692590308,4.00546315035939,-0.658173176883973,-0.566234486775211,2450,190,2640
64,0.05742,0.0319,0.0319,5184,2488,175,2105,0.790461885092002,156467,1.11352296948288e-13,0.037032259765625,0.183902022786458,0.06835773828125,-1.0339247825067e-09,7.89562049868013e-09,7.90162946200084e-09,2.75116347733051e-08,-1.25570997622567e-17,8.63572397184249e-16,8.56891188163395e-16,2.25023227739641e-15,-3.63441065240692e-05,0.000500261578461486,0.000497666758716348,0.00327642893081617,8.87765656329983e-09,6.14287560216578e-08,6.15901293949546e-08,2.84393426410192e-07,-4.2054018579274e-07,9.40186734294523e-06,9.33760086937392e-06,3.68137775832469e-05,5.89741091292089e-08,4.11875007205272e-06,4.08687117553194e-06,1.43863451985722e-05,-6.73893576097474e-08,4.36519873520098e-08,8.01064800364982e-08,1.81892919182617e-07,-1.21049012669042e-09,1.71825596453989e-08,1.70907141966774e-08,5.36281482971038e-08,21.8604164353462,2.00606622143055,66.7239625824579,4.71648227788686,-1.35022423517376,-1.26826417317829,2467,175,2642
64,0.055332,0.03074,0.03074,5476,2494,182,2158,0.806427503736921,157060,6.16765215634325e-13,0.0387073424479167,0.2017358515625,0.0698658255208333,2.37415057476464e-09,4.26253679048768e-08,4.23576341867582e-08,1.75224591084446e-07,-2.53653415197586e-16,1.58205801478747e-15,1.5900124799938e-15,7.71508209393054e-15,-0.000335066133216962,0.00251951469345384,0.00252210959681955,0.00975545116891974,1.11874667578494e-08,3.73794348129526e-07,3.71031285071014e-07,1.3418420442936e-06,-1.0950445829964e-06,5.52101841298009e-05,5.47881007613952e-05,0.000168713138282765,-9.82251494051163e-08,2.98218481354085e-05,2.95881107179351e-05,0.000112791637487146,-1.17764910364089e-08,2.09329183244161e-07,2.08020973468025e-07,9.36878512642966e-07,1.81171632190668e-08,8.69534721728901e-08,8.81532685992541e-08,4.14154405881608e-07,128.643390778976,7.0017967743469,468.40158541038,23.6316207651965,-10.1252265306174,-5.85874856197149,2454,180,2634
\end{filecontents*}
\pgfplotstabletypeset[
	create on use/method/.style={
        create col/set list={
        	Adv.~front closed, Adv.~front open,
        	Cartesian grid closed, Cartesian grid open,
        	Halton closed, Halton open,
        	Random closed, Random open
        }
    },
	columns/method/.style={
		column name = {Node distribution},
		column type = {l},
		string type
	},
	columns/stabv/.append style={
		precision = 3
	},
	columns={
		method,hY,NY,NZ,sreintrungerms,srebndrungerms,
		sreintfrankerms,srebndfrankerms,stabw,stabv
	}]
{test-choice-XYZ-unfitted.csv}
\end{small}
\end{table}

\subsubsection{Choice of spacing parameters \texorpdfstring{$h_X$}{hX} and \texorpdfstring{$h_{\Scal}$}{hS}}
\label{sssec:spacing}

In order to keep the linear system \eqref{eq:ls} underdetermined
and solvable, the spacing parameters $h_X$ and $h_{\Scal}$ must
be sufficiently large, because they determine the number of rows of
the system matrix. Nevertheless, as seen by the 
considerations in Section \ref{ssec:on-the-choice-of-L-and-B},
they should remain comparable to $h$ 
to ensure that the numerical differentiation scheme
is asymptotically accurate, in the sense that
\[
\varepsilon_L(u) = \Ocal(h^{q-k})
\quad \text{and} \quad
\varepsilon_B(u) = \Ocal(h^{q-k}).
\]
Therefore, we suggest to choose $h_X$ and $h_{\Scal}$ to be
directly proportional to $h$, and investigate how the errors
of the quadrature formulas depend on the
ratios $h_X/h$ and $h_{\Scal}/h$.

In Algorithms~\ref{algo:mfd} (MFD) and \ref{algo:spline} (BSP)
we recommend to choose the spacing parameters as
$h_X = 1.6 \, h$ and $h_{\Scal} = 4 \, h$
based on numerical evidence that we have gathered in
our numerical experiments on several 2D and 3D domains.
The case of the ellipse $\Omega_1$ is presented in
Figure~\ref{fig:test-choice-hX-hS-2D}, where quadrature errors
are plotted as a function of the ratios $h_X/h$ and $h_{\Scal}/h$
for $h = 0.025$ and $q = 5$.
Figure~\ref{fig:test-choice-hX-hS-3D} shows a similar test
on the ellipsoid $\Omega_4$ in $\R^3$ with semi-axes of
length 1, 0.7, 0.7,
\[
\Omega_4 = \left\{
(x_1,x_2,x_3) \in \R^3 \mathrel{\Big|}
x_1^2 + \frac{x_2^2}{0.7^2} + \frac{x_3^2}{0.7^2} < 1
\right\},
\]
and with $h = 0.05$ and $q = 5$.

In both cases, we use node sets $X,Y,Z$ generated
by the advancing front algorithm, with $Z \subset Y$.
As in the previous section, data points are the RMS of 64 quadrature
errors given by different seeds. The Runge function $f_1$ is
centered at the origin.
To aid visualization, instead of plotting the absolute values
of the errors, we plot their relative sizes compared
to the choices $h_X = 1.6 \, h$ and $h_{\Scal} = 4 \, h$.
The vertical axes are in logarithmic scale.
In addition, we plot the stability constant $K_w$.
In the BSP case, the size of the bounding box $\bbox$
has been chosen depending on $h_{\Scal}$ to avoid
boundary effects, and to ensure that the length of each side
of $\bbox$ is a multiple of $h_{\Scal}$.

\begin{figure}[htbp!]
\centering
\subcaptionbox{MFD in 2D: Errors and stability constant $K_w$
as functions of the ratio $h_X/h$; errors are relative to 
their values for $h_X/h = 1.6$.}{
	\includegraphics[width=0.47\textwidth]{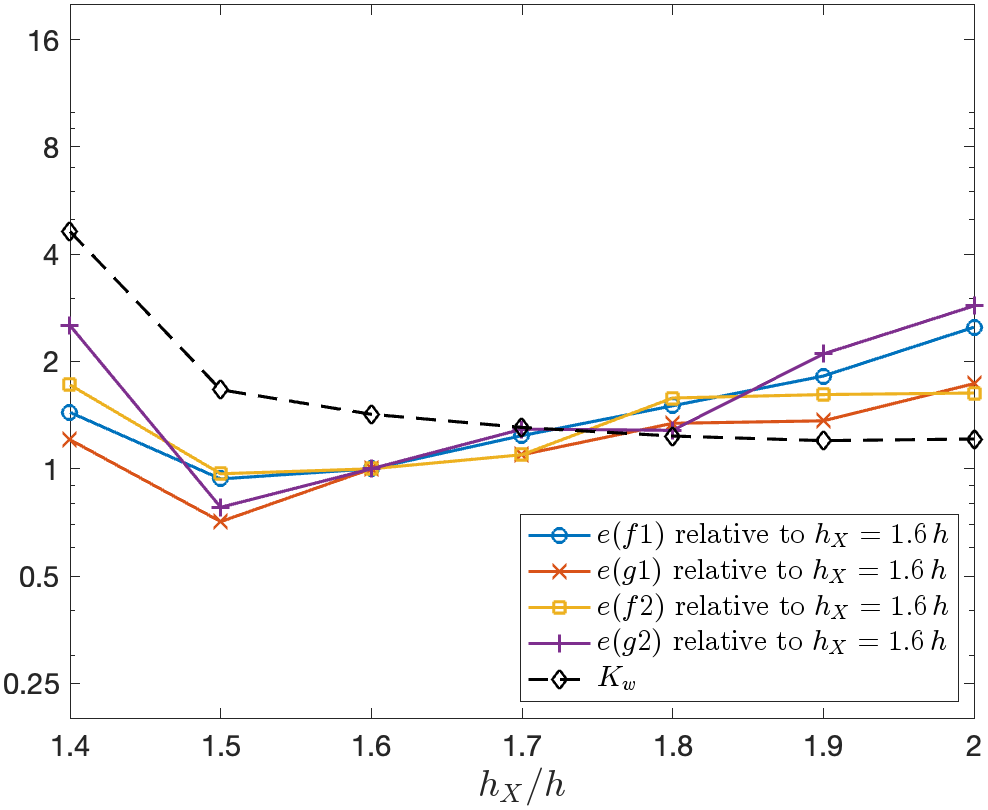}
} \hfill
\subcaptionbox{BSP in 2D: Errors and stability constant $K_w$
as functions of the ratio $h_{\Scal}/h$; errors are relative to 
their values for $h_{\Scal}/h = 4$.}{
	\includegraphics[width=0.47\textwidth]{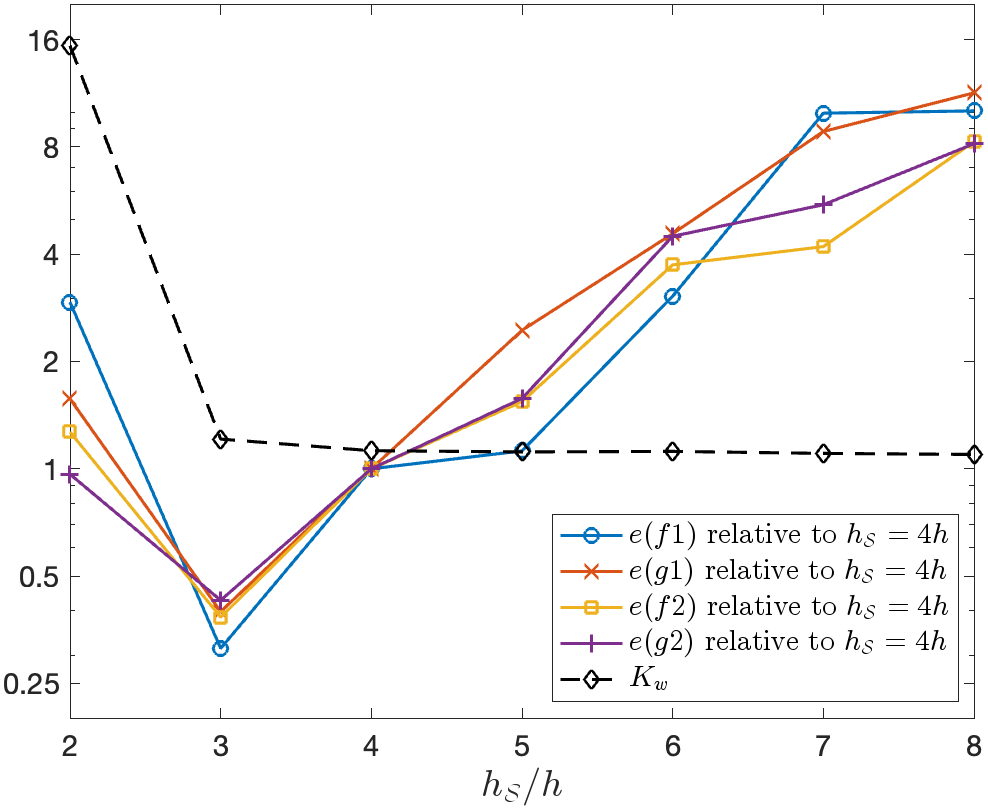}
}
\caption{RMS quadrature errors on ellipse $\Omega_1$ with
$h = 0.025$, $q=5$ and 64 different seeds. 
The dashed lines show the average stability constant.
On the left:~MFD algorithm.
On the right:~BSP algorithm.
}
\label{fig:test-choice-hX-hS-2D}
\end{figure}

\begin{figure}[htbp!]
\centering
\subcaptionbox{MFD in 3D: Errors and stability constant $K_w$
as functions of the ratio $h_X/h$; errors are relative to 
their values for $h_X/h = 1.6$.}{
	\includegraphics[width=0.47\textwidth]{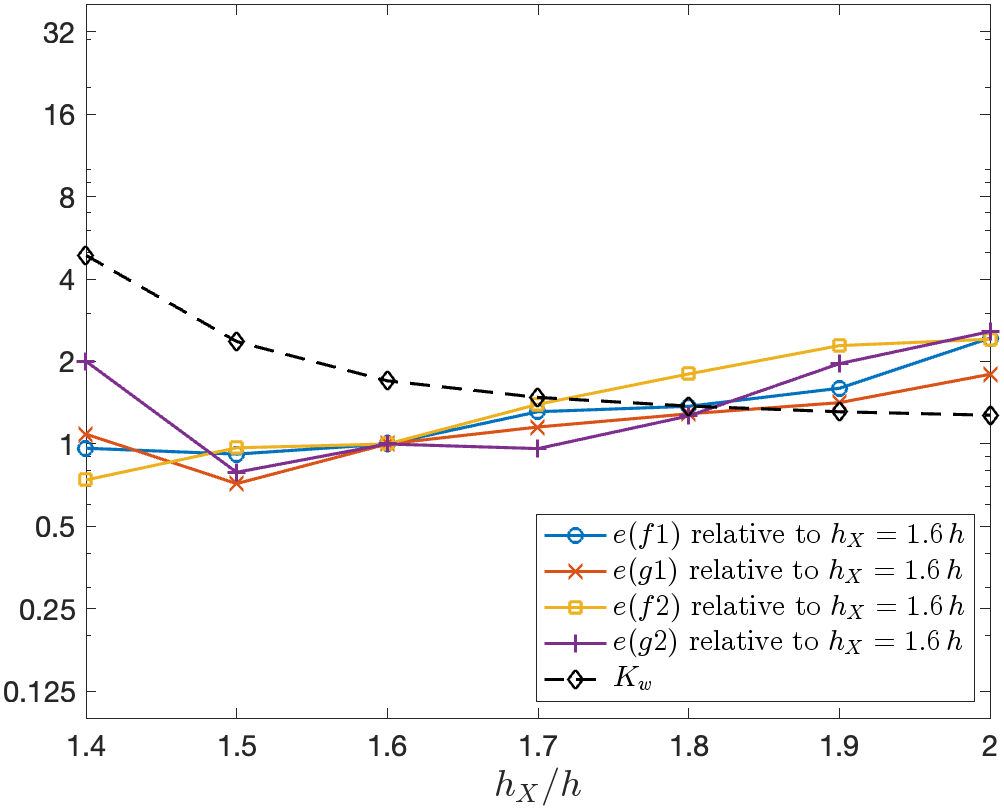}
} \hfill
\subcaptionbox{BSP in 3D: Errors and stability constant $K_w$
as functions of the ratio $h_{\Scal}/h$; errors are relative to 
their values for $h_{\Scal}/h = 4$.}{
	\includegraphics[width=0.47\textwidth]{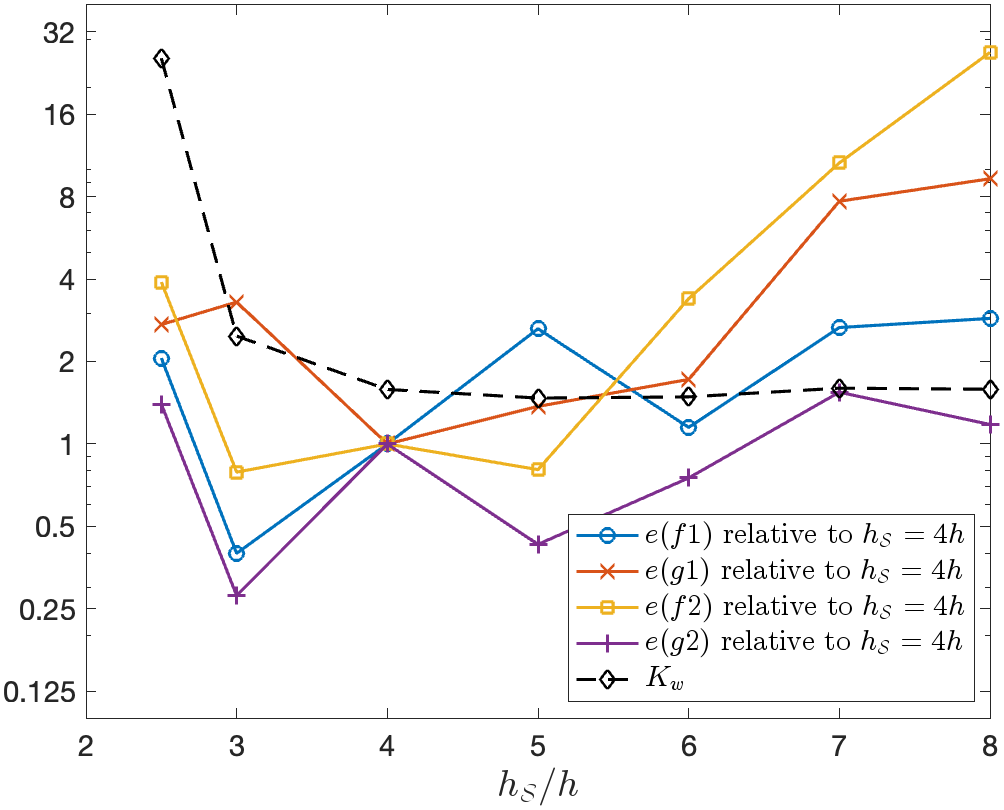}
}
\caption{RMS quadrature errors on ellipsoid $\Omega_4$ with
$h = 0.05$, $q=5$ and 64 different seeds. 
The dashed lines show the average stability constant $K_w$.
On the left:~MFD algorithm.
On the right:~BSP algorithm.
}
\label{fig:test-choice-hX-hS-3D}
\end{figure}

The two figures clearly show that quadrature errors
decrease as the ratios $h_X/h$ and $h_{\Scal}/h$ get smaller,
but only up to a point: when the ratios get too small,
the errors and the stability constant $K_w$
quickly rise up again.
Reducing the ratios even further leads to overdetermined systems
for which no exact solution exists, and so their errors are not
included in the plots. The stability constant $K_v$
is not included in the figures because its
dependence on $h_X/h$ and $h_S/h$ is very weak.
The ratios $h_X = 1.6 \, h$ and $h_{\Scal} = 4 \, h$
were found to be a good compromise between accuracy and stability,
as they are close to the minimum of the curves in
the two figures, but leaning on the side of stability,
so that a safety margin is left
in Algorithms~\ref{algo:mfd} and \ref{algo:spline}.
This choice of the ratios was also validated by running the convergence
tests of Section~\ref{ssec:conv} with slightly different
values of $h_X/h$ and $h_{\Scal}/h$, and comparing the results.

Even though increasing the ratios %
reduces the number of rows of the system matrix $A$, it does not
change the total number of nonzeros in $A$. Indeed, the number of
nonzeros in every column of the horizontal concatenation
$\big(L^T\,-\!B^T\big)$ depends in the MFD case only on
the size $n_L$ or $n_B$ of the
set of influence of the corresponding node in $Y$ or $Z$, as chosen
in Step~4 of Algorithm~\ref{algo:mfd}. Likewise, in the BSP case
the number of nonzeros in every column in general equals the
number $q^d$ of B-splines of degree $q-1$ that include the
corresponding quadrature node in the interior of their support.
Taking into account the last row of $A$ corresponding to the
non-homogeneous constraint, we get the following upper bound for
the number of nonzeros:
\begin{equation}\label{nnzA}
\nnz(A)\le N_Z+
\begin{cases}
n_LN_Y+n_BN_Z, & \text{for Algorithm~\ref{algo:mfd},}\\
q^d(N_Y+N_Z), & \text{for Algorithm~\ref{algo:spline}.}
\end{cases}
\end{equation}
Note that we do not test how the performance of the MFD
quadrature depends on the parameters $n_L$ and $n_B$, leaving
this to future work, where also more sophisticated methods for
the selection of the sets of influence could be investigated. As
mentioned in Section~\ref{sssec:mfd}, we rely on the safe and
simple but possibly not optimal selection of $K$ nearest
neighbors in $X$ to a node in $Y$ or $Z$, with $K$ being twice
the dimension of the corresponding polynomial space. As a result,
for a fixed polynomial order $q$, $\nnz(A)$ for BSP is about
twice as large compared to MFD in 3D, but somewhat smaller than
it in 2D, compare numerical results in Table~\ref{tab:nodes}.

\subsubsection{Choice of operators \texorpdfstring{$\Lcal$}{Lcal} 
and \texorpdfstring{$\Bcal$}{Bcal}}
\label{sssec:choice-LBcal}

In Section~\ref{ssec:choice-of-operators-Lcal-and-Bcal} we
discussed two possible ways to choose operators
$\Lcal$ and $\Bcal$ for our scheme: the \emph{elliptic approach}
with $\Lcal = \Delta$ and $\Bcal = \partial_\nu$, and the
\emph{divergence approach}, with $\Lcal = \diver$
and $\Bcal = \gamma_\nu$. Theoretical considerations indicate
that the elliptic approach should be appropriate for
domains with smooth boundary.
However, the theory of the elliptic approach breaks down when
the boundary is not smooth. At the same time,
the divergence approach looks promising for piecewise smooth
domains, even though the theory to justify it is far from complete 
because of the lack of results on the regularity of the
corresponding boundary  value problem \eqref{eq:bvp-divergence}.

In this section we compare the numerical performance of both
approaches on two 2D domains, the ellipse $\Omega_1$ with a
smooth boundary, and the disk sector $\Omega_2$ defined in the
polar coordinates $(r,\theta)$ centered
at the origin by
\[
\Omega_2 = \left\{
(r,\theta) \in \R^2 \mathrel{\Big|}
0 < r < 1 \text{ and } 0 < \theta < 3\pi/2
\right\}.
\]
The domain $\Omega_2$ has non-smooth boundary, with a reentrant
corner at $(0,0)$, and serves as a standard benchmark for
testing numerical methods for the elliptic boundary value problem
\eqref{eq:bvp-laplacian} that take into account the reduced smoothness
of the solution at reentrant corners.

We use numerical differentiation schemes $(\Psi,L,B)$
based on meshless finite difference formulas for both approaches,
the same sets of quadrature nodes $Y$ and $Z$ generated
by the advancing front algorithm with $Z \subset Y$,
and the same non-homogeneous constraints
with $(\fhat,\ghat) \equiv (0,1)$.
In the divergence approach we generate $X$ with
the usual advancing front method and spacing parameter
$h_X = 1.6 \, h$, whereas in the elliptic approach
we pick $X = Y$. We have found this simple choice to
be more accurate and reliable compared to the generation
of a coarser set of nodes $X$. (Figure~\ref{fig:test-choice-hX-hS-2D}
is specific to the divergence approach.)
The Runge function $f_1$ is centered at $x_R = (0,0)$
in the case of the ellipse $\Omega_1$, and at
$x_R = (\cos(3\pi/4)/2,\sin(3\pi/4)/2)$
in the case of the disk sector $\Omega_2$.
We have chosen the polynomial order as $q = 4+k$, where
$k$ is the order of the differential operator $\Lcal$, so that
the errors are expected to decay as $h^4$ as $h\to0$, unless the
convergence order is saturated by the low regularity of the
solution $u$ of the boundary value problem
\eqref{eq:bvp-joint-error-estimate}.

Figure~\ref{fig:choice-LBcal} presents relative quadrature errors
as a function of spacing parameter $h$ for both the elliptic  and
the divergence approach for the ellipse $\Omega_1$ (top row) and
for the disk sector $\Omega_2$ (bottom row). Relative errors are
the RMS of 8 outcomes given by distinct seeds.

\begin{figure}[h!]
\centering
\subcaptionbox{Errors on ellipse $\Omega_1$
for $\Lcal = \Delta$ and $\Bcal = \partial_\nu$.}{
	\includegraphics[width=0.45\textwidth]{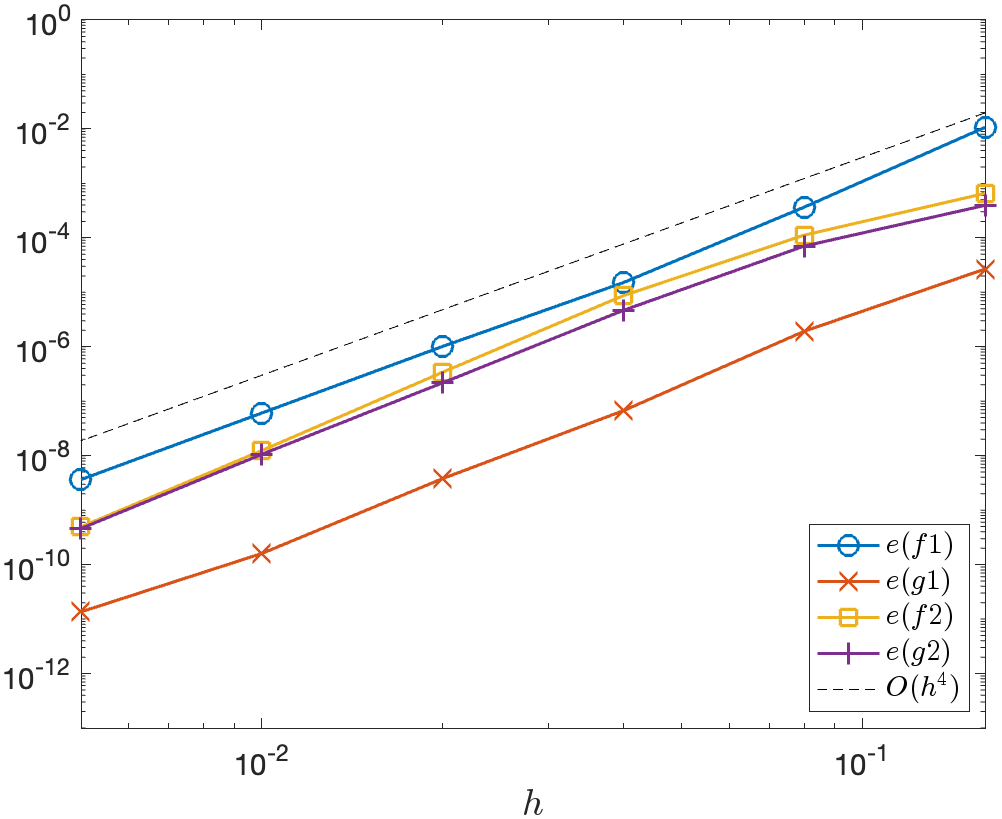}
} \hfill
\subcaptionbox{Errors on ellipse $\Omega_1$
for $\Lcal = \diver$ and $\Bcal = \gamma_\nu$.}{
	\includegraphics[width=0.45\textwidth]{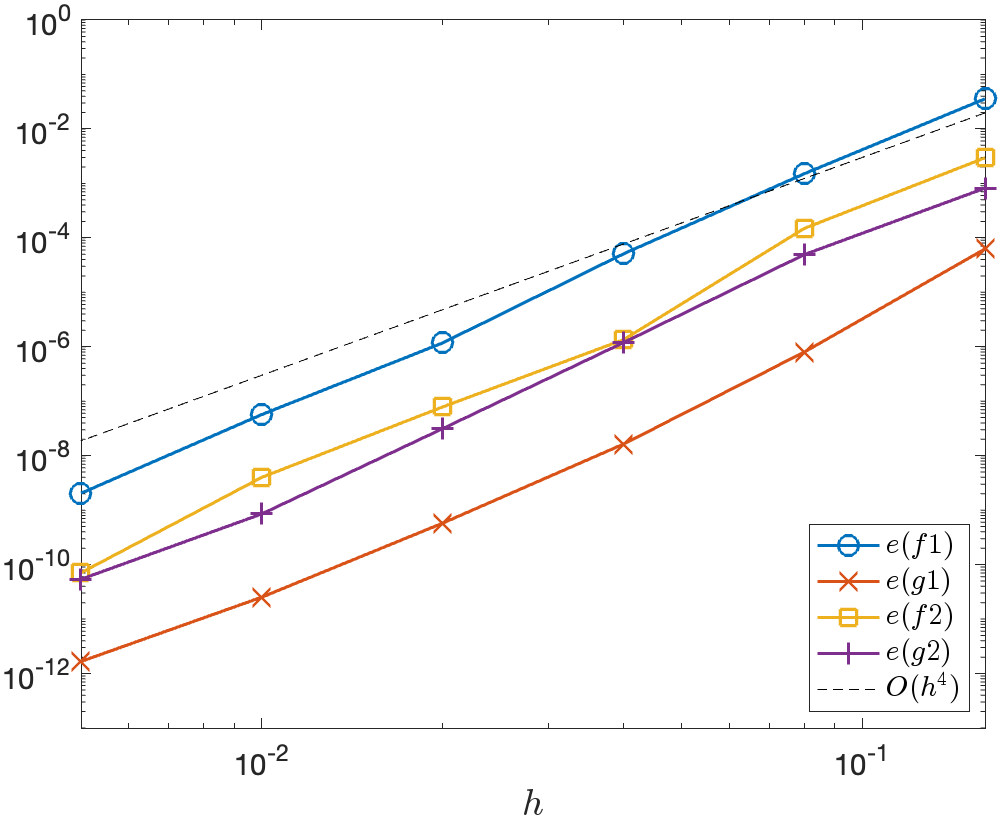}
} \\[2em]
\subcaptionbox{Errors on disk sector $\Omega_2$
for $\Lcal = \Delta$ and $\Bcal = \partial_\nu$.}{
	\includegraphics[width=0.45\textwidth]{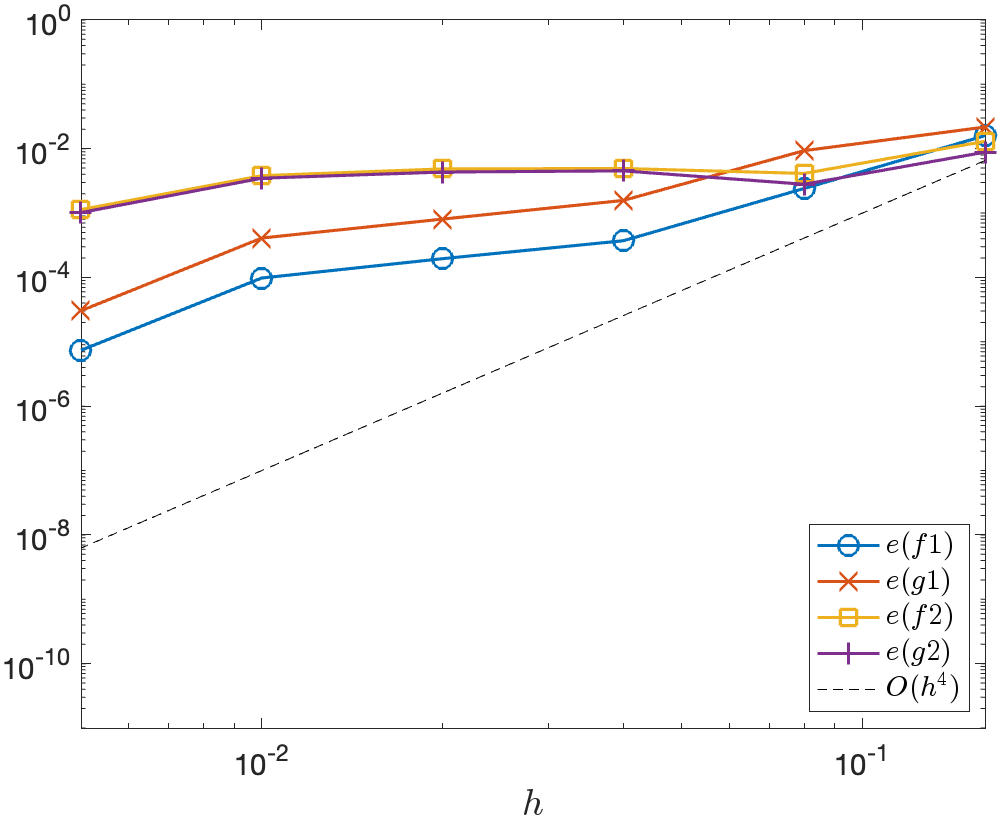}
} \hfill
\subcaptionbox{Errors on disk sector $\Omega_2$
for $\Lcal = \diver$ and $\Bcal = \gamma_\nu$.}{
	\includegraphics[width=0.45\textwidth]{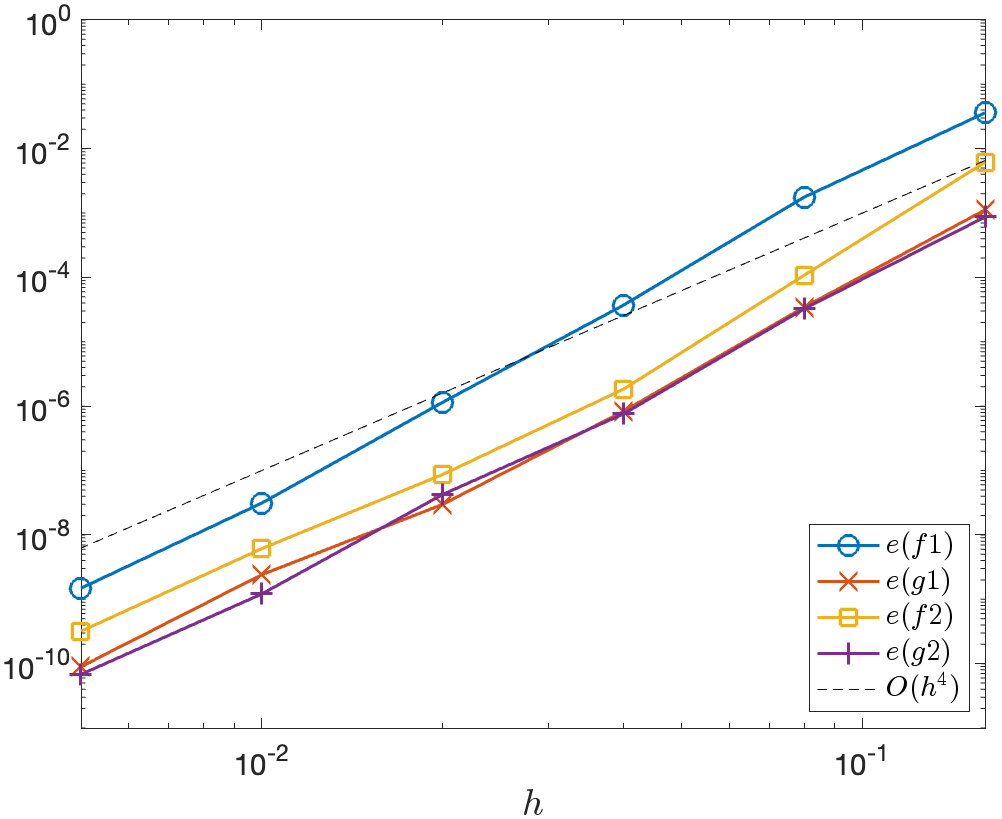}
}
\caption{Comparison of RMS quadrature errors for the elliptic (left)
and divergence (right) approaches on the ellipse $\Omega_1$ (top)
and disk sector $\Omega_2$ (bottom), averaged over 8 seeds.
Numerical differentiation based on MFD 
with polynomial order $q=4+k$, with $k=2$, the order of
$\Delta$, (left) and $k=1$, the order of $\diver$, (right).}
\label{fig:choice-LBcal}
\end{figure}

Comparing the plots in the top row of
Figure~\ref{fig:choice-LBcal} for the smooth domain $\Omega_1$,
we observe that although most errors are smaller for the
divergence approach, the decay is $h^4$ in both cases, as
predicted by the theory presented in Sections~\ref{sssec:neumann}
and \ref{sssec:mfd}. In contrast to this, the plots in the bottom
row for the non-smooth disk sector $\Omega_2$ show that in the case 
of the elliptic approach the quadrature errors are barely decreasing
as $h \to 0$, whereas the divergence approach still delivers
the expected $h^4$ convergence order.
This is the very compelling reason why the divergence
approach is recommended in Algorithms~\ref{algo:mfd}
and \ref{algo:spline}, despite the additional complexity
introduced by the numerical differentiation of a  
vector field instead of a scalar function.
Moreover, the divergence approach can achieve the same
convergence order $q-k$ with a smaller polynomial
order $q$, which means that smaller sets of influence can be used.

Note that the stability constants $K_\muw$ and $K_\sigmav$ of the
quadrature formulas generated by the elliptic approach
on the disk sector never exceed 1.1
in the tests of Figure~\ref{fig:choice-LBcal}(c).
This provides further evidence that the
large quadrature errors are not caused by the instability
of the computed formulas, but rather by the ineffectiveness
of the numerical differentiation scheme on solutions $u$
to the boundary value problem \eqref{eq:bvp-laplacian}
on the disk sector due to their low regularity near the reentrant
corner, which leads to large recovery errors
$\varepsilon_L(u)$ and $\varepsilon_B(u)$ and influences all
estimates of Section~\ref{sec:qf} starting
from \eqref{eq:joint-error-estimate}.

In addition to the experiments in this section, those
in Section~\ref{ssec:conv} for $4\le q\le 8$ on the
disk sector $\Omega_2$ and an L-shaped 3D domain (see
Figures~\ref{fig:disk-sector} and \ref{fig:ell-bricks}), also
confirm that  the divergence approach does not
suffer from saturation of the convergence order of the
quadrature formulas.
This can be interpreted as numerical evidence in support
of our conjecture that the boundary value problem
\eqref{eq:bvp-divergence} for the divergence operator
has solutions of high regularity order on Lipschitz
domains with piecewise smooth boundary, whenever the right hand
side is highly regular, see
Section~\ref{sssec:div} for a more detailed discussion.

\subsubsection{Choice of non-homogeneous constraints}
\label{sssec:choice-fghat}

At least one non-homogeneous constraint
\eqref{eq:non-homogeneous-constraint} has to be included
in the linear system \eqref{eq:ls}, or else the quadrature will
not be accurate for $f,g$ with nonzero combined integral
$I(f,g)\ne0$, and the minimum-norm solution will be trivial.
In Table~\ref{tab:choice-fghat-ellipse} we have compared
quadrature errors and stability constants for
different combinations of the non-homogeneous
constraints introduced
in Section~\ref{ssec:on-the-choice-of-fhat-and-ghat}:
\begin{equation} \label{eq:choice-fghat}
(\fhat,\ghat) \equiv (1,0),
\qquad
(\fhat,\ghat) \equiv (0,1),
\qquad
(\fhat,\ghat) \equiv (1,-1),
\qquad
(\fhat,\ghat) \equiv (0,\partial_\nu \Phi(\cdot,x_0))
\text{ with $x_0 \in \Omega$.}
\end{equation}
The integration domains are the ellipse $\Omega_1$
and the disk sector $\Omega_2$, as defined previously.
The spacing parameter is $h = 0.025$ and the polynomial
order is $q = 5$. Nodes in $X,Y,Z$ are generated
by the advancing front algorithm with $Z \subset Y$,
and relative errors are the RMS of 64 quadrature errors
given by different seeds. The Runge function $f_1$ is
centered at $x_R = (0,0)$ in the case of the
ellipse $\Omega_1$, and at $x_R = (\cos(3\pi/4)/2,\sin(3\pi/4)/2)$
in the case of the disk sector $\Omega_2$.
The fundamental solution is centered at $x_0 = (0.1, 0.05)$,
a point that does not belong to
any axis of symmetry of the two domains.

\begin{table}[htbp!]
\caption{Comparison of RMS quadrature errors and average
stability constants for different non-homogeneous constraints.
MFD algorithm on ellipse $\Omega_1$ and disk sector $\Omega_2$
with polynomial order $q=5$ and 64 distinct seeds.}
\label{tab:choice-fghat-ellipse}
\centering
\begin{small}
\begin{filecontents*}{test-choice-fghat.csv}
Nsamples,hX,hY,hZ,NX,NY,NZ,NrowsA,sizeratioA,nnzA,residual,tnodegen,tmatassembly,tsolve,sreintoneavg,sreintonestd,sreintonerms,sreintonemax,srebndoneavg,srebndonestd,srebndonerms,srebndonemax,sreintrungeavg,sreintrungestd,sreintrungerms,sreintrungemax,srebndrungeavg,srebndrungestd,srebndrungerms,srebndrungemax,sreintfrankeavg,sreintfrankestd,sreintfrankerms,sreintfrankemax,srebndfrankeavg,srebndfrankestd,srebndfrankerms,srebndfrankemax,sreintexpsumavg,sreintexpsumstd,sreintexpsumrms,sreintexpsummax,srebndexpsumavg,srebndexpsumstd,srebndexpsumrms,srebndexpsummax,stabw,stabv,stabwmax,stabvmax,minw,minv,nnzw,nnzv,nnzwv
64,0.04,0.025,0.025,1349,3391,221,2698,0.746954595791805,215627,3.46692211713434e-14,0.221892875,0.204995613932292,0.143526658203125,-8.54037831133723e-17,9.28783660283004e-16,9.25448085002548e-16,2.45020508794227e-15,1.76517662777995e-11,1.30931826687535e-10,1.31098689846008e-10,3.41252868255228e-10,7.5524761885489e-07,3.24382083019608e-06,3.30580708103743e-06,9.20492506928914e-06,5.13655909456587e-11,2.01213599153296e-09,1.99701499032904e-09,4.32586767856571e-09,-1.2006161458765e-07,1.50816257711588e-07,1.91846125238238e-07,5.66168778788991e-07,1.696846063896e-09,1.11291759711985e-07,1.10431906937241e-07,2.95243909217642e-07,1.80718577970766e-08,1.40730675827718e-09,1.81257168422321e-08,2.17742033810169e-08,1.12013016809975e-10,3.76304598983862e-10,3.89794159350189e-10,1.08998911510025e-09,1.42311707724548,1.0009297816057,1.82744710126386,1.01854687299147,-0.00720530856607737,-0.02067958413837,3391,221,3612
64,0.04,0.025,0.025,1349,3391,221,2698,0.746954595791805,212458,3.47902432476109e-14,0.225755466796875,0.204045424479167,0.146903557942708,-1.81396133404203e-11,1.28431944447472e-10,1.28709281847342e-10,3.41252879055102e-10,-9.04111182882486e-17,8.5525580887375e-16,8.53350788000812e-16,1.76803964652575e-15,7.51283452643557e-07,3.18026672562943e-06,3.2435305633508e-06,8.84043629949384e-06,1.42582225566206e-12,2.01302039863176e-09,1.99723226852656e-09,4.37013564620924e-09,-1.09307661430137e-07,1.58257540841947e-07,1.91317220700059e-07,5.6596939222429e-07,1.66791200643929e-09,1.06287881143141e-07,1.05467427232973e-07,2.95303343094807e-07,1.80721231076523e-08,1.35113483783079e-09,1.81217734923916e-08,2.19056140493912e-08,9.15425465613733e-11,3.78832894553324e-10,3.8684879206275e-10,1.08943694664898e-09,1.4223157907953,1.00090150282475,1.8274471015422,1.01765385562244,-0.00720530834015369,-0.0202751374783562,3391,221,3612
64,0.04,0.025,0.025,1349,3391,221,2698,0.746954595791805,215848,3.47630136422362e-14,0.23534896875,0.217287580078125,0.15633464453125,-3.07727971498997e-11,2.282606093019e-10,2.28551444757676e-10,5.94921292468564e-10,-1.31212626289555e-11,9.73286787810958e-11,9.74526840355454e-11,2.53669649016159e-10,7.55216848585696e-07,3.24383977541412e-06,3.30581835092503e-06,9.20504291713504e-06,2.05928194421315e-11,1.95508351760392e-09,1.9398585984245e-09,4.44061428935149e-09,-1.20092375557498e-07,1.50772816856305e-07,1.91831766292361e-07,5.65821177678486e-07,1.66606642245808e-09,1.11307839520377e-07,1.10447390154661e-07,2.9534752256644e-07,1.80410846288605e-08,1.38430856519048e-09,1.80932888785438e-08,2.20032971976448e-08,8.12401699466967e-11,3.92048912711422e-10,3.97367230597943e-10,1.0890282470275e-09,1.42311707733561,1.0009297816342,1.8274471017491,1.01854687251462,-0.00720530817221239,-0.0206795841202745,3391,221,3612
64,0.04,0.025,0.025,1349,3391,221,2698,0.746954595791805,212458,3.4541904936226e-14,0.230376205078125,0.214593874348958,0.154501599609375,-6.72134016413281e-11,1.72686971233777e-09,1.71464330410782e-09,4.23245015451755e-09,-4.90709383327611e-11,1.71849061448345e-09,1.70571804320952e-09,3.97894654318062e-09,7.51833977089325e-07,3.18106006536163e-06,3.24442380970655e-06,8.83996026073689e-06,-4.7089431732809e-11,1.98942821481296e-09,1.97438624052077e-09,4.73040739678529e-09,-1.09355380701545e-07,1.58198835236812e-07,1.91296695977585e-07,5.64835431731337e-07,1.62163820923876e-09,1.05718669019792e-07,1.04902025061465e-07,2.93292398412206e-07,1.80226728050224e-08,2.25625087771225e-09,1.81611635467613e-08,2.23916137193057e-08,4.25473494101045e-11,1.81837506355017e-09,1.80461471668801e-09,4.04484238189763e-09,1.42239736831096,1.0009015028842,1.82744714221755,1.01765385135304,-0.00720530835112663,-0.0202751363927509,3391,221,3612
64,0.04,0.025,0.025,1349,3391,221,2699,0.747231450719823,215848,3.55652583288235e-14,0.236462893229167,0.219740828125,0.159270057942708,-2.8860588776243e-16,8.76972641964151e-16,9.1671014085653e-16,2.63868240239936e-15,-1.28082417575019e-16,7.90110134774042e-16,7.94307779845291e-16,2.41096315435329e-15,7.75138371177248e-07,3.54313569208452e-06,3.59979120747632e-06,1.27905292105491e-05,2.64573948103202e-11,1.9954634277783e-09,1.97998926753016e-09,4.42615441085833e-09,-1.08997197223131e-07,1.53521149191521e-07,1.87298880504706e-07,4.44319571760376e-07,1.20428496763976e-09,1.11874763197134e-07,1.11003833386246e-07,2.95679386875947e-07,1.79122516761143e-08,1.32518143477061e-09,1.79604406067849e-08,2.09729694877322e-08,8.73074567950714e-11,3.83567919617441e-10,3.90446056477947e-10,1.04444788021042e-09,1.41586154256983,1.0009090259128,1.70494273135758,1.01779174969171,-0.00605315190947456,-0.020477424009186,3391,221,3612
64,0.04,0.025,0.025,1349,3391,221,2699,0.747231450719823,215848,3.50735244394485e-14,0.233408321614583,0.214536285807292,0.1586354296875,-1.88477314457097e-16,7.99641522914657e-16,8.1545031924297e-16,1.69629583011388e-15,1.58594286432338e-11,1.29894183750204e-10,1.29847555435526e-10,2.53116734799427e-10,7.52227845180189e-07,3.5218006791861e-06,3.57423119420197e-06,1.2707986558541e-05,1.41875597708427e-11,1.82267142224023e-09,1.8084313925764e-09,3.90548564662093e-09,-1.10929689931521e-07,1.52302881210977e-07,1.8745432407476e-07,4.53660967275969e-07,-4.52733897749837e-10,1.09566149297407e-07,1.08707736433963e-07,3.09871502193942e-07,1.79158225791982e-08,1.35076828422101e-09,1.79658777871623e-08,2.10733185650834e-08,8.99171194079601e-11,3.98884767868722e-10,4.05842416916309e-10,1.07407571982458e-09,1.4155977111052,1.00090824367895,1.7042413694748,1.01779776079913,-0.00606720713066439,-0.0205173893262063,3391,221,3612
64,0.04,0.025,0.025,1349,3391,221,2699,0.747231450719823,212679,3.44149459359068e-14,0.235782494791667,0.217383188151042,0.160927555338542,-1.36722121247527e-11,1.27223035338881e-10,1.26963492908235e-10,2.40885504992352e-10,1.75799396671594e-17,7.71166467644324e-16,7.65319948263423e-16,2.25023227739641e-15,6.98732304131796e-07,3.43051173344787e-06,3.47458719391818e-06,1.16977827529068e-05,2.76153940134851e-11,1.81362330040334e-09,1.79961047874077e-09,4.08908775343933e-09,-1.07607189287883e-07,1.58634180345166e-07,1.90659146936701e-07,4.75440562529018e-07,-2.59051289571604e-09,1.08889852746383e-07,1.08066855060677e-07,3.05896779013509e-07,1.78991840488546e-08,1.3241126513768e-09,1.79473304147764e-08,2.10523333219969e-08,6.87168263878414e-11,3.89529921074194e-10,3.92536271922826e-10,1.13062199755115e-09,1.41605461969319,1.00089749925137,1.70451617581334,1.01779554689509,-0.00606644628956179,-0.0207322731329815,3391,221,3612
64,0.04,0.025,0.025,1355,3402,269,2710,0.738218469081994,217336,3.61461294153971e-14,0.227911635416667,0.2005428984375,0.088917060546875,2.32651685032979e-16,1.10704958010024e-15,1.1227360435667e-15,3.01563703131356e-15,5.28255745051559e-09,5.41894622658758e-08,5.40233325095267e-08,1.24445343288794e-07,6.84540566269904e-08,3.1625584212254e-06,3.13850027547596e-06,7.22644935805803e-06,2.4040968320111e-08,1.2547603804392e-07,1.26791957913516e-07,2.57572213176079e-07,2.22645235848773e-07,1.95396613563549e-07,2.95218864050906e-07,7.14922550297391e-07,-6.75810318099754e-08,8.78667910403019e-08,1.1030473641508e-07,3.26717737683301e-07,9.316093664285e-09,7.33622572596691e-08,7.33806263025071e-08,1.56497611955609e-07,-6.11462522339209e-09,8.37034213939126e-08,8.32717151207604e-08,2.31353869112973e-07,1.34798323977821,1.00046900262544,1.67954714638492,1.00686091943648,-0.0108967039909058,-0.0164493974461373,3402,269,3670
64,0.04,0.025,0.025,1355,3402,269,2710,0.738218469081994,214203,3.62280663910787e-14,0.230922390625,0.207573354817708,0.0914095572916667,-5.10575934636585e-09,5.40980854604612e-08,5.39160779309058e-08,1.24439621462762e-07,-2.06748861670098e-17,9.01517161417853e-16,8.94685244719478e-16,2.64638542937726e-15,6.43216412799342e-08,3.16691707132535e-06,3.14273641939691e-06,7.43402008103882e-06,1.79678075901051e-08,9.85510903854481e-08,9.94153135393551e-08,2.03807255544864e-07,2.15599463872134e-07,1.88403815744924e-07,2.8534978823152e-07,6.89308997480047e-07,-7.24842098202435e-08,6.20051523894815e-08,9.50711686734653e-08,2.29146461667864e-07,4.24696979851234e-09,9.39418636112672e-08,9.33017615806867e-08,2.14581624628669e-07,-1.1330119150541e-08,7.57925677928466e-08,7.60468731229385e-08,1.76849912159385e-07,1.34492344898918,1.0004678291471,1.67954766843628,1.00686079778045,-0.0108966788951561,-0.0164738796881599,3402,269,3670
64,0.04,0.025,0.025,1355,3402,269,2710,0.738218469081994,217605,3.6318472285738e-14,0.238564591796875,0.225656646484375,0.09864112890625,-8.12319123242019e-09,8.34995360755089e-08,8.3241927926516e-08,1.91742209532666e-07,-2.8514166752593e-09,2.93101527959856e-08,2.92197267264007e-08,6.73056847012484e-08,6.03250667426186e-08,3.15109738455816e-06,3.12696446022943e-06,7.1593551730463e-06,1.59005406923583e-08,9.55348641007267e-08,9.61099863711404e-08,1.9948506400213e-07,2.14526129822125e-07,1.94597354824044e-07,2.88613758060563e-07,6.75461094064432e-07,-7.57166084391639e-08,6.49242543051874e-08,9.9409766795873e-08,2.48641194437264e-07,1.19161398403831e-09,1.14711609846165e-07,1.1381813500048e-07,2.63196273523539e-07,-1.42476586377572e-08,8.66219684023451e-08,8.711556176726e-08,2.15253524055151e-07,1.34798358124616,1.00046900846321,1.67954795079371,1.00686073197265,-0.0108966653212133,-0.0164494093695797,3402,269,3670
64,0.04,0.025,0.025,1355,3402,269,2710,0.738218469081994,214203,3.57942668904334e-14,0.234049211588542,0.215359828776042,0.0986416666666667,1.0758802463831e-07,3.3844052838298e-06,3.35958367166665e-06,7.73447426478013e-06,1.12682441667837e-07,3.39689129445167e-06,3.37213181058296e-06,7.82431927248521e-06,1.77233561384089e-07,4.65347451945282e-06,4.62037662926304e-06,1.12047796649053e-05,1.30669238489774e-07,3.38339333540309e-06,3.3593987648042e-06,7.87973002886334e-06,3.28431865867141e-07,3.47678656117316e-06,3.46511710950683e-06,8.20566629043323e-06,4.04968487130331e-08,3.4415749403422e-06,3.4148219162796e-06,7.81326108195967e-06,1.16697678033891e-07,3.31789629899371e-06,3.2939410117312e-06,7.6095033334352e-06,1.01110409714146e-07,3.32732926084767e-06,3.3027802038416e-06,7.65741482581149e-06,1.34491692282877,1.00046773996162,1.67949028435683,1.00686655693645,-0.0108974188531849,-0.0164741113748668,3402,269,3670
64,0.04,0.025,0.025,1355,3402,269,2711,0.738490874421139,217605,3.67674522356923e-14,0.239072543619792,0.219252725260417,0.1026521015625,5.8899160767843e-18,9.38470523222622e-16,9.31128485171579e-16,2.45020508794226e-15,5.99571698843285e-17,7.58443522700472e-16,7.54879704210705e-16,1.85246980056408e-15,4.71842725374625e-08,3.1159110145671e-06,3.09183217867793e-06,6.60041255342228e-06,1.31688120416903e-08,9.27407967410684e-08,9.29509796906004e-08,2.04254161139917e-07,2.23609427212693e-07,2.11475935452409e-07,3.06634089510904e-07,9.00493154046122e-07,-7.42336061894074e-08,5.99750615598273e-08,9.5139019581068e-08,2.23976794071053e-07,8.54752594875475e-09,7.3107114864626e-08,7.30356096956438e-08,1.56959429064059e-07,-1.08916876194613e-08,7.39804164396723e-08,7.42038655110236e-08,1.69584811692031e-07,1.35340529160326,1.00046719762084,1.76815762850518,1.00690521414693,-0.0116901262170236,-0.016256289012604,3402,269,3670
64,0.04,0.025,0.025,1355,3402,269,2711,0.738490874421139,217605,3.62312122036728e-14,0.240151080078125,0.216621770182292,0.101792229166667,-8.83487411517644e-18,8.40504366298904e-16,8.33958872593868e-16,2.26172777348517e-15,5.42942079318686e-09,5.38881423512634e-08,5.37404555293279e-08,1.19362271085229e-07,4.16174641413483e-09,3.0635882384083e-06,3.03956257349147e-06,7.08278077460445e-06,2.25906822464865e-08,1.23888911058561e-07,1.24975923619158e-07,3.05535831765143e-07,2.19260097188783e-07,2.01436947419939e-07,2.96677974950861e-07,8.34219066428219e-07,-7.23368281395156e-08,6.94713787439502e-08,9.99173591655226e-08,2.66008841624672e-07,9.2507669859167e-09,4.90294509848332e-08,4.95166943396755e-08,1.28313217542349e-07,-5.94227317662597e-09,6.64659589383647e-08,6.62118371239395e-08,1.86250839470455e-07,1.34928484278964,1.00046384897055,1.7553746330905,1.00664819294925,-0.0114557080993942,-0.0163928131137248,3402,269,3670
64,0.04,0.025,0.025,1355,3402,269,2711,0.738490874421139,214472,3.6583880352321e-14,0.2398553203125,0.220213666666667,0.1018913359375,-5.36049356070811e-09,5.36536623469537e-08,5.35020601942548e-08,1.20568942958615e-07,1.19914339768657e-16,8.21008909966768e-16,8.23348625980755e-16,1.85246980056408e-15,1.55246066955168e-08,3.04705815865829e-06,3.02319915523051e-06,7.11388281629656e-06,1.69988857423631e-08,9.84639518617997e-08,9.91595946787093e-08,2.54930233207852e-07,2.1385113265504e-07,1.97619636640361e-07,2.90130343917062e-07,7.82206525285209e-07,-7.7906091511879e-08,4.5212759677599e-08,8.98977875519363e-08,1.76943829901748e-07,3.95821746880373e-09,7.17270808227316e-08,7.12745011431414e-08,1.5404866345901e-07,-1.12907699438069e-08,5.09167276200986e-08,5.1763757779652e-08,1.3318660310771e-07,1.35051258593617,1.00046393707469,1.75527232141579,1.00665371127828,-0.0114623139886841,-0.0164344562192388,3402,269,3670
\end{filecontents*}
\pgfplotstabletypeset[
	create on use/constraints/.style={
        create col/set list={
        	{$(1,0)$}, {$(0,1)$}, {$(1,-1)$},
        	{$(0,\partial_\nu\Phi)$},
        	{$(1,0)$ and $(0,1)$},
        	{$(1,0)$ and $(0,\partial_\nu\Phi)$},
        	{$(0,1)$ and $(0,\partial_\nu\Phi)$},
        	{$(1,0)$}, {$(0,1)$}, {$(1,-1)$},
        	{$(0,\partial_\nu\Phi)$},
        	{$(1,0)$ and $(0,1)$},
        	{$(1,0)$ and $(0,\partial_\nu\Phi)$},
        	{$(0,1)$ and $(0,\partial_\nu\Phi)$}
        }
    },
	columns/constraints/.style={
		column name = {Constraints $(\fhat,\ghat)$},
		column type = {l},
		string type
	},
	create on use/domain/.style={
        create col/set list={
        	$\Omega_1$, $\Omega_1$, $\Omega_1$,
        	$\Omega_1$, $\Omega_1$, $\Omega_1$, $\Omega_1$,
        	$\Omega_2$, $\Omega_2$, $\Omega_2$,
        	$\Omega_2$, $\Omega_2$, $\Omega_2$, $\Omega_2$
        }
    },
	columns/domain/.style={
		column name = {Domain},
		column type = {c},
		string type
	},
	columns/hY/.append style={
		precision = 1
	},
	columns/stabw/.append style={
		precision = 4
	},
	columns/stabv/.append style={
		precision = 4
	},
	every row no 6/.append style={
		after row=\midrule
	},
	columns={
		domain,constraints,hY,sreintrungerms,srebndrungerms,
		sreintfrankerms,srebndfrankerms,stabw,stabv
	}]
{test-choice-fghat.csv}
\end{small}
\end{table}

The results in the table show that quadrature errors
depend quite weakly on the choice of non-homogeneous
constraints. In particular, on the ellipse, whose boundary
is smooth, there is essentially no difference in the errors,
whereas on the disk sector, with a non-smooth boundary,
a difference of one order of magnitude or more can be
observed between enforcing exactness for
constants and enforcing exactness for
the normal derivative of the fundamental solution
centered at $x_0 \in \Omega$.
This is the reason to recommend the choice $(\fhat,\ghat) \equiv (0,1)$
in Algorithms~\ref{algo:mfd} and \ref{algo:spline}
whenever the measure of $\partial\Omega$ is known or can be
sufficiently accurately approximated at low cost, as in the case of domains
whose boundary is described by parametric patches with standard
parameter domains for which accurate quadrature formulas are
available.
Whenever this is challenging, the purely moment-free
approach based on the fundamental solution remains effective,
although a smaller spacing parameter $h$ may be needed to achieve
the same errors on domains with non-smooth boundary.
As far as stability is concerned, all the non-homogeneous
constraints perform equally well. Boundary quadrature formulas
are close to being positive, but no constraint can consistently
deliver non-negative weights across all 64 seeds.

\subsubsection{Choice of minimization norm}
\label{sssec:choice-norm}

In Section~\ref{ssec:on-the-choice-of-minimization-norm},
we have discussed relative merits of the norms
\[
\norm{(\muw,\sigmav)}_1 = \norm{\muw}_1 + \norm{\sigmav}_1,
\qquad
\norm{(\muw,\sigmav)}_2 = \sqrt{\norm{\muw}_2^2 + \norm{\sigmav}_2^2}
\]
as candidates for the optimization norm
$\norm{(\muw,\sigmav)}_\sharp$ to be used in the final step
of the general framework described in
Algorithm~\ref{algo:abstract-algorithm}, where the 
quadrature weights $\muw^*$ and $\sigmav^*$ are computed
by minimizing $\norm{(\muw,\sigmav)}_\sharp$
over all solutions to the non-homogeneous
linear system $Ax=b$ of \eqref{eq:ls}.
We now present numerical experiments to compare their
performance and support our choice of the 2-norm for
the practical algorithms of Section~\ref{ssec:palgos}.

The 1-norm minimization problem can be reformulated
as the linear program
\begin{equation} \label{eq:linprog}
\text{Minimize} \quad \mathbbm{1}^T x^+ + \mathbbm{1}^T x^-
\quad \text{subject to} \quad A x^+ - A x^- \! = b
\quad \text{and} \quad x^+, x^- \geq 0,
\end{equation}
with $x=x^+ -x^-$, and solved by e.g.\ the dual simplex algorithm
or an interior-point method
\cite{nocedal1999numerical,huangfu2018parallelizing}.
For our numerical tests, we have used the implementations
provided by the 2024a release of MATLAB's Optimization Toolbox.
For the solution of the 2-norm minimization problem,
we have used in this test (and all other tests in
Section~\ref{ssec:validation-of-parameters})
a direct method based on the sparse
QR factorization provided by the SuiteSparseQR library
\cite{davis-algorithm}, version 4.3.3 available from \cite{SuiteSparse}.

The integration domain for these tests is the ellipse $\Omega_1$,
the polynomial order is $q = 5$, and 
the sets of nodes $X,Y,Z$ are generated as usual by the
advancing front algorithm with $Z \subset Y$.
The Runge function $f_1$ is centered at $x_R = (0,0)$.

Table~\ref{tab:choice-norm-convergence} reports quadrature errors
and stability constants of the formulas produced by
Algorithm~\ref{algo:mfd} when minimizing the norms
$\norm{(\muw,\sigmav)}_1$ and $\norm{(\muw,\sigmav)}_2$,
denoted in the table by $L^1$ and $L^2$, respectively.
Separate rows are devoted to the 1-norm minimization computed
by the dual simplex algorithm and the interior-point method,
because their results are significantly different.
For the same tests, Table~\ref{tab:choice-norm-sparsity} reports
further information such as the runtime $t_{\text{solve}}$
of the solver, the number of the rows $m$ of the system matrix $A$,
and the number of nonzero weights in the quadrature
formulas $\muw$ and $\sigmav$. The computations have been performed 
on an Intel %
NUC Mini PC system with a Core %
i7-1165G7 CPU and 64 GB of DDR4 RAM.
For each value of the spacing parameter $h$, results
are averaged over 8 different seeds. To enhance readability, we
have rounded to the nearest integer the values that are meant
to be integers, such as the number $m$ of the rows of the system
matrix.

It turned out that the interior-point method for 1-norm minimization may
stall or lose feasibility when solving the linear
program~\eqref{eq:linprog}.
For this reason, we have included the column "conv" in the tables
that reports how many runs have converged out of 8,
and all statistics are computed only for the seeds for
which convergence has been achieved. 

\begin{table}[htbp!]
\caption{Comparison of quadrature errors and average
stability constants for different minimization norms and solvers.
MFD algorithm on ellipse $\Omega_1$ with polynomial
order $q=5$ and 8 distinct seeds. Column conv reports
how many runs have converged out of 8. Value 1 in column
$K_v$ implies that the formula $(Z, v)$ was positive for all
seeds.}
\label{tab:choice-norm-convergence}
\centering
\begin{small}
\begin{filecontents*}{test-choice-norm.csv}
Nsamples,hX,hY,hZ,NX,NY,NZ,NrowsA,sizeratioA,nnzA,residual,tnodegen,tmatassembly,tsolve,sreintoneavg,sreintonestd,sreintonerms,sreintonemax,srebndoneavg,srebndonestd,srebndonerms,srebndonemax,sreintrungeavg,sreintrungestd,sreintrungerms,sreintrungemax,srebndrungeavg,srebndrungestd,srebndrungerms,srebndrungemax,sreintfrankeavg,sreintfrankestd,sreintfrankerms,sreintfrankemax,srebndfrankeavg,srebndfrankestd,srebndfrankerms,srebndfrankemax,sreintexpsumavg,sreintexpsumstd,sreintexpsumrms,sreintexpsummax,srebndexpsumavg,srebndexpsumstd,srebndexpsumrms,srebndexpsummax,stabw,stabv,stabwmax,stabvmax,minw,minv,nnzw,nnzv,nnzwv
8,0.16,0.1,0.1,97,232,55,195,0.679442508710801,16145,6.72941741025179e-14,0.033564890625,0.0610703333333333,0.252469994791667,-7.74008584660583e-09,8.63281049843e-08,8.11226412619844e-08,1.36506852306894e-07,-3.83744968734566e-15,1.11309344377025e-14,1.10966890952791e-14,3.08603283757222e-14,-0.00191422939605912,0.00230551352294099,0.00288361625561424,0.0062804059172199,2.10121657765987e-07,2.10775512426344e-06,1.98278940307305e-06,4.11513688317864e-06,1.81400384264291e-05,0.000202544511469096,0.000190329464953761,0.000410450337032694,1.53384058518171e-05,0.000142987171276526,0.000134628865589924,0.000271893654464137,1.8828419348682e-06,1.15242487725678e-06,2.16960053055937e-06,3.55583351544309e-06,-3.07047796089051e-08,5.11926309981501e-07,4.7984660255721e-07,9.48858876875379e-07,1.00304887945745,1,1.01595640624769,1.00000000000003,-0.00687193496236141,0,148,41,189
7,0.16,0.1,0.1,97,232,55,195,0.679442508710801,16146,4.13877821122236e-07,0.0220224583333333,0.034314375,0.0651618035714286,6.82703590900045e-08,9.48321783082638e-08,1.11217127257636e-07,2.24824462095198e-07,9.03937510873778e-09,2.28555864663795e-08,2.30100571577739e-08,6.08209388256217e-08,-0.00213240594516741,0.000796730018406848,0.00225638010153687,0.00311222854522366,-5.15526892155458e-07,1.13773666141754e-06,1.1727284486645e-06,2.09784252361219e-06,4.57114189825763e-05,0.000225221285840971,0.000213466123894223,0.00040820626695939,3.91388807070903e-05,0.000108400962894525,0.000107721583206622,0.000169665141700155,1.56816894367934e-06,1.15849608443294e-06,1.89987803809017e-06,3.02912798968296e-06,-7.0161006092372e-08,3.01307060695851e-07,2.87644035070704e-07,7.13053150873485e-07,1.00351218143344,0.999999990960625,1.0159772730797,1.00000000010632,-0.00683534011433193,4.44502728444521e-07,232,55,287
8,0.16,0.1,0.1,97,232,55,195,0.679442508710801,16145,1.02551390158229e-14,0.021587453125,0.036749921875,0.00699642708333333,-1.79837487504367e-08,1.13083007662869e-07,1.07297301985834e-07,1.70902519778427e-07,-8.03654384784432e-17,3.32744527532523e-16,3.21461753913773e-16,6.42923507827545e-16,-0.00145916007120332,0.00347929344276932,0.00356671355966238,0.00668713051953241,-3.83119149979421e-07,3.46229372092187e-06,3.26126100868577e-06,6.13300896041938e-06,7.22351903648368e-05,0.000600978491865578,0.000566785829878836,0.000929232555406443,-1.60136440345196e-05,8.9773344084405e-05,8.54884988746975e-05,0.000217270517617202,2.2669094050016e-06,1.49709619924223e-06,2.66458498673753e-06,4.90865773202421e-06,3.57227810828567e-08,3.3988055882841e-07,3.19929776785652e-07,4.87054479689992e-07,1.3587196513791,1.0055392541698,1.56584250993916,1.04279195320761,-0.0333450745728629,-0.118231450283243,232,55,287
8,0.08,0.05,0.05,352,874,110,706,0.717479674796748,56920,1.44465950038108e-12,0.0686370885416667,0.0852240364583333,6.80044926041667,-7.89086286844033e-10,2.53587289971105e-09,2.49989542004023e-09,4.59327809897485e-09,-6.21425753034562e-14,1.30677180802663e-13,1.37126436364624e-13,3.72413441909106e-13,-8.50101275977521e-05,0.000179023275231014,0.000187802790805463,0.000338794832556553,-1.17783334007433e-08,6.17592056809448e-08,5.89589152567078e-08,9.57210098580325e-08,3.30484358134753e-06,5.98144369142943e-06,6.49826524171623e-06,1.61894060415027e-05,-6.10549232862514e-07,4.77756388782008e-06,4.51051521227634e-06,9.14138961609866e-06,2.41278821485701e-07,2.29096448679837e-08,2.42228642086374e-07,2.70480327078395e-07,-1.15853310090603e-09,9.8374578621594e-09,9.27474145121774e-09,1.90302380789089e-08,1.01992478650922,1.00000000000006,1.04793743107056,1.00000000000037,-0.00798585071305622,-1.6940658945086e-21,622,80,702
6,0.08,0.05,0.05,353,874,110,706,0.716751269035533,56962,4.5766272448486e-07,0.0661055069444444,0.0780647083333333,0.378180333333333,1.05923838901957e-08,2.08201742629735e-08,2.17584844009817e-08,5.2725596699053e-08,1.21773178008129e-08,1.80757748646021e-08,2.05076839164855e-08,4.86039229359521e-08,-9.97064962832032e-05,9.97708483012406e-05,0.000135042847064787,0.000197305816634291,-2.29849680620249e-08,4.56479193795887e-08,4.75894162941386e-08,8.50189757766838e-08,1.36221881050524e-06,3.55333892815086e-06,3.51816543617379e-06,7.52188852732991e-06,-5.62097493920092e-08,3.53168476164515e-06,3.22446231935763e-06,6.41180687642763e-06,2.56244077472237e-07,3.34707484633423e-08,2.58059301469988e-07,3.04531406270594e-07,1.30252309148169e-08,2.25192946252575e-08,2.43363000282896e-08,5.25336953681295e-08,1.01425105055444,0.999999987822682,1.03106576100062,0.999999999449731,-0.00603749759677362,6.32346843546657e-07,874,110,985
8,0.08,0.05,0.05,352,874,110,706,0.717479674796748,56920,1.90374974021125e-14,0.07158596875,0.077675140625,0.0244303020833333,6.81565020714421e-10,3.16355201077637e-09,3.0367061986201e-09,6.25204338384746e-09,-1.20548157717665e-16,4.68607005775952e-16,4.546155721691e-16,8.03654384784432e-16,-0.000171881966515299,0.000197582976194483,0.000252393668373045,0.000462655494787573,-1.37292078297739e-08,4.78014599868658e-08,4.67744404285029e-08,8.59582009468836e-08,1.76679848909504e-06,7.9527947868998e-06,7.64608747715403e-06,1.63000839374074e-05,-7.01075947625121e-07,6.76833525336384e-06,6.36989594291188e-06,1.30591654788309e-05,2.2666845288991e-07,2.69755348702977e-08,2.28068645547457e-07,2.65415628573934e-07,2.16616466348492e-10,1.13141804036647e-08,1.05856632349763e-08,2.2090817245935e-08,1.37186536628503,1.0006589736797,1.48391664232634,1.00318494135822,-0.0161969252192602,-0.00696858468200652,874,110,984
8,0.04,0.025,0.025,1349,3395,221,2698,0.746128318584071,212701,1.61280796458867e-08,1.186814125,0.499971125,2288.30846425,7.05116816093679e-10,2.03085553038254e-09,2.02633095432675e-09,5.72416648776832e-09,3.41553113533383e-16,7.41029234642513e-15,6.94010353715385e-15,1.31799319104647e-14,6.48255061292684e-07,2.18736472234147e-06,2.14632907529753e-06,3.38935356895598e-06,6.7758003973324e-11,1.4931210795591e-09,1.39832949820427e-09,2.33637679298958e-09,-7.6012374357698e-08,1.36795823956209e-07,1.4883494658806e-07,3.06584200373152e-07,5.34046813480059e-10,7.3641592443502e-08,6.88874722022225e-08,1.15041215913685e-07,1.93843009021873e-08,2.21603086405699e-09,1.94948214959303e-08,2.36870593672633e-08,6.53666359562963e-10,1.81032413464183e-09,1.81518428762278e-09,5.04615057558877e-09,1.06878881739832,1,1.0900225269862,1.00000000000001,-0.0031542483526751,0,2537,155,2692
0,0.04,0.025,0.025,1349,3395,221,2698,0.746128318584071,212701,NaN,NaN,NaN,NaN,NaN,NaN,NaN,NaN,NaN,NaN,NaN,NaN,NaN,NaN,NaN,NaN,NaN,NaN,NaN,NaN,NaN,NaN,NaN,NaN,NaN,NaN,NaN,NaN,NaN,NaN,NaN,NaN,NaN,NaN,NaN,NaN,NaN,NaN,NaN,NaN,NaN,NaN,NaN,NaN,NaN
8,0.04,0.025,0.025,1349,3395,221,2698,0.746128318584071,212701,3.47653936227351e-14,1.17524575,0.346009,0.117486875,4.86563611136719e-12,1.11393588275862e-10,1.04312700497681e-10,1.55686219765166e-10,-1.20548157717665e-16,8.00202739406146e-16,7.58166030275089e-16,1.28584701565509e-15,1.04807708337001e-07,2.54271720025791e-06,2.3808021906814e-06,4.32599241677775e-06,2.10831004287889e-10,1.5366511993834e-09,1.45278508643387e-09,2.91059985359306e-09,-9.75559615732014e-08,1.29500684239349e-07,1.55535492588104e-07,2.19636160327545e-07,-1.04558261352198e-08,8.51193651253361e-08,8.03054626385571e-08,1.72557230912727e-07,1.85604669704664e-08,1.13499313168201e-09,1.85908073357404e-08,2.0750671115484e-08,1.0163144227807e-10,3.01138780636401e-10,2.99462759975763e-10,7.00742351976745e-10,1.42696635257018,1.00223134134937,1.56430329360763,1.01668460944121,-0.0048328586670537,-0.0193234489921056,3395,221,3616
\end{filecontents*}
\pgfplotstabletypeset[
	create on use/normsolver/.style={
        create col/set list={
        	$L^1$ dual simplex, $L^1$ interior-point, $L^2$ SuiteSparseQR,
        	$L^1$ dual simplex, $L^1$ interior-point, $L^2$ SuiteSparseQR,
        	$L^1$ dual simplex, $L^1$ interior-point, $L^2$ SuiteSparseQR
        }
    },
	create on use/converged/.style={
        create col/set list={
        	8/8, 7/8, 8/8, 8/8, 6/8, 8/8, 8/8, 0/8, 8/8
        }
    },
	columns/normsolver/.style={
		column name = {Norm and solver},
		column type = {l},
		string type
	},
	columns/converged/.style={
		column name = {conv},
		column type = {c},
		string type
	},
	every nth row={3}{
		before row=\midrule
	},
	columns/hY/.append style={
		precision=1
	},
	columns/stabw/.append style={
		precision = 3
	},
	columns/stabv/.append style={
		precision = 3, int detect, sci generic={}
	},
	columns={
		normsolver,hY,converged,
		sreintrungerms,srebndrungerms,
		sreintfrankerms,srebndfrankerms,stabw,stabv
	}]
{test-choice-norm.csv}
\end{small}
\end{table}

\begin{table}[htbp!]
\caption{Comparison of solver runtime and sparsity of the
quadrature formulas for different minimization norms and solvers.
MFD algorithm on ellipse $\Omega_1$ with polynomial
order $q=5$ and 8 distinct seeds. Column conv reports
how many runs have converged out of 8, $m$ denotes the number of
rows of the system matrix.}
\label{tab:choice-norm-sparsity}
\centering
\begin{small}
\pgfplotstabletypeset[
	create on use/normsolver/.style={
        create col/set list={
        	$L^1$ dual simplex, $L^1$ interior-point, $L^2$ SuiteSparseQR,
        	$L^1$ dual simplex, $L^1$ interior-point, $L^2$ SuiteSparseQR,
        	$L^1$ dual simplex, $L^1$ interior-point, $L^2$ SuiteSparseQR
        }
    },
	create on use/converged/.style={
        create col/set list={
        	8/8, 7/8, 8/8, 8/8, 6/8, 8/8, 8/8, 0/8, 8/8
        }
    },
	columns/normsolver/.style={
		column name = {Norm and solver},
		column type = {l},
		string type
	},
	columns/converged/.style={
		column name = {conv},
		column type = {c},
		string type
	},
	every nth row={3}{
		before row=\midrule
	},
	columns/hY/.append style={
		precision=1
	},
	columns/stabw/.append style={
		precision = 3
	},
	columns/stabv/.append style={
		precision = 3, int detect, sci generic={}
	},
	columns={
		normsolver,hY,converged,tsolve,NrowsA,NY,nnzw,NZ,nnzv
	}]
{test-choice-norm.csv}
\end{small}
\end{table}

As we see in Table~\ref{tab:choice-norm-convergence},
minimization of the 1-norm leads on average to smaller
quadrature errors compared to the 2-norm,
although the advantage becomes less significant
as $h$ decreases.
The $L^1$ methods also deliver more stable formulas,
especially on the boundary: the appearance of the exact
value 1 for $K_\sigmav$ indicates that a positive formula
$(Z,\sigmav)$ has been found for all seeds, because
Algorithm~\ref{algo:mfd} relies on the non-homogeneous
constraint $(\fhat,\ghat)\equiv(0,1)$.
The $L^2$ method delivers positive formulas
on the boundary only for a few seeds.

The $L^1$ methods, however, struggle with computational
efficiency and robustness. The interior-point method
does not converge for any seed when $h = 0.025$,
even if parameters such as \texttt{OptimalityTolerance},
\texttt{ConstraintTolerance}, and \texttt{MaxIterations}
are increased significantly from their default values.
The dual simplex algorithm always terminates successfully
in these tests, but its running time is orders of magnitude
longer compared to SuiteSparseQR. This is the reason why
the 2-norm is recommended in Algorithms~\ref{algo:mfd}
and \ref{algo:spline} for general use.

The two algorithms for the minimization of the 1-norm
that we have considered so far are not the only iterative
algorithms that struggle with the system $Ax=b$: the ones
for the minimization of the 2-norm, such as LSQR \cite{paige-lsqr},
also require too many iterations to be practical,
unless an effective preconditioning strategy is employed.
(Recall that the condition number of the minimum-norm
least squares problem naturally grows as $h \to 0$.)
Since out-of-the-box performance of standard left and right
preconditioners such as symmetric successive
over-relaxation and incomplete Cholesky factorization
was not found to be satisfactory, we opted for a direct solver
based on a sparse QR factorization of $A^T$, as implemented
in the command \texttt{spqr\_solve} provided by the
MATLAB interface to the open source library SuiteSparseQR
\cite{davis-algorithm}. Fill-in during sparse QR factorization
is mitigated by the reordering algorithms used by SuiteSparseQR,
which were found to be quite effective, especially in the
case of 2D domains.
The QR factorization provided by SuiteSparseQR is rank-revealing,
and the numerical rank of $A$ is taken into account when
computing the solution to $Ax=b$ with minimal 2-norm.
To maximize accuracy, the tolerance $10^{-15}$
was passed to \texttt{spqr\_solve} for numerical rank
computation in the tests in this paper, unless stated otherwise.

By the well-known properties of the linear program
\eqref{eq:linprog}, its solution is not necessarily unique, and
there is at least one solution $x$ with the number of nonzero
components not exceeding the number $m$ of rows in the matrix $A$.
The dual simplex algorithm is known to terminate at a solution
with this property, and indeed we see that
$\nnz(\muw)+\nnz(\sigmav)$ is less than $m$ in all cases
when this method is used, meaning that the combined
quadrature weights $x = (\muw,\sigmav)^T$ have a significant
proportion of zeros, which may be seen as an advantage.
Moreover, the quadrature formulas $(Y,\muw)$ and $(Z,\sigmav)$
are also sparse when considered individually, because
$\nnz(\muw) < N_Y$ and $\nnz(\sigmav) < N_Z$.
In contrast to this, all weights coming from the interior-point
method and the 2-norm minimization are always nonzero.
Even sparser weights can be obtained by the dual simplex
method with a larger ratio $h_X/h$, although the accuracy
of the resulting quadrature formulas will be reduced.

\subsection{Convergence tests in 2D and 3D}
\label{ssec:conv}

Now that numerical evidence for the choices made
in Algorithms~\ref{algo:mfd} and \ref{algo:spline}
has been given, we move on to the evaluation
of relative quadrature errors
for $h \to 0$, i.e.\ for an increasing number of
quasi-uniform quadrature nodes, on the test 
domains $\Omega_2,\Omega_3\subset\R^2$ and $\Omega_5,\Omega_6\subset\R^3$
depicted in Figure~\ref{fig:domains}.

\begin{figure}[tbp]
\centering
\hspace*{\fill}
\subcaptionbox{Ellipse $\Omega_1$.}{
	\includegraphics[width=0.22\textwidth]{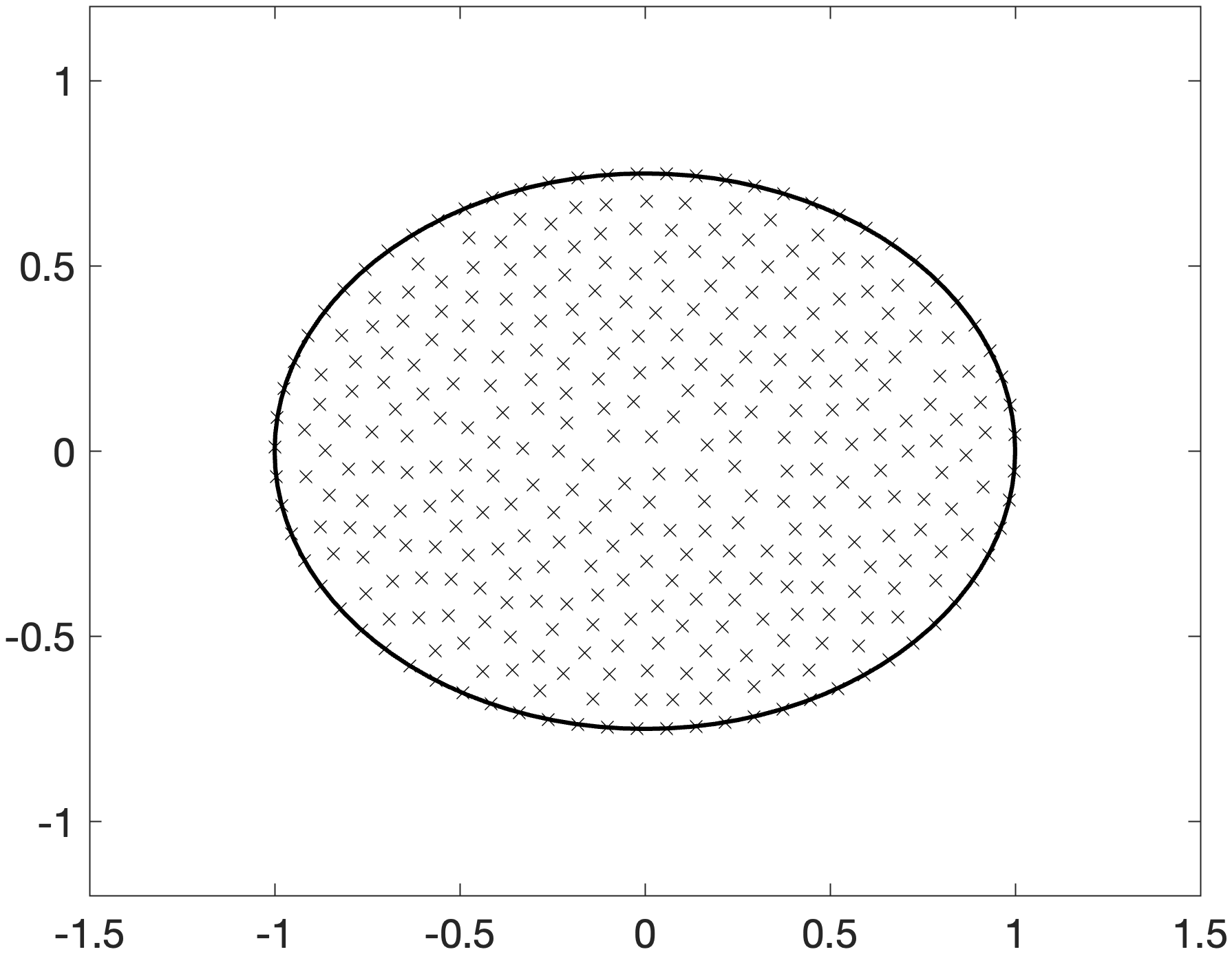}
} \hspace*{\fill}
\subcaptionbox{Disk sector $\Omega_2$.}{
	\includegraphics[width=0.22\textwidth]{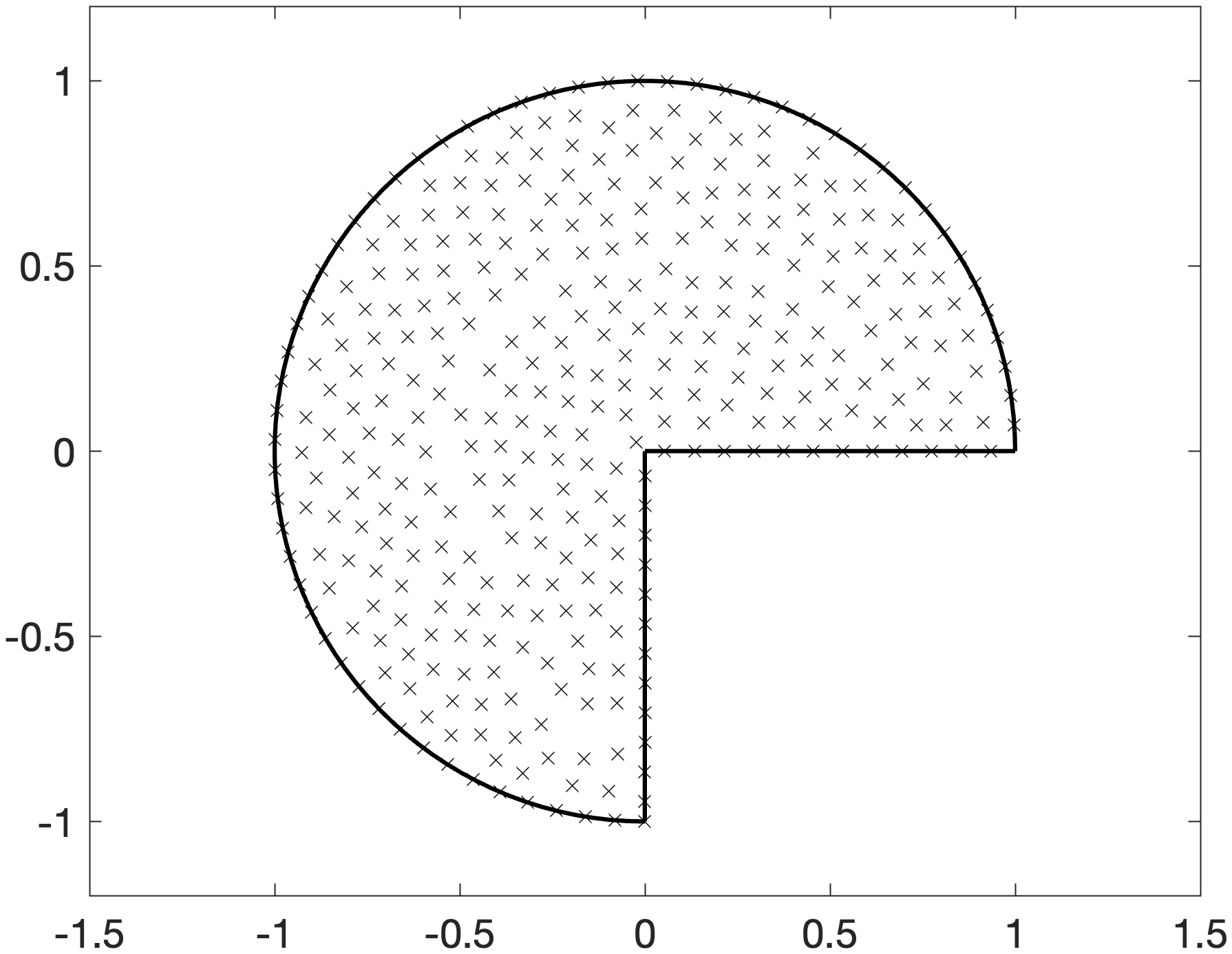}
} \hspace*{\fill}
\subcaptionbox{Cassini oval $\Omega_3$.}{
	\includegraphics[width=0.22\textwidth]{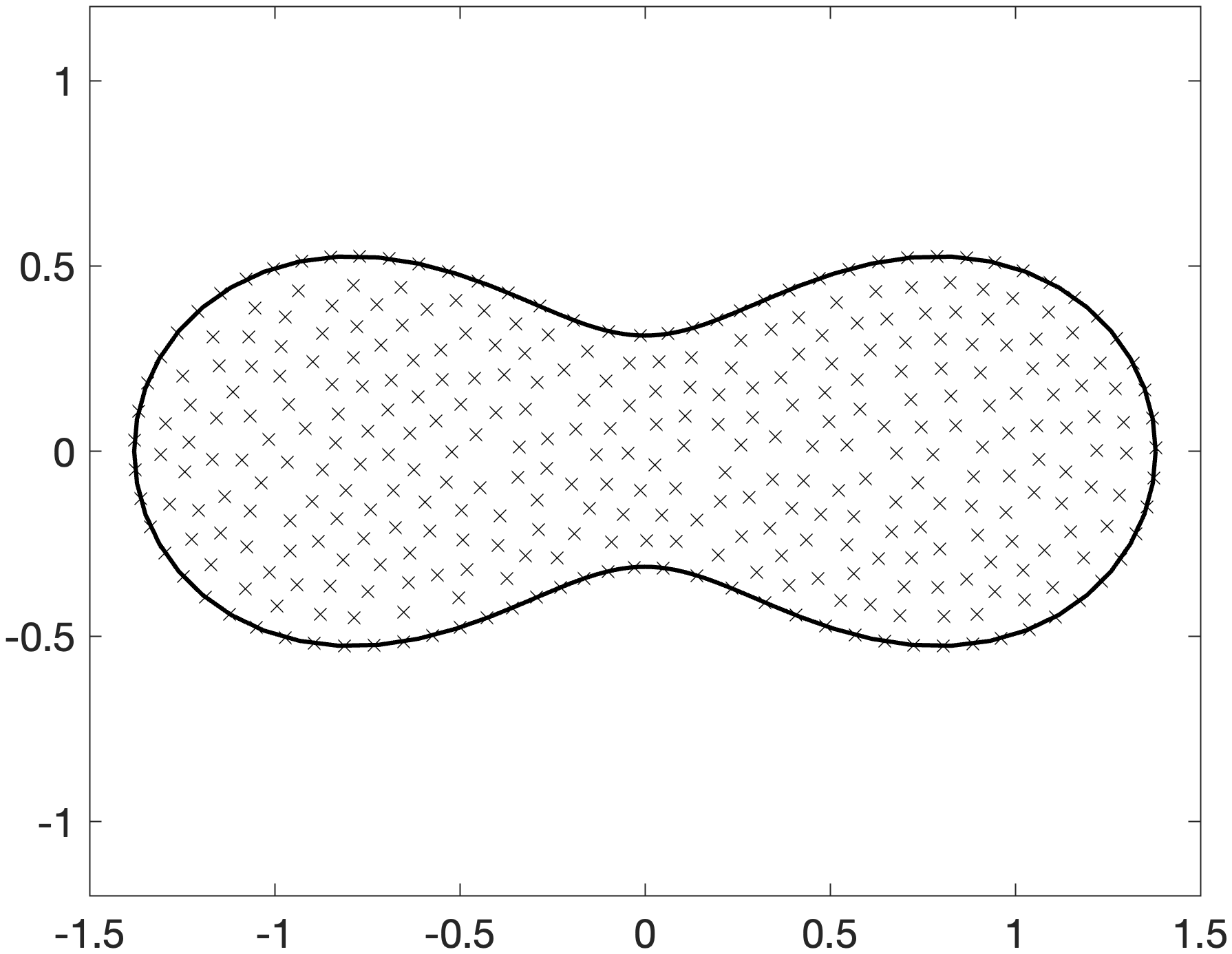}
} \hspace*{\fill} \\[1em]
\subcaptionbox{Ellipsoid $\Omega_4$.}{
	\includegraphics[width=0.22\textwidth]{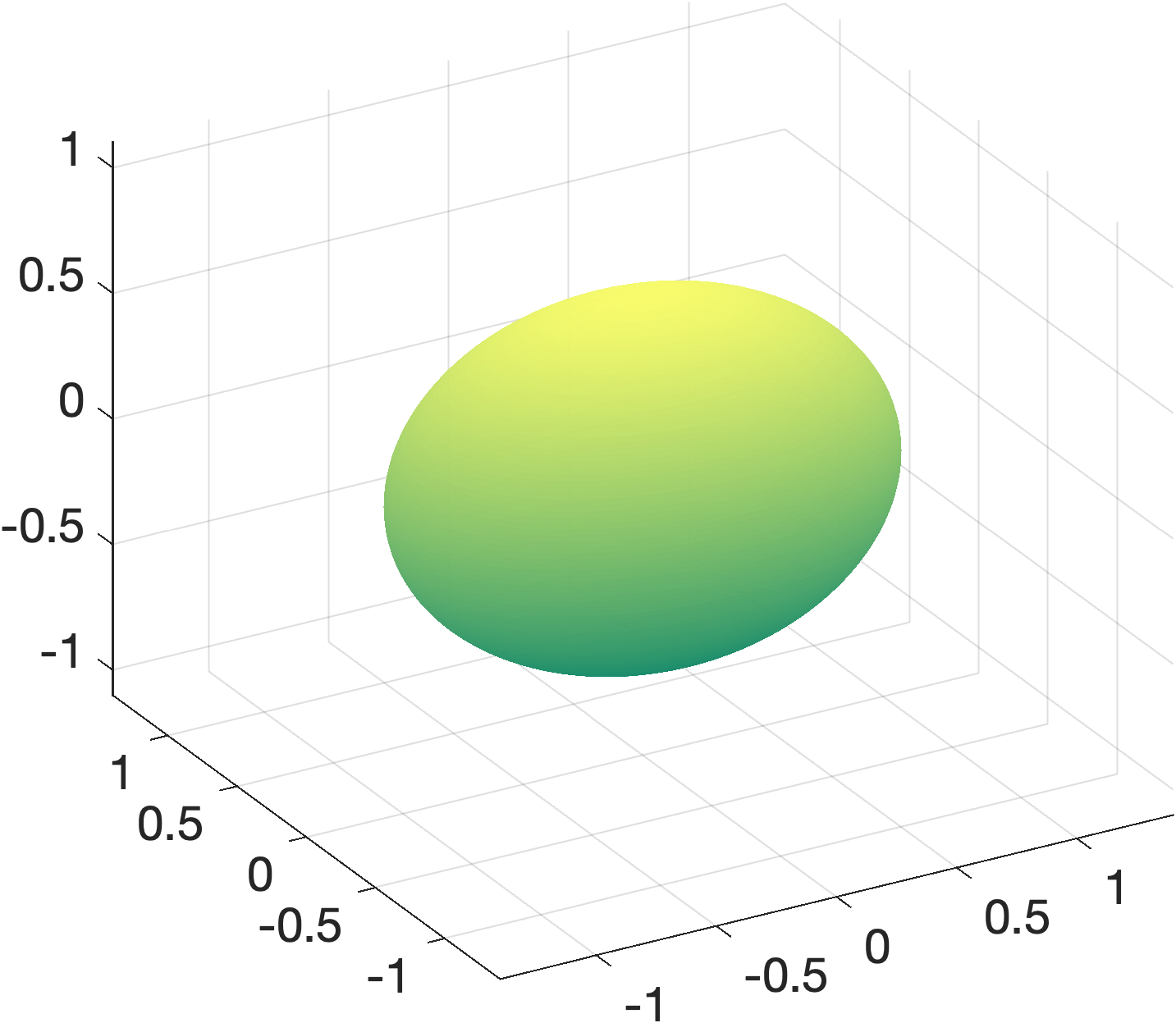}
} \hfill
\subcaptionbox{L-shaped domain $\Omega_5$.}{
	\includegraphics[width=0.22\textwidth]{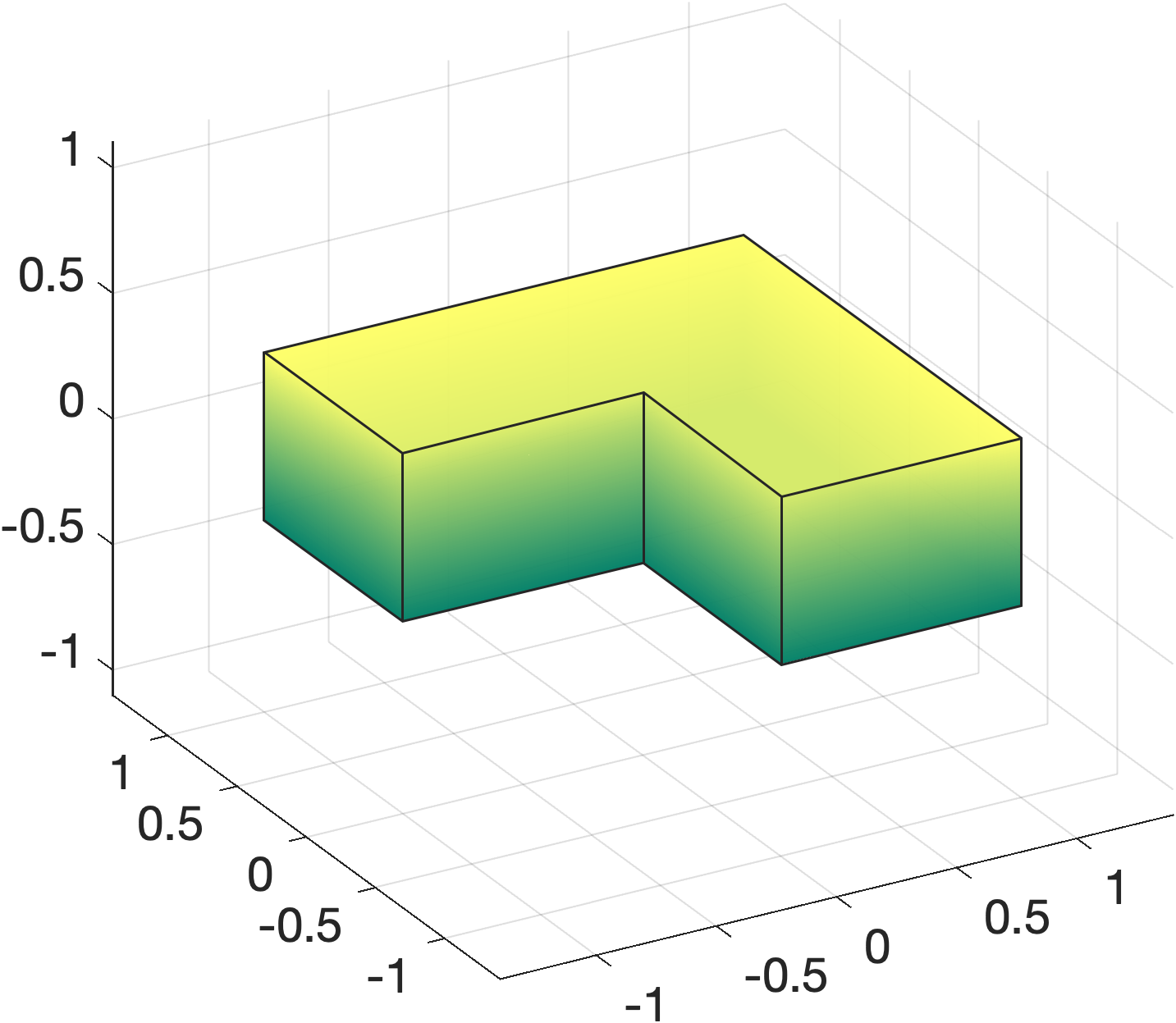}
} \hfill
\subcaptionbox{Torus $\Omega_6$.}{
	\includegraphics[width=0.22\textwidth]{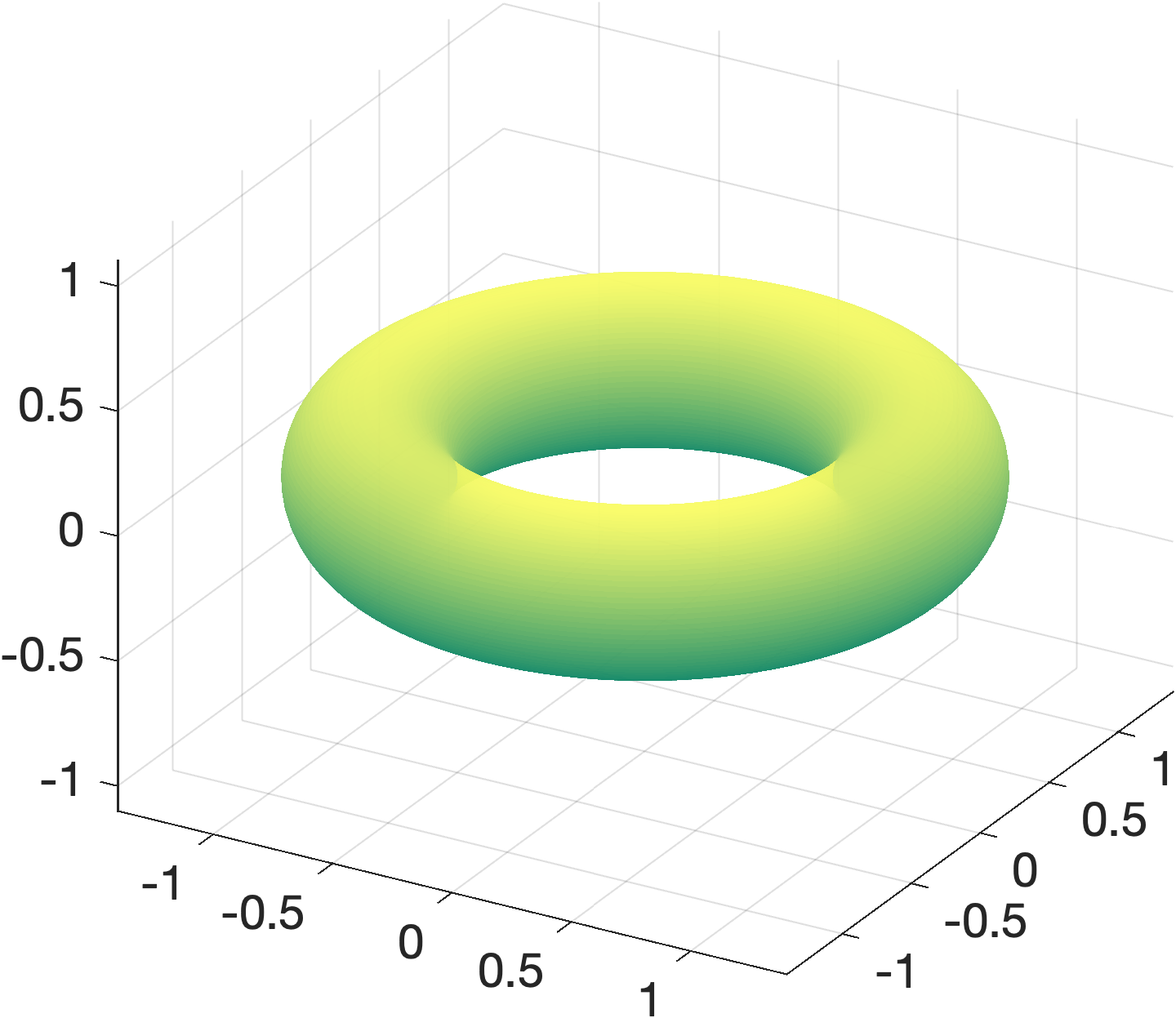}
} \hfill
\subcaptionbox{Deco-tetrahedron $\Omega_7$. \label{subfig:decotet}}{
	\includegraphics[width=0.22\textwidth]{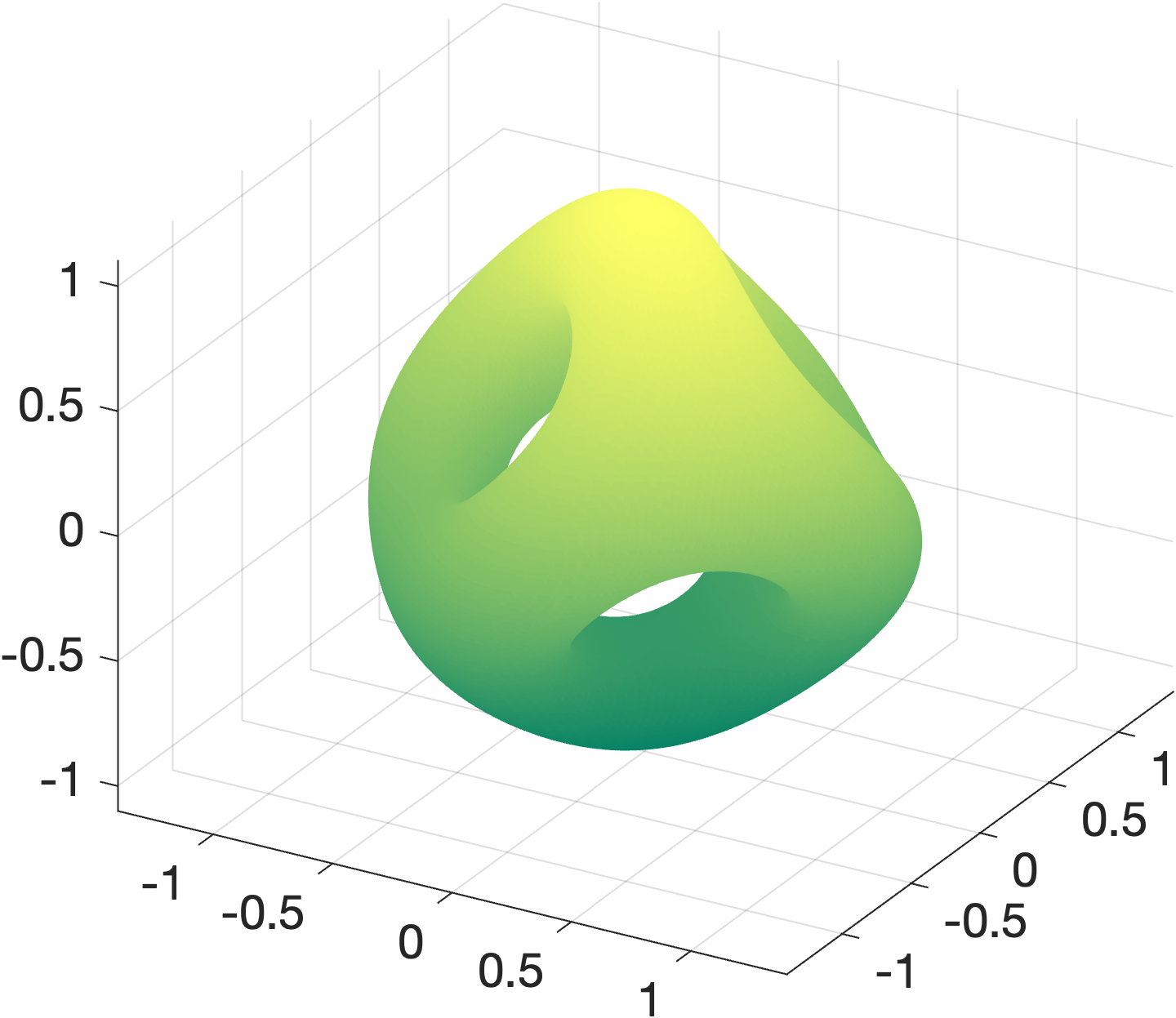}
} \\[1em]
\caption{Integration domains used in numerical tests.
The node set $Y$ generated by the advancing front method
for a closed quadrature formula with $h=0.08$ is displayed
as crosses on top of the 2D domains.}
\label{fig:domains}
\end{figure}

Domain $\Omega_2$ is the disk sector
already defined in Section~\ref{sssec:choice-LBcal}.
The other 2D domain is the region enclosed by a Cassini oval,
a quartic plane curve defined as the locus
of points such that the product of the distances
to two fixed points $(-a,0)$ and $(a,0)$,
called \emph{foci}, is a constant $b^2 \in \R_+$:
\[
\Omega_3 = \Bigl\{ (x_1,x_2) \in \R^2 \mathrel{\Big|}
\bigl((x_1+a)^2+x_2\bigr)^2
\bigl((x_1-a)^2+x_2\bigr)^2 - b^4 < 0 \Bigr\}.
\]
For $a < b < a\sqrt{2}$, the domain $\Omega_3$
has smooth boundary and is simply connected,
although it is not convex: its shape resembles
the number eight rotated sideways. %
For our numerical experiments, we have taken
$a = 0.95$ and $b = 1$ to stress its non-convex shape
and to ensure that its area is close to that of $\Omega_2$.
As far as 3D domains are concerned, $\Omega_5$ is
an L-shaped domain obtained by cutting one quadrant
from a rectangular cuboid with sides of length $(2,2,2/3)$,
and $\Omega_6$ is the solid torus with major radius $R = 1$
and minor radius $r = 0.32$.

For all tests we use node sets $X,Y,Z$ generated
by the advancing front algorithm provided by the library
NodeGenLib (see Section~\ref{sssec:choice-of-nodes}),
with $Z \subset Y$ and spacing parameter $h$ ranging
from 0.005 to 0.16 on 2D domains,
from 0.0198 to 0.1 on 3D domains using the MFD algorithm,
and from 0.0157 to 0.1 on 3D domains using BSP algorithm.
All the 2D domains have been scaled to have approximately
the same area, and all 3D domains have been scaled to have
approximately the same volume. This implies
that the advancing front algorithm generates a similar
number $N_Y$ of internal quadrature nodes for a given
spacing parameter $h > 0$, and so the plots of interior
quadrature errors can be directly compared across different
domains of the same dimension.

The exact number of quadrature nodes in $Y$ and $Z$, the size
of the system matrix $A$, and the number of its nonzero
elements are reported in Table~\ref{tab:nodes}
for the domains $\Omega_2$ and $\Omega_5$, exemplifying 2D
and 3D domains, respectively. The data for the other domains is
very similar to $\Omega_2$ or $\Omega_5$, except that the
size of $Z$ varies somewhat as it depends on the measure of
$\partial\Omega$.
Variables $m_{MFD}$ and $m_{BSP}$ denote the number
of rows of $A$ in the case of Algorithms~\ref{algo:mfd}
and \ref{algo:spline}, respectively, whereas $nnz_{\text{MFD}}$
and $nnz_{\text{BSP}}$ denote the respective number of nonzero
elements in $A$. These quantities are reported for the order $q=5$.
The ratios $m_{MFD}/N$ and $m_{BSP}/N$ are
below one, and this confirms that the linear system
$Ax=b$ is underdetermined. Numbers for $nnz_{\text{MFD}}$
and $nnz_{\text{BSP}}$ agree with the estimates in
\eqref{nnzA}.

\begin{table}[htbp!]
\caption{Number of quadrature nodes and the size and
sparsity of system matrix $A$ for the tests on the
disk sector $\Omega_2$ and L-shaped domain $\Omega_5$,
with polynomial order $q = 5$.
Variables $m_{\text{MFD}}$ and $m_{\text{BSP}}$ denote
the number of rows in $A$, and $nnz_{\text{MFD}}$
and $nnz_{\text{BSP}}$ the number of nonzero
elements in $A$, produced by MFD and BSP algorithms, 
and $N=N_Y + N_Z$ is the number of columns of $A$.}
\label{tab:nodes}
\centering
\begin{small}
\begin{filecontents*}{test-convergence-sizeA.csv}
hY,NY,NZ,NrowsA_mfd,NrowsA_bsp,sizeratioA_mfd,sizeratioA_bsp,nnzA_mfd,nnzA_bsp
0.16,99,42,88,97,0.624113475177305,0.687943262411348,7442,6906
0.08,358,84,298,183,0.67420814479638,0.414027149321267,24584,21690
0.04,1360,168,1094,423,0.715968586387435,0.276832460732984,87513,75568
0.02,5285,336,4200,1257,0.74719800747198,0.223625689379114,328940,279831
0.01,20834,672,16392,3893,0.762205896029015,0.181019250441737,1275094,1071940
0.005,82748,1343,64931,13654,0.772151597673948,0.162371716354901,5014829,4197855
0.1,2241,1068,1974,1387,0.596554850407978,0.419159867029314,514071,619302
0.07937005259841,4203,1674,3574,1639,0.60813340139527,0.278883784243662,950932,1214908
0.0629960524947437,7971,2617,6565,2689,0.620041556479033,0.253966754816774,1780880,2423874
0.05,15339,4108,12327,3592,0.633876690492107,0.184707152774207,3389160,4863433
0.039685026299205,29757,6476,23419,6073,0.646344492589628,0.167609637623161,6514127,9815254
0.0314980262473718,58108,10221,45019,9073,0.658856415284872,0.132784030206794,12621388,19578824
0.025,114065,16139,87282,15709,0.670348069183742,0.12064913520322,24614849,39281133
0.0198425131496025,224776,25513,170405,25093,0.680832957101591,0.100256103943841,48248545,78994325
\end{filecontents*}
\pgfplotstabletypeset[
	create on use/domain/.style={
        create col/set list={
        	$\Omega_2$, $\Omega_2$, $\Omega_2$,
        	$\Omega_2$, $\Omega_2$, $\Omega_2$,
        	$\Omega_5$, $\Omega_5$, $\Omega_5$, $\Omega_5$,
        	$\Omega_5$, $\Omega_5$, $\Omega_5$, $\Omega_5$
        }
    },
	columns/domain/.style={
		column name = {Domain},
		column type = {c},
		string type
	},
	every row no 5/.append style={
		after row=\midrule
	},
	columns={
		domain, hY, NY, NZ,
		NrowsA_mfd, sizeratioA_mfd, nnzA_mfd,
		NrowsA_bsp, sizeratioA_bsp, nnzA_bsp
	}]
{test-convergence-sizeA.csv}
\end{small}
\end{table}

For most tests, the minimum-norm solution to the
underdetermined system $Ax=b$ is found with
SuiteSparseQR, as described
in Section~\ref{sssec:choice-norm}.
The range of the spacing parameter $h$ ensures
that the sparse QR factors of $A^T$ computed by
\texttt{spqr\_solve} fit into the 64 GB of RAM
of the system used to run the tests.
Even though fill-in is mitigated by reordering algorithms
such as colamd (column approximate minimum degree permutation),
space and time complexity of the sparse QR factorization
are superlinear, and this prevents values of $h$
below $\approx 0.02$ from being used in our 3D tests.

When $A$ has full rank, the minimum-norm solution $x$
can be found in the form $x = A^T y$, which leads
to the so-called \emph{normal equations
of the second kind} $A A^T y = b$.
The matrix $A$ assembled by Algorithms~\ref{algo:mfd}
and \ref{algo:spline} typically does not have full rank,
so this made us consider the regularized system
\[
(A A^T + \omega I) \tilde{y} = b,
\]
with a sufficiently small positive constant $\omega$.
The matrix $A A^T + \omega I$ inherits sparsity from $A$,
and admits a Cholesky decomposition because it is
symmetric positive definite by construction.
Remarkably, in the BSP case,
a sparse Cholesky factorization with fill-reducing reordering
of $A A^T + \omega I$, such as the one computed by the
command \texttt{lchol} in the CHOLMOD
library \cite{chen2008algorithm}, delivers a Cholesky
factor whose number of nonzero elements is only at most 50\%
larger than $nnz_{\text{BSP}}$, and so is empirically a linear
function of $N_Y + N_Z$. This implies that $\tilde{y}$
can be computed with
time and space complexity linear in the number
of quadrature nodes, which is asymptotically optimal.
The coefficient $\omega$ is chosen as %
$4 \cdot 10^{-16}$ multiplied by the smallest power
of two such that $A A^T + \omega I$ is numerically
positive definite, i.e.\ so that the command \texttt{lchol}
can find the Cholesky factor successfully.
Values of $\omega$ around $10^{-15}$ are typical.
The same asymptotic reduction in fill-in during
factorization is not observed in the MFD case,
an important difference for which we currently have no explanation.

Since the perturbation of size $\omega$ affects
the accuracy of the weights $w,v$ and
causes unwanted saturation of the quadrature errors
at around $10^{-10}$, see for example Figure~\ref{subfig:saturation},
we decided to use the Cholesky-based
solver only in the BSP case for 3D domains.
For all other tests, we use the more expensive
but also more accurate QR-based solver.
This is the reason why, on 3D domains, a slightly
wider range of spacing parameters $h$ is considered
in the BSP case compared to the MFD case.

Results for each test domain are presented in 
Figures~\ref{fig:disk-sector} to \ref{fig:torus}.
Relative quadrature errors for the polynomial orders 
$q$ between 4 and 8
are shown as functions of the spacing parameter $h$.
Results obtained by Algorithms~\ref{algo:mfd}
and \ref{algo:spline} are shown
in logarithmic scale side by side, with a common
range on the horizontal and vertical axes to aid
visualization and comparisons.
Reference slopes corresponding to different
convergence orders $q-1$ are shown using dotted lines.
The center $x_R$ of the Runge function $f_1$ for each test
domain is specified in the caption of its corresponding figure,
and is chosen so that $x_R \in \Omega$.
Boundary quadrature errors for the test function $g_2$ were omitted
for the sake of brevity, as these values are very similar to the
ones for $g_1$ in all cases.
Just like the numerical tests in
Section~\ref{ssec:validation-of-parameters},
relative quadrature errors are computed as the RMS
over multiple seeds, according to the formula \eqref{eq:err-rms}.
For the tests in this section, 8 distinct seeds were used,
a value large enough to smoothen the convergence
curves to the point where the empirical orders of
convergence (EOC) of the methods for different
$q$ can be estimated reliably using least-squares
linear regression in logarithmic scale (see the EOC values
in the legends of Figures~\ref{fig:disk-sector} to \ref{fig:torus}).
We remark, however, that high-order convergence would still be
evident even for a single seed, and that the
averaging step is not required to get accurate
and reliable results in applications.

Note that quadrature errors %
are dropped from the plots
whenever the linear system $Ax=b$ is overdetermined,
i.e.\ the matrix $A$ has more rows than columns.
Moreover, in all MFD tests, values are dropped from
the plots whenever $n_L > N_X$, that is,
when the target size of the sets of influence for
the discretization of $\Lcal$ exceeds the size of the set $X$.
These two measures can be understood as rough safeguards
against ill-chosen parameters $h$ and $q$,
and are only needed on the coarsest node distributions.

We observe in all the figures a high order of convergence
for both methods MFD and BSP, including the non-smooth domains
$\Omega_2$ and $\Omega_5$, with the slopes of the error curves
increasing with $q$ and almost always exceeding the expected
order $q-1$. This behavior demonstrates superconvergence with
respect to our theoretical upper bounds
\eqref{eq:fbamfd}, \eqref{eq:gbamfd},
\eqref{eq:fbaBS}, and \eqref{eq:gbaBS}.
In the 2D setting, errors close to
the accuracy limits of double-precision floating-point
numbers are reached in both tests, with values
around $10^{-15}$ being typical for the highest order
methods on the largest number of nodes considered for the tests.
To reach such small errors in 2D tests, the rank detection
tolerance used by SuiteSparseQR (see
Section~\ref{sssec:choice-norm})
had to be manually set to $10^{-15}$, because with default
tolerances the errors would plateau around $10^{-12}$.
In the 3D setting, where such high precision is not reached,
the rank detection tolerance used by SuiteSparseQR
was set to $10^{-12}$, a choice motivated by better errors
and stability constants for coarser sets of nodes.

Concerning stability constants, the worst-case values
of $K_w$ and $K_v$ across all test domains and all seeds
are plotted for $q=5$ and $q=7$ as a function of $h$ for
both MFD and BSP in Figure~\ref{fig:maxstab}.
The figure confirms the overall stability of the quadrature
formulas produced by the two algorithms, and show that stability
constants improve as $h$ or $q$ are reduced, with the only
exception of a mild increase for small values
of $h$ in 2D, for which the integration errors
are however close to machine precision, see the convergence plots
in Figures~\ref{fig:disk-sector} and \ref{fig:cassini-oval}.
For the smallest values of $h$ considered in the tests,
$K_w$ never exceeds 5, and $K_v$ is very close to 1.
Moreover, for $q=5$, all stability constants are quite small
even for the coarsest sets of nodes.
In all tests, the boundary quadrature formula
is significantly more stable than the interior one,
with values barely above 1 for $q = 5$, regardless of $h$.

Comparison of the results for MFD and BSP reveals
similar behavior between the two approaches, with comparable
quadrature errors across all tests for methods with the same
polynomial order $q$, and comparable stability constants,
except for the case $q = 8$ and $h > 0.05$.
The main difference in accuracy, although modest,
is that the MFD method performs better than BSP
on 3D domains with non-smooth boundary such as $\Omega_5$,
whereas the BSP method performs better than MFD
for the integration of the Franke function
(the smoother one among $f_1$ and $f_2$) on the torus.

Note that the numbers in the columns
$m_{\text{MFD}}$ and $m_{\text{BSP}}$ of Table~\ref{tab:nodes}
indicate that the system matrices for MFD
are much larger than those for BSP, which increases the fill-in
produced by direct solvers.
On the one hand, the difference in memory usage can
be quite large, especially if the Cholesky-based solver is used.
On the other hand, as can be seen in Table~\ref{tab:nodes}
and the estimate \eqref{nnzA}, 
the total number of nonzeros in the larger system matrix for MFD
is significantly smaller compared to the BSP matrix in the more
computationally demanding 3D setting, even though we have not
attempted to optimize the size of the sets of influence in MFD, as
discussed at the end of Section~\ref{sssec:spacing}. This may
bring computational advantages for MFD if an effective
preconditioning strategy is devised to make iterative
methods competitive with direct solvers for the solution
of the minimum-norm problem.
Another advantage of MFD is that values of $h$ larger
than 0.04 can be used with high polynomial orders
such as $q=6,7,8$ without a significant increase
in the stability constants.

In the case of Algorithm~\ref{algo:spline}, the support
of some B-splines may have small intersection with $\Omega$,
and this is known to be a source of instability in
immersed methods, see e.g.\
\cite{chu2022stabilization,davydov2014two,de2023stability}. 
In our numerical experiments, however, we have not
observed any instability as $h \to 0$,
and a likely explanation is that the minimization of 2-norm has
a regularizing effect on the computation of quadrature weights.
Nevertheless, it is possible that stabilization techniques
from the literature on immersed methods could further
improve the accuracy and stability of the BSP approach.

\begin{figure}[p]
\centering
\subcaptionbox{Errors for $f_1$ with MFD weights.}{
	\includegraphics[width=0.47\textwidth]{
	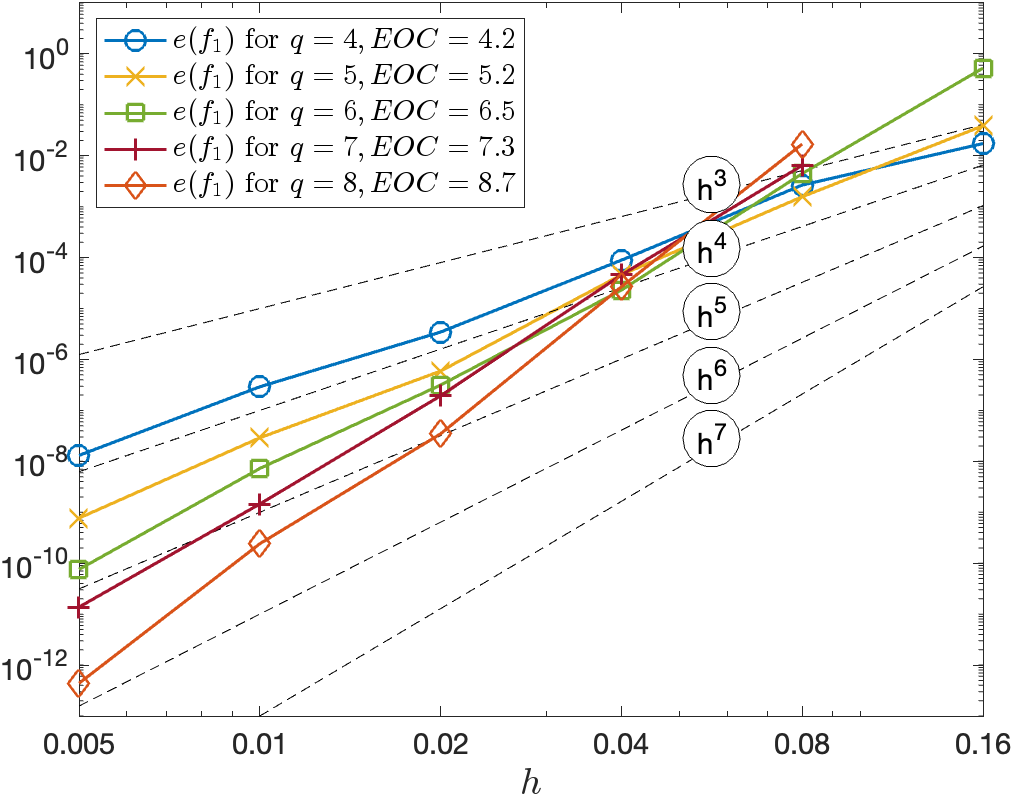}
} \hfill
\subcaptionbox{Errors for $f_1$ with BSP weights.}{
	\includegraphics[width=0.47\textwidth]{
	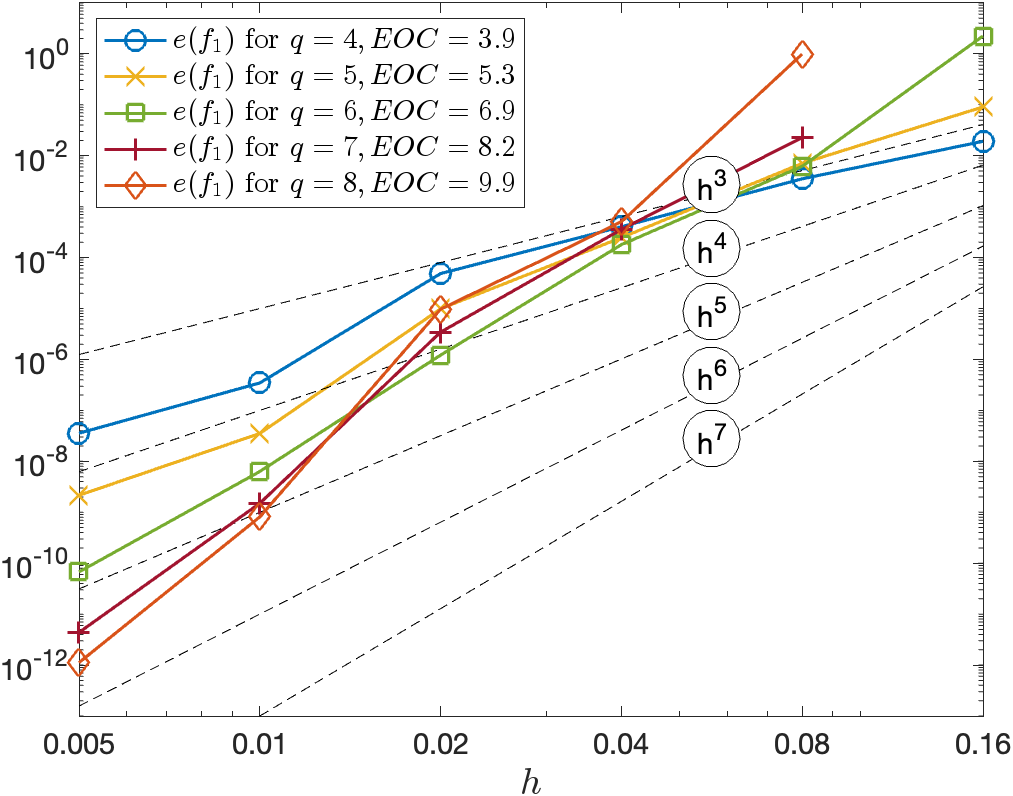}
} \\[3em]
\subcaptionbox{Errors for $g_1$ with MFD weights.}{
	\includegraphics[width=0.47\textwidth]{
	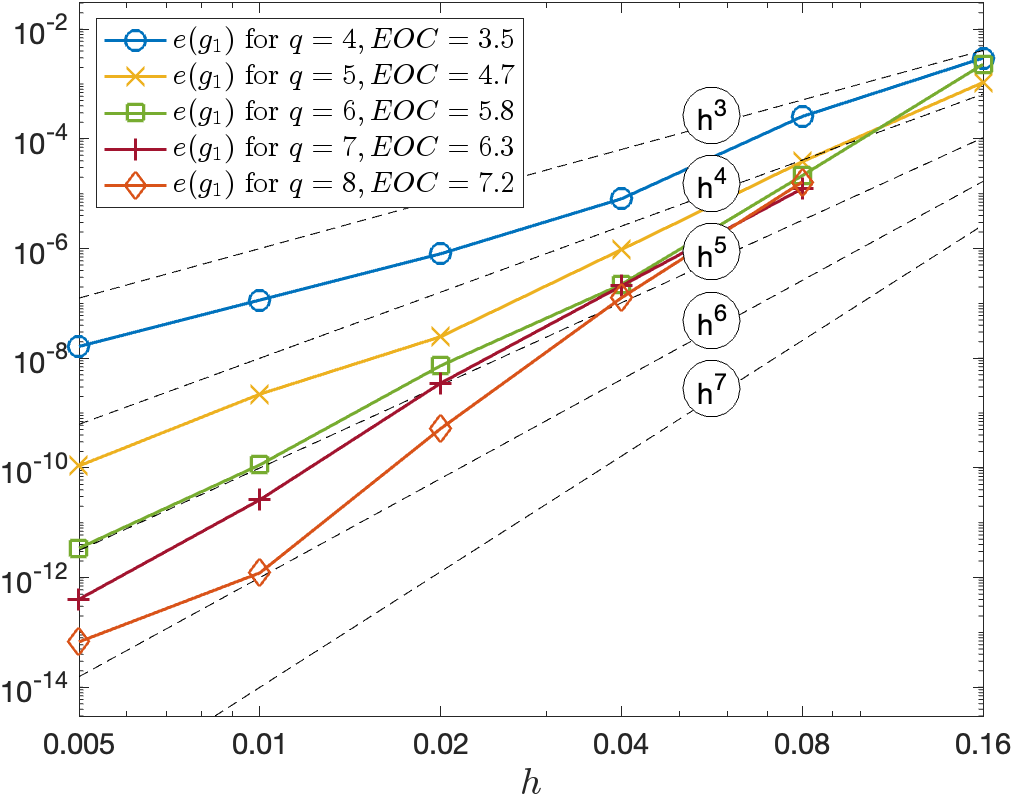}
} \hfill
\subcaptionbox{Errors for $g_1$ with BSP weights.}{
	\includegraphics[width=0.47\textwidth]{
	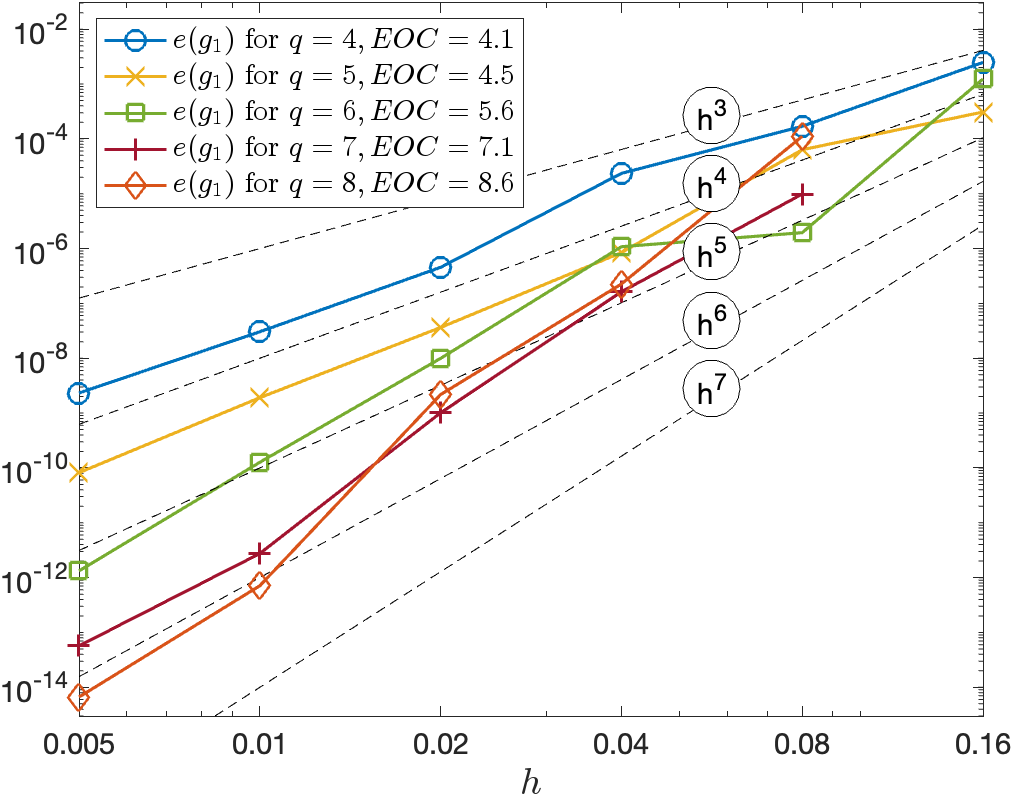}
} \\[3em]
\subcaptionbox{Errors for $f_2$ with MFD weights.}{
	\includegraphics[width=0.47\textwidth]{
	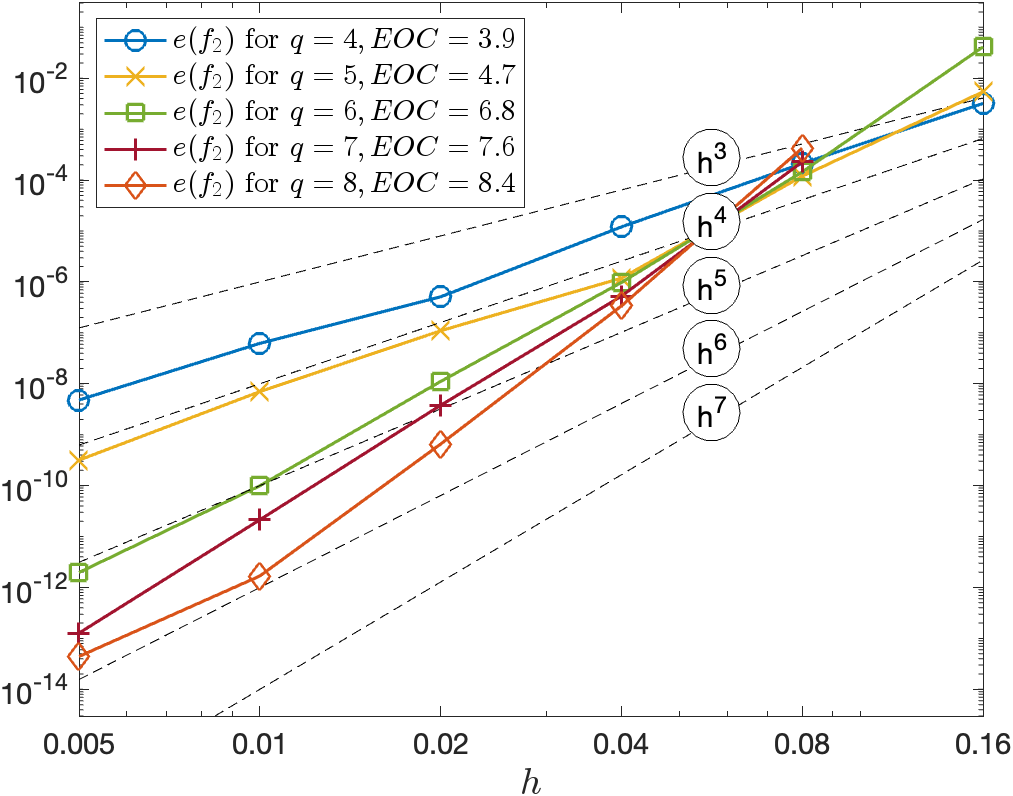}
} \hfill
\subcaptionbox{Errors for $f_2$ with BSP weights.}{
	\includegraphics[width=0.47\textwidth]{
	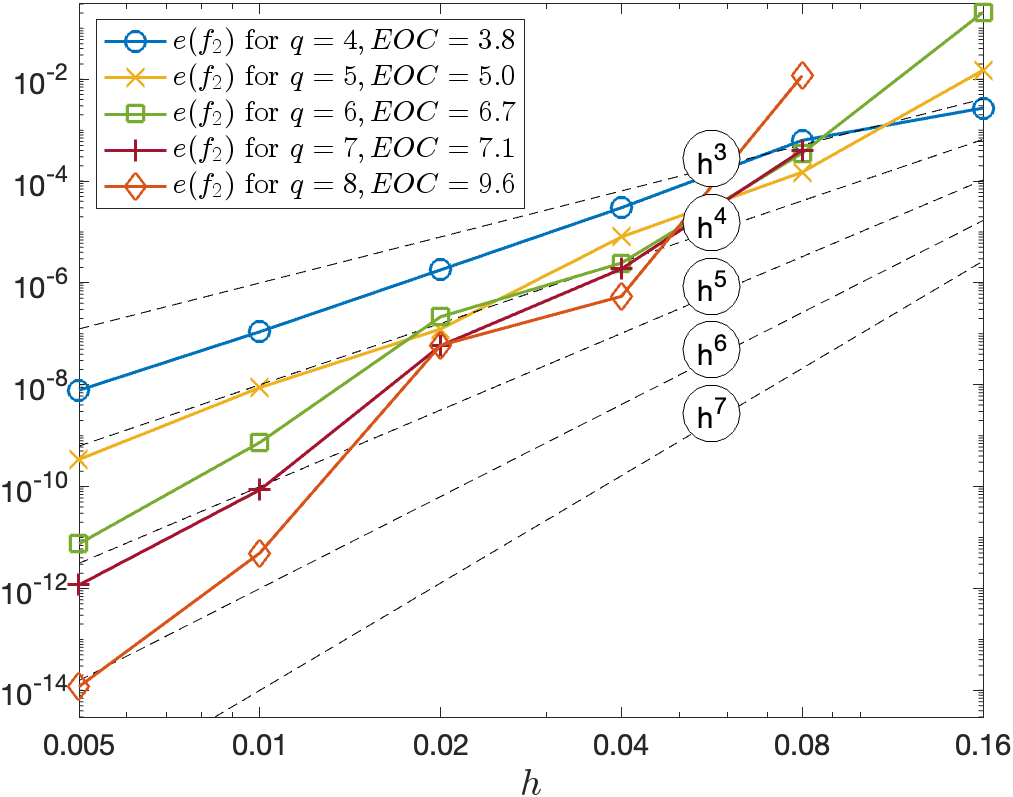}
} \\[1em]
\caption{RMS relative quadrature errors on 2D domain
$\Omega_2$ (disk sector) and its boundary, averaged over
8 distinct seeds and plotted as a function of $h$.
Nodes are generated with an advancing front method.
Runge function $f_1$ is centered at
$x_R = (\cos(3\pi/4)/2,\sin(3\pi/4)/2)$.
Results of Algorithms~\ref{algo:mfd} (MFD) and \ref{algo:spline} (BSP)
are compared for polynomial orders $q=4,\ldots,8$.
Reference slopes of convergence order $q-1$ are shown
using dotted lines. Estimated order of convergence EOC
is obtained by linear least-squares fitting.}
\label{fig:disk-sector}
\end{figure}

\begin{figure}[p]
\centering
\subcaptionbox{Errors for $f_1$ with MFD weights.}{
	\includegraphics[width=0.47\textwidth]{
	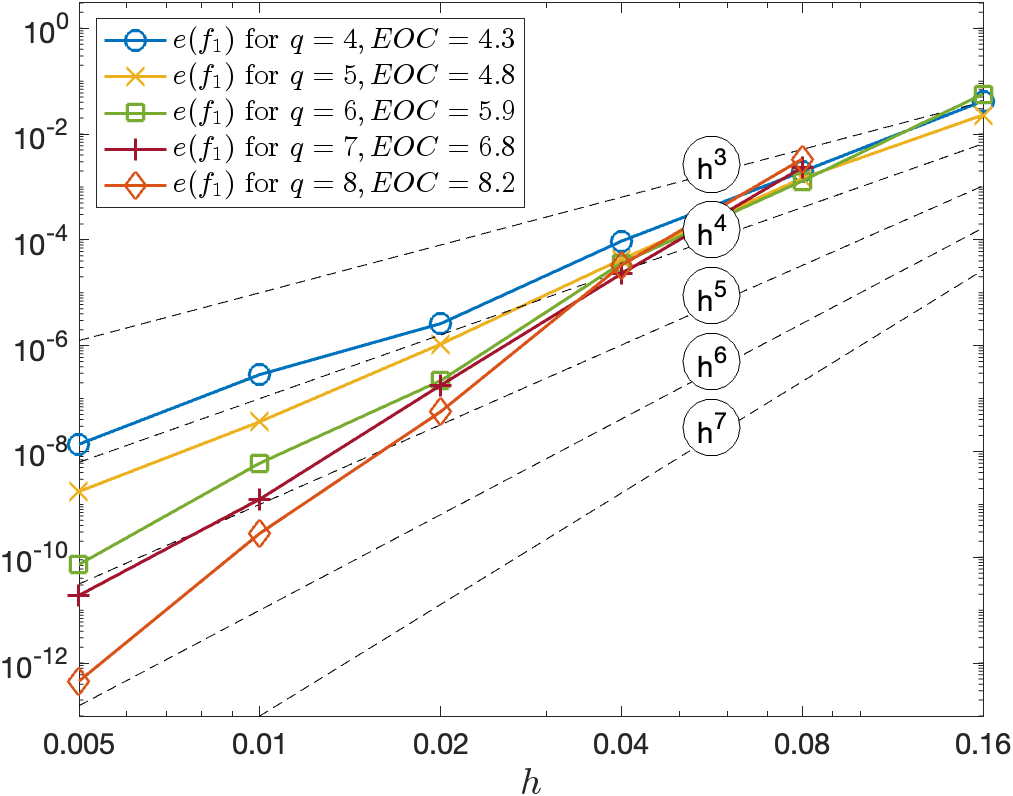}
} \hfill
\subcaptionbox{Errors for $f_1$ with BSP weights.}{
	\includegraphics[width=0.47\textwidth]{
	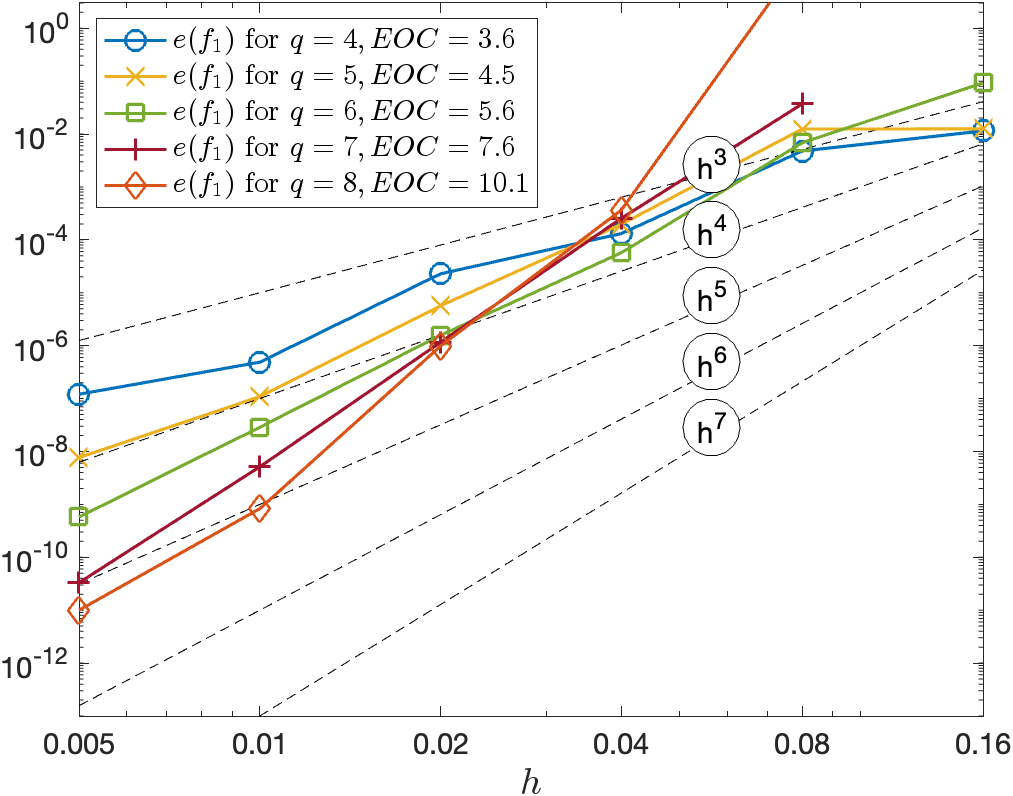}
} \\[3em]
\subcaptionbox{Errors for $g_1$ with MFD weights.}{
	\includegraphics[width=0.47\textwidth]{
	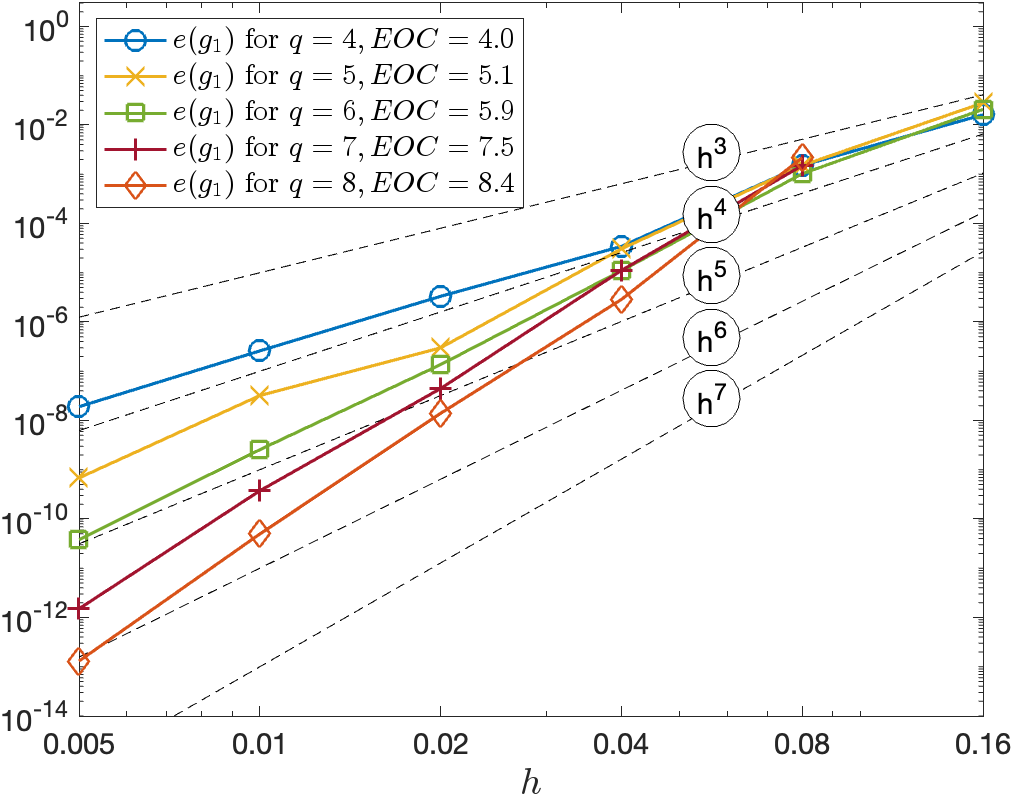}
} \hfill
\subcaptionbox{Errors for $g_1$ with BSP weights.}{
	\includegraphics[width=0.47\textwidth]{
	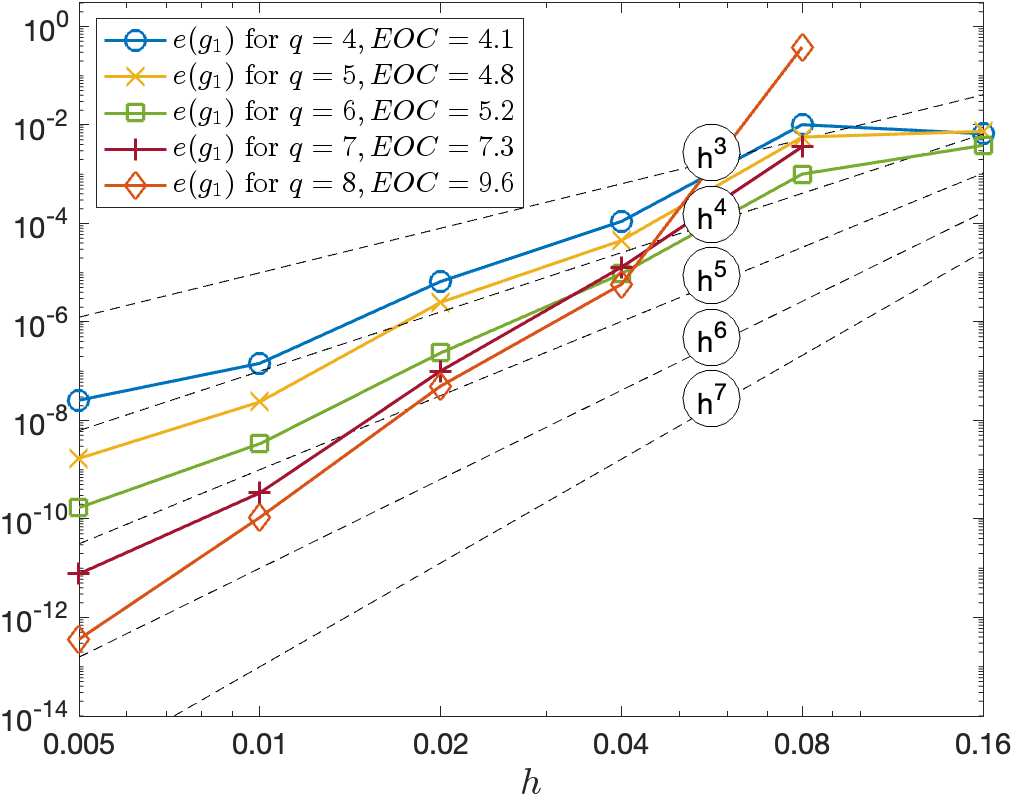}
} \\[3em]
\subcaptionbox{Errors for $f_2$ with MFD weights.}{
	\includegraphics[width=0.47\textwidth]{
	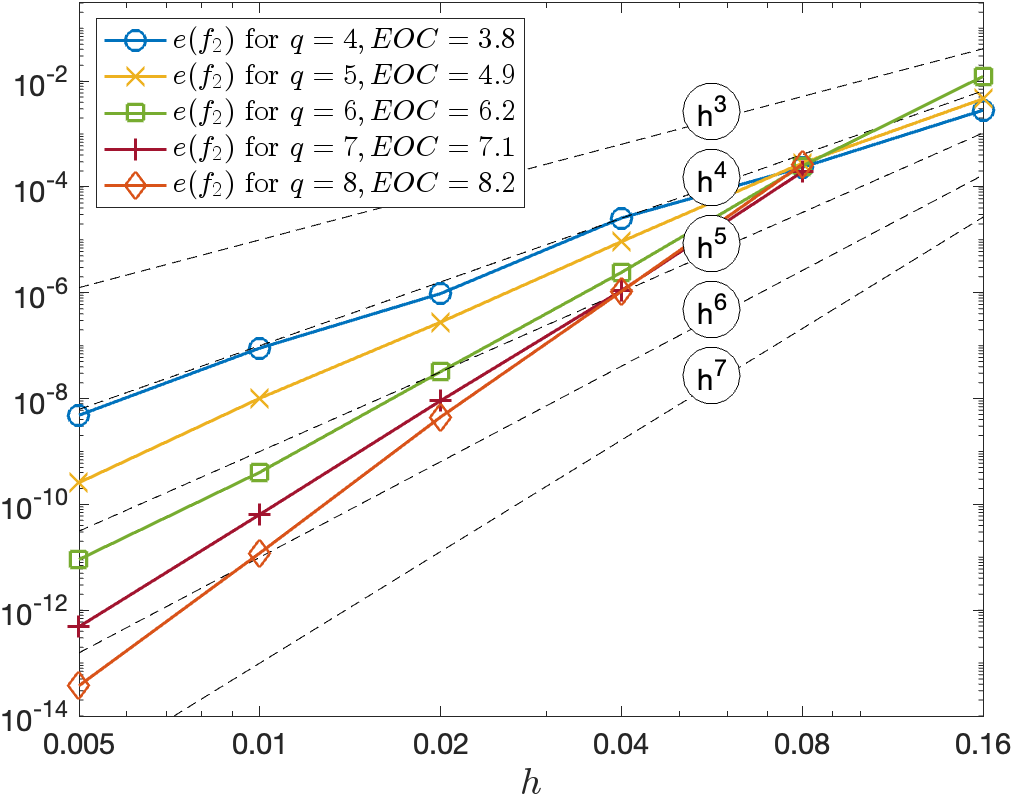}
} \hfill
\subcaptionbox{Errors for $f_2$ with BSP weights.}{
	\includegraphics[width=0.47\textwidth]{
	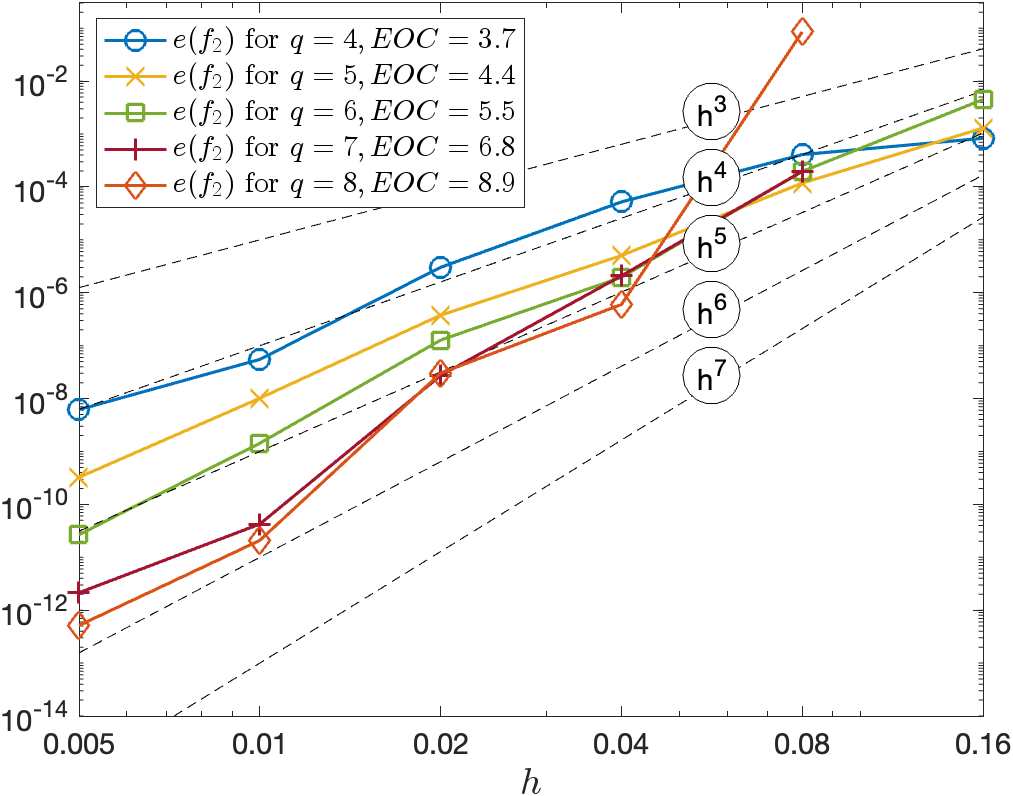}
} \\[1em]
\caption{RMS relative quadrature errors on 2D domain
$\Omega_3$ (Cassini oval) and its boundary, averaged over
8 distinct seeds and plotted as a function of $h$.
Nodes are generated with an advancing front method.
Runge function $f_1$ is centered at $x_R = (0,0)$.
Results of Algorithms~\ref{algo:mfd} (MFD) and \ref{algo:spline} (BSP)
are compared for polynomial orders $q=4,\ldots,8$.
Reference slopes of convergence order $q-1$ are shown
using dotted lines. Estimated order of convergence EOC
is obtained by linear least-squares fitting.}
\label{fig:cassini-oval}
\end{figure}

\begin{figure}[p]
\centering
\subcaptionbox{Errors for $f_1$ with MFD weights.}{
	\includegraphics[width=0.47\textwidth]{
	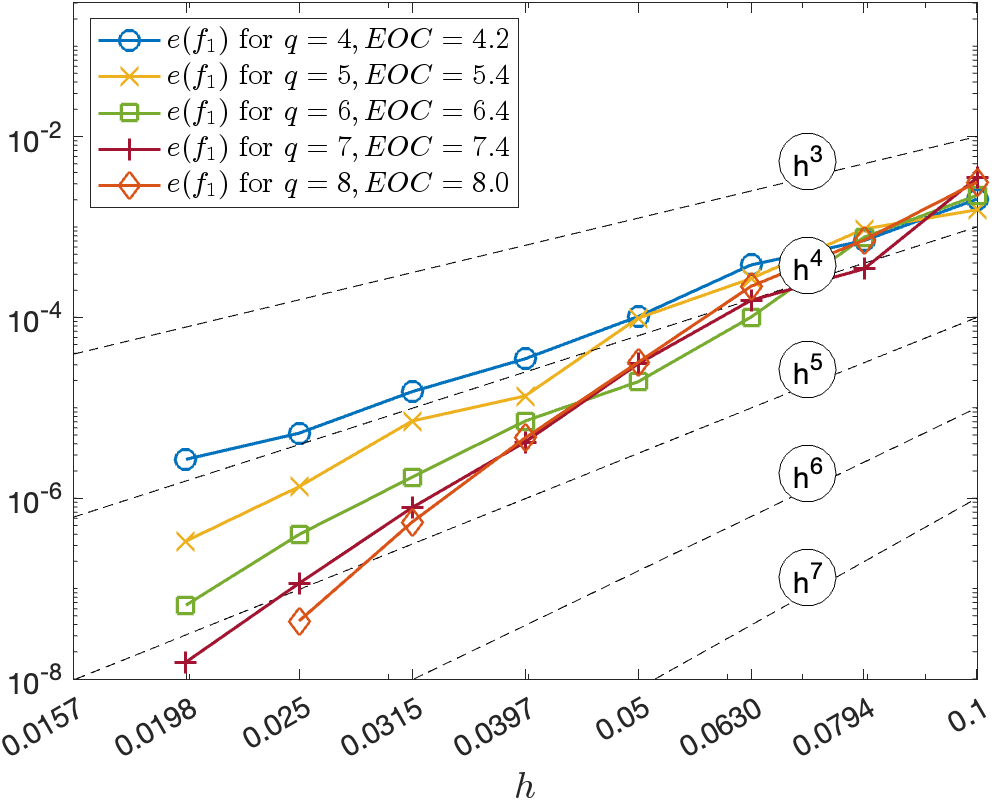}
} \hfill
\subcaptionbox{Errors for $f_1$ with BSP weights.}{
	\includegraphics[width=0.47\textwidth]{
	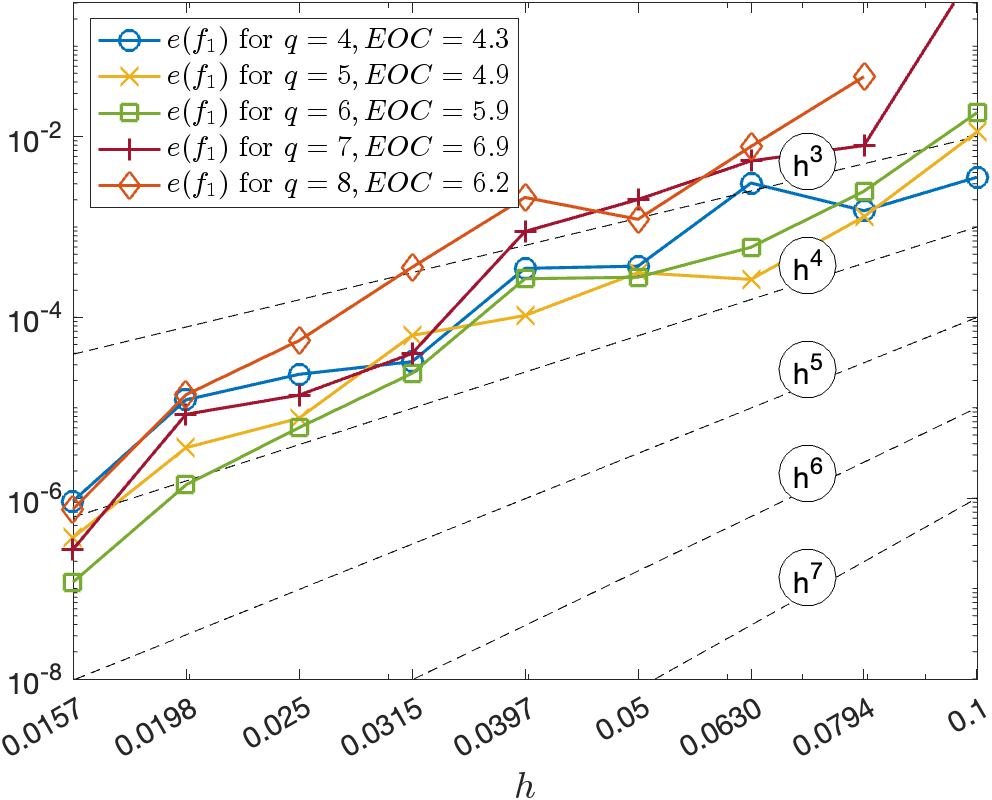}
} \\[3em]
\subcaptionbox{Errors for $g_1$ with MFD weights.}{
	\includegraphics[width=0.47\textwidth]{
	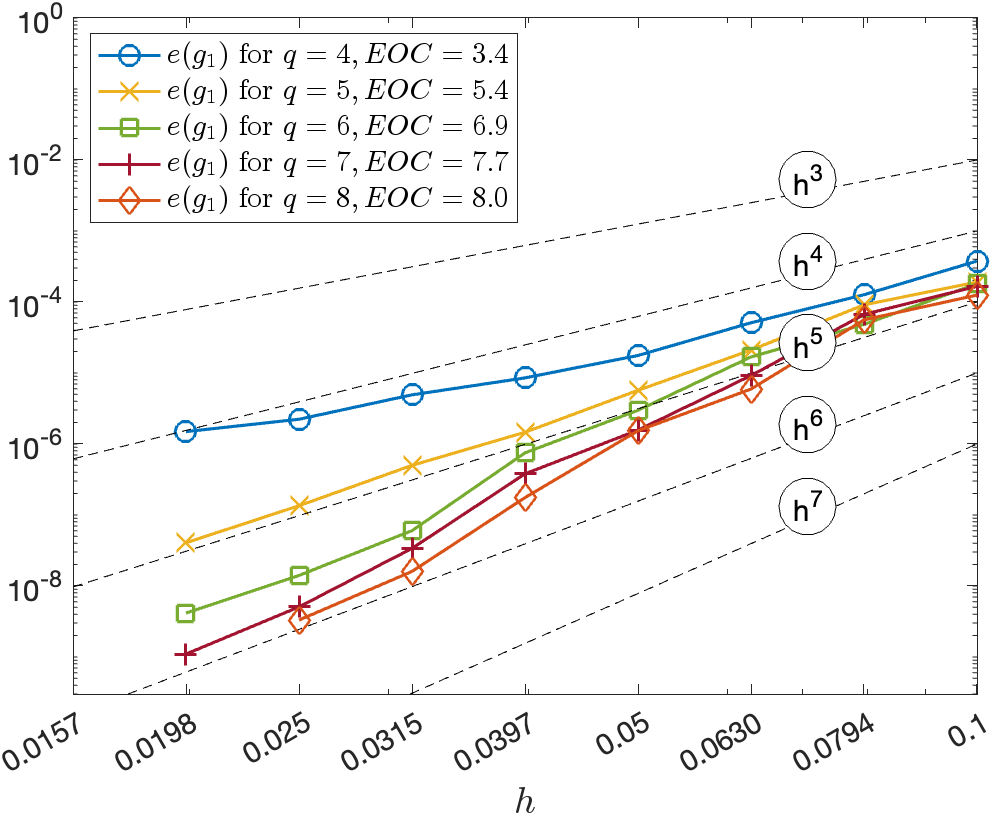}
} \hfill
\subcaptionbox{Errors for $g_1$ with BSP weights.}{
	\includegraphics[width=0.47\textwidth]{
	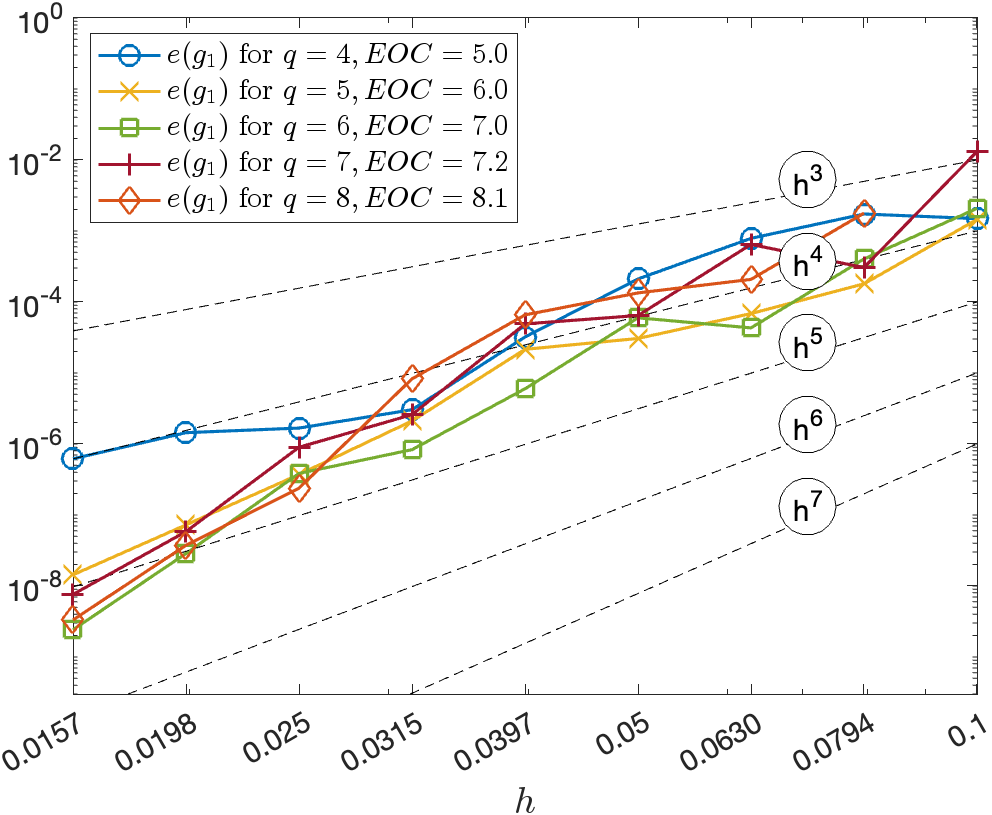}
} \\[3em]
\subcaptionbox{Errors for $f_2$ with MFD weights.}{
	\includegraphics[width=0.47\textwidth]{
	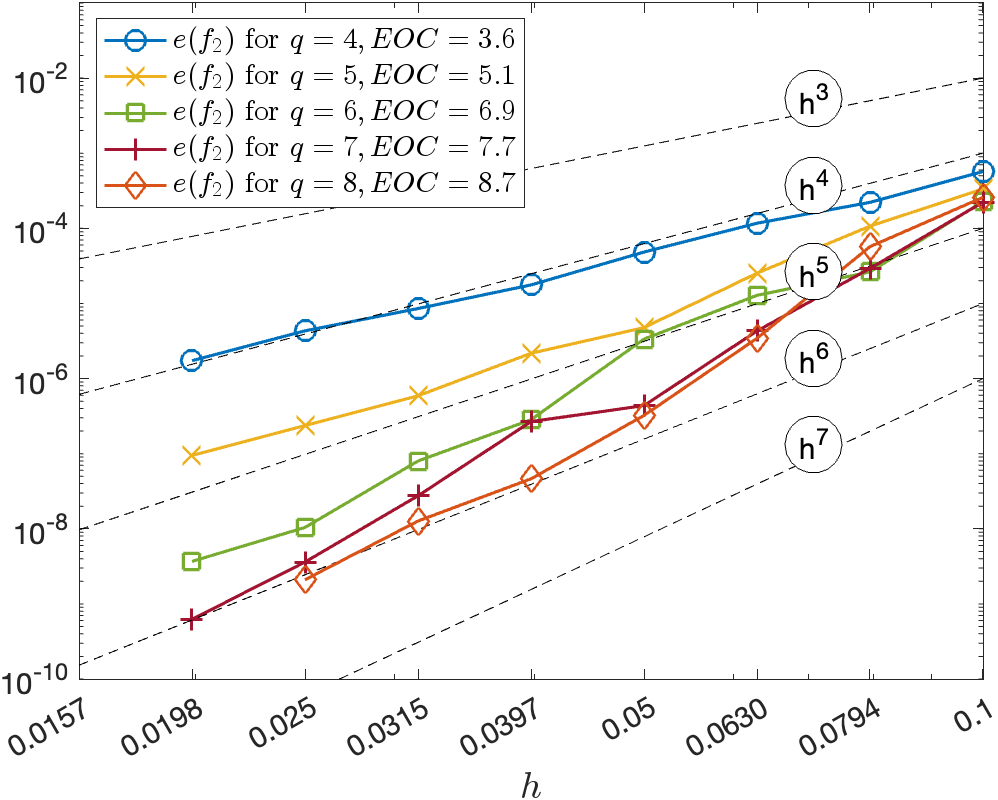}
} \hfill
\subcaptionbox{Errors for $f_2$ with BSP weights.}{
	\includegraphics[width=0.47\textwidth]{
	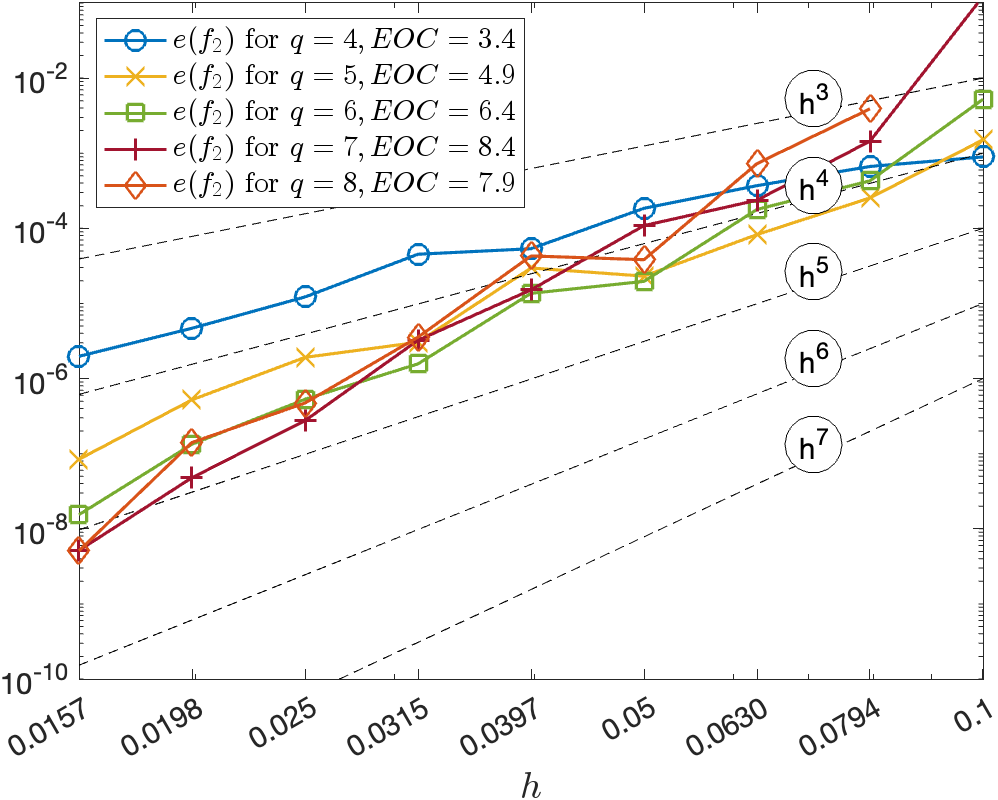}
} \\[1em]
\caption{RMS relative quadrature errors on 3D
domain $\Omega_5$ (L-shaped) and its boundary, averaged over
8 distinct seeds and plotted as a function of $h$.
Nodes are generated with an advancing front method.
Runge function $f_1$ is centered at $x_R = (1/2,1/2,0)$.
Results of Algorithms~\ref{algo:mfd} (MFD) and \ref{algo:spline} (BSP)
are compared for polynomial orders $q=4,\ldots,8$.
Reference slopes of convergence order $q-1$ are shown
using dotted lines. Estimated order of convergence EOC
is obtained by linear least-squares fitting.}
\label{fig:ell-bricks}
\end{figure}

\begin{figure}[p]
\centering
\subcaptionbox{Errors for $f_1$ with MFD weights.}{
	\includegraphics[width=0.47\textwidth]{
	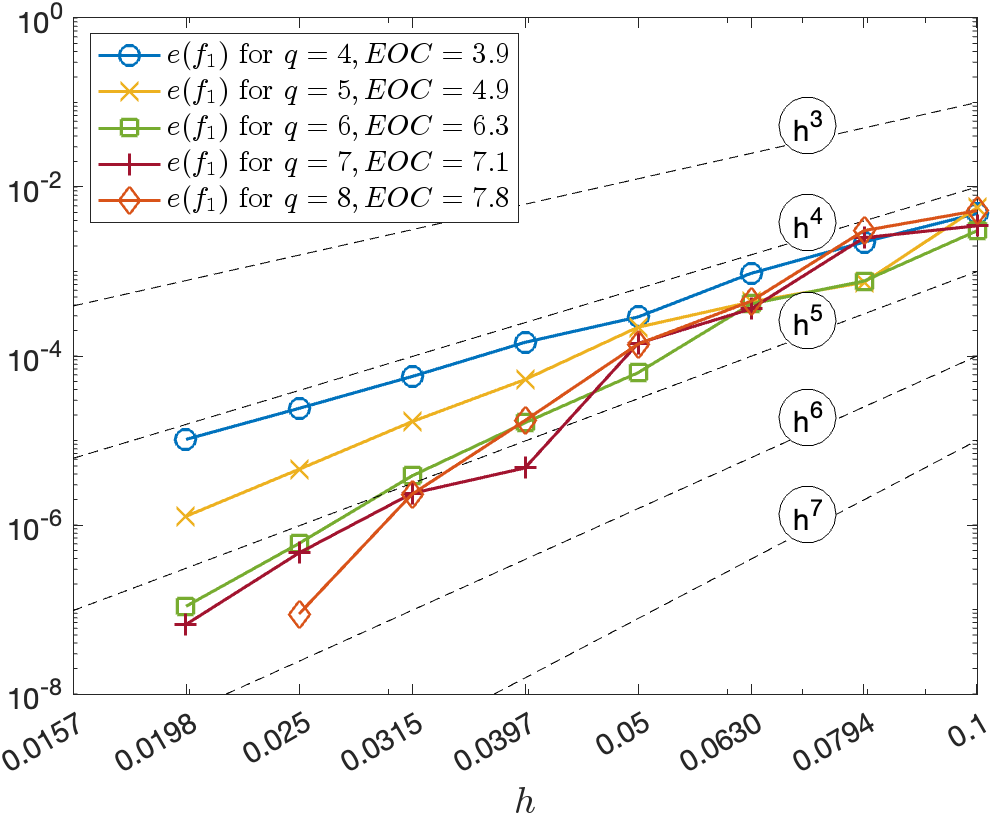}
} \hfill
\subcaptionbox{Errors for $f_1$ with BSP weights.}{
	\includegraphics[width=0.47\textwidth]{
	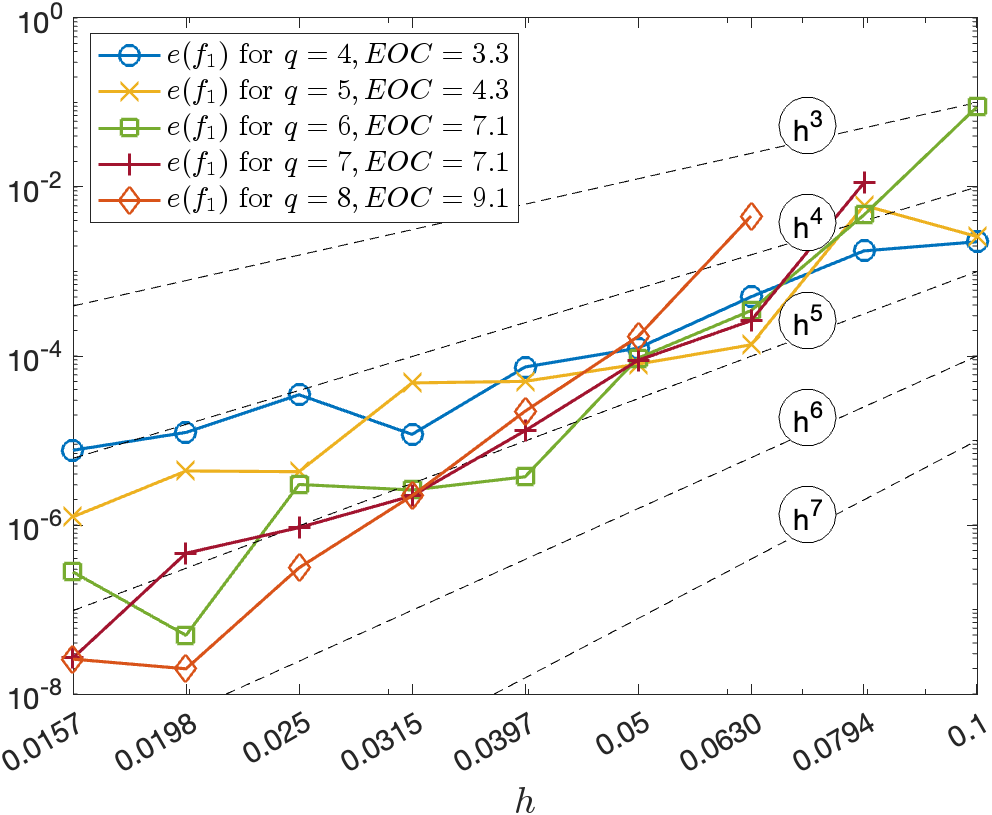}
} \\[3em]
\subcaptionbox{Errors for $g_1$ with MFD weights.}{
	\includegraphics[width=0.47\textwidth]{
	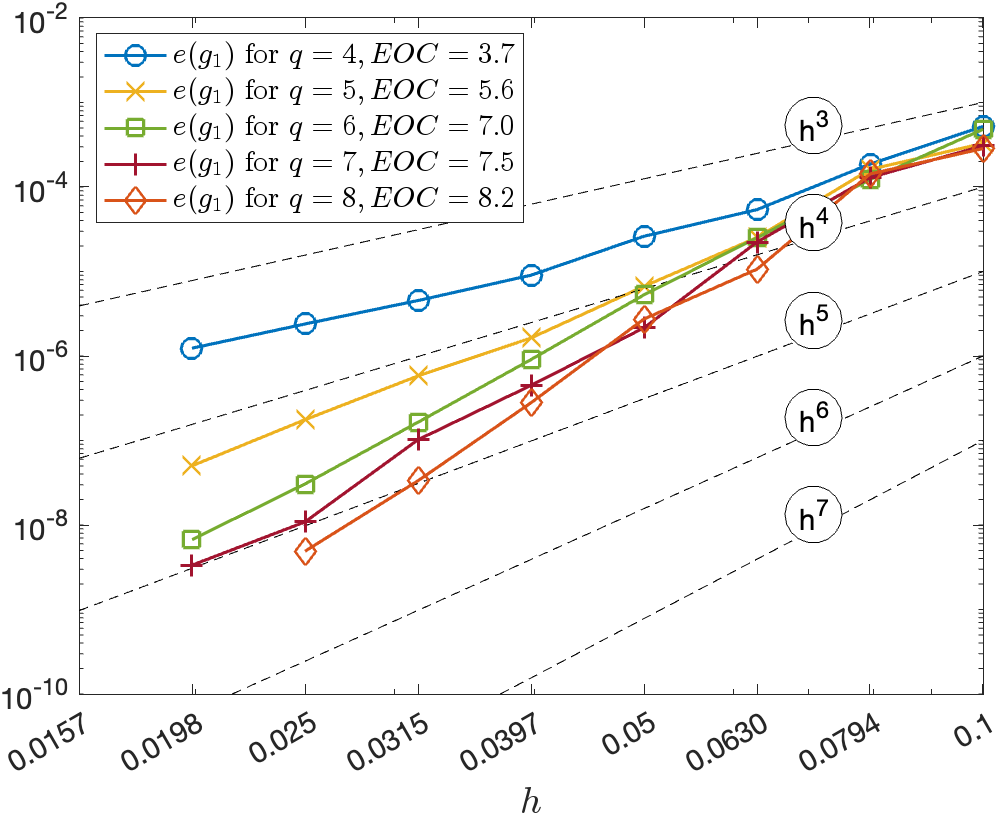}
} \hfill
\subcaptionbox{Errors for $g_1$ with BSP weights.}{
	\includegraphics[width=0.47\textwidth]{
	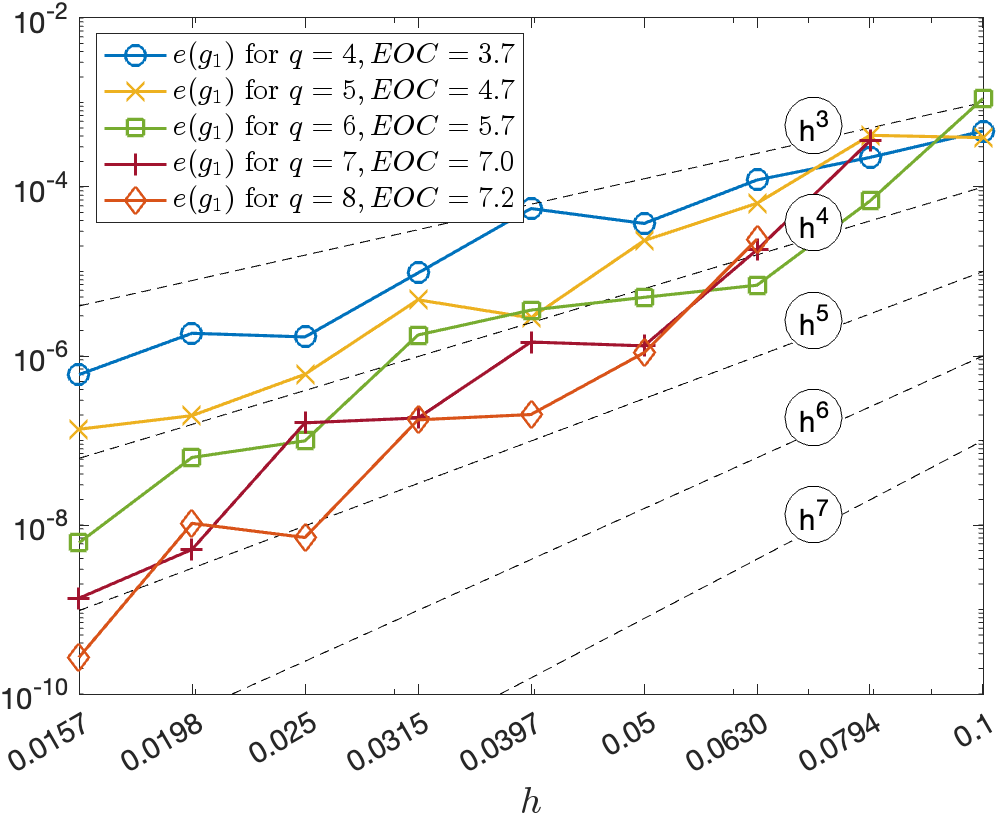}
} \\[3em]
\subcaptionbox{Errors for $f_2$ with MFD weights.}{
	\includegraphics[width=0.47\textwidth]{
	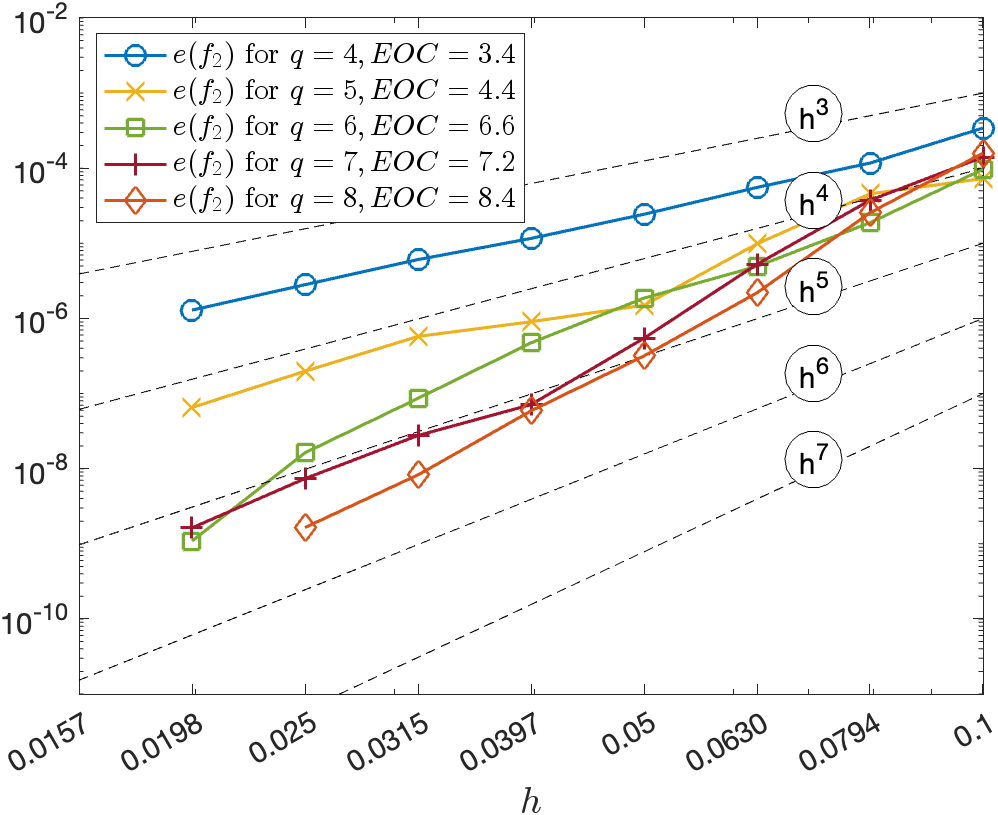}
} \hfill
\subcaptionbox{Errors for $f_2$ with BSP weights.\label{subfig:saturation}}{
	\includegraphics[width=0.47\textwidth]{
	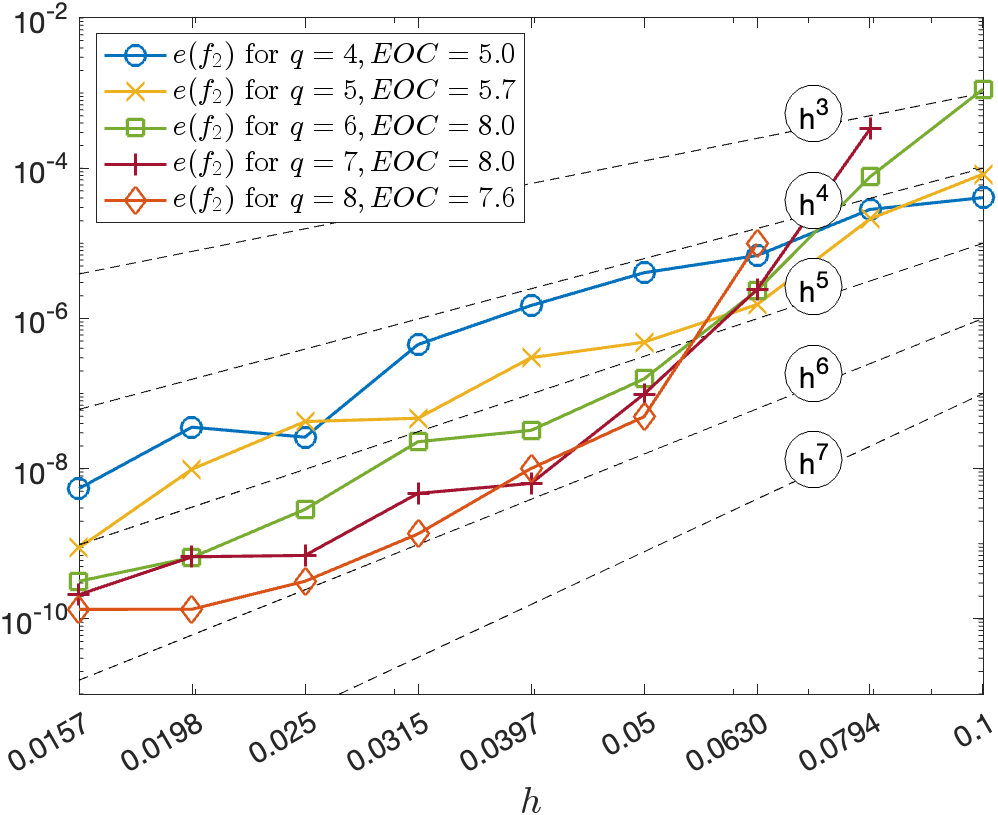}
} \\[1em]
\caption{RMS relative quadrature errors on 3D domain
$\Omega_6$ (torus) and its boundary, averaged over
8 distinct seeds and plotted as a function of $h$.
Nodes are generated with an advancing front method.
Runge function $f_1$ is centered at $x_R = (1,0,0)$.
Results of Algorithms~\ref{algo:mfd} (MFD) and \ref{algo:spline} (BSP)
are compared for polynomial orders $q=4,\ldots,8$.
Reference slopes of convergence order $q-1$ are shown
using dotted lines. Estimated order of convergence EOC
is obtained by linear least-squares fitting.}
\label{fig:torus}
\end{figure}

\begin{figure}[htbp!]
\centering
\subcaptionbox{Maximum MFD stability constants on 2D domains.}{
	\includegraphics[width=0.23\textwidth]{
	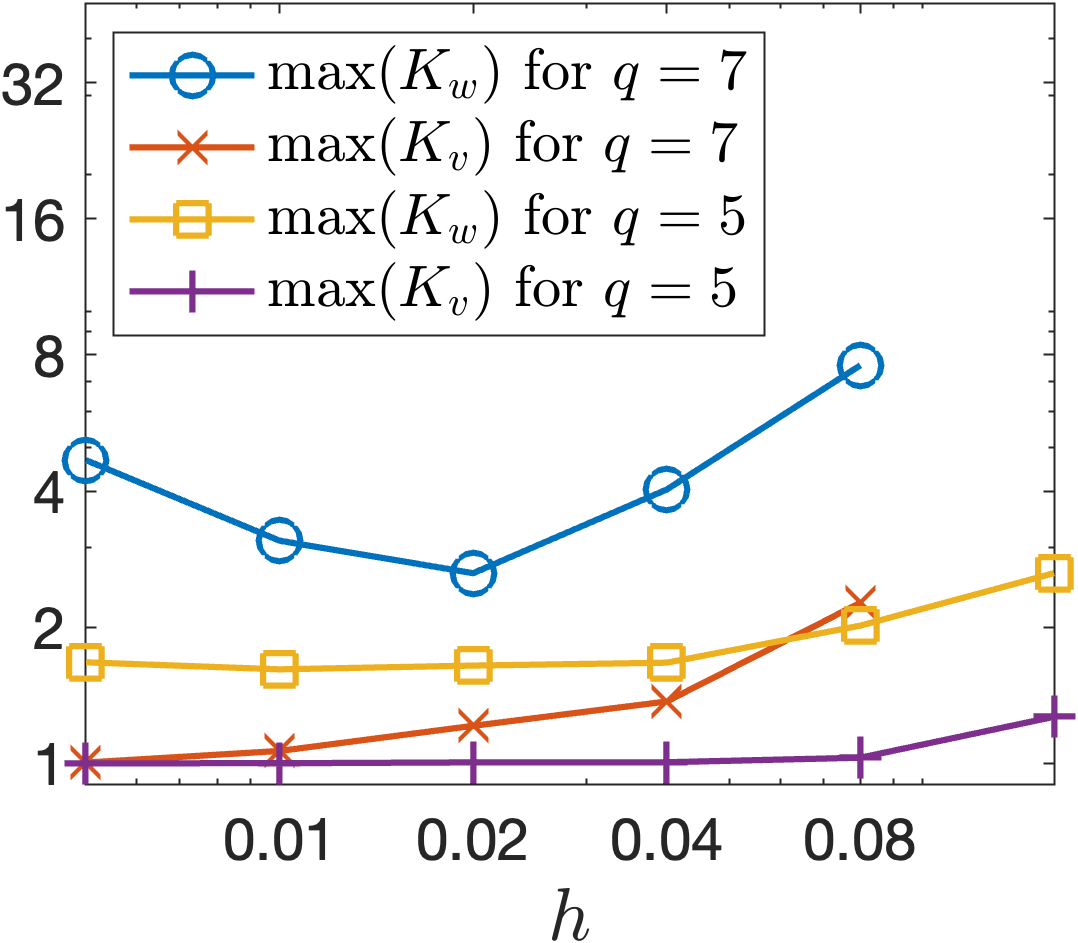}
} \hfill
\subcaptionbox{Maximum MFD stability constants on 3D domains.}{
	\includegraphics[width=0.23\textwidth]{
	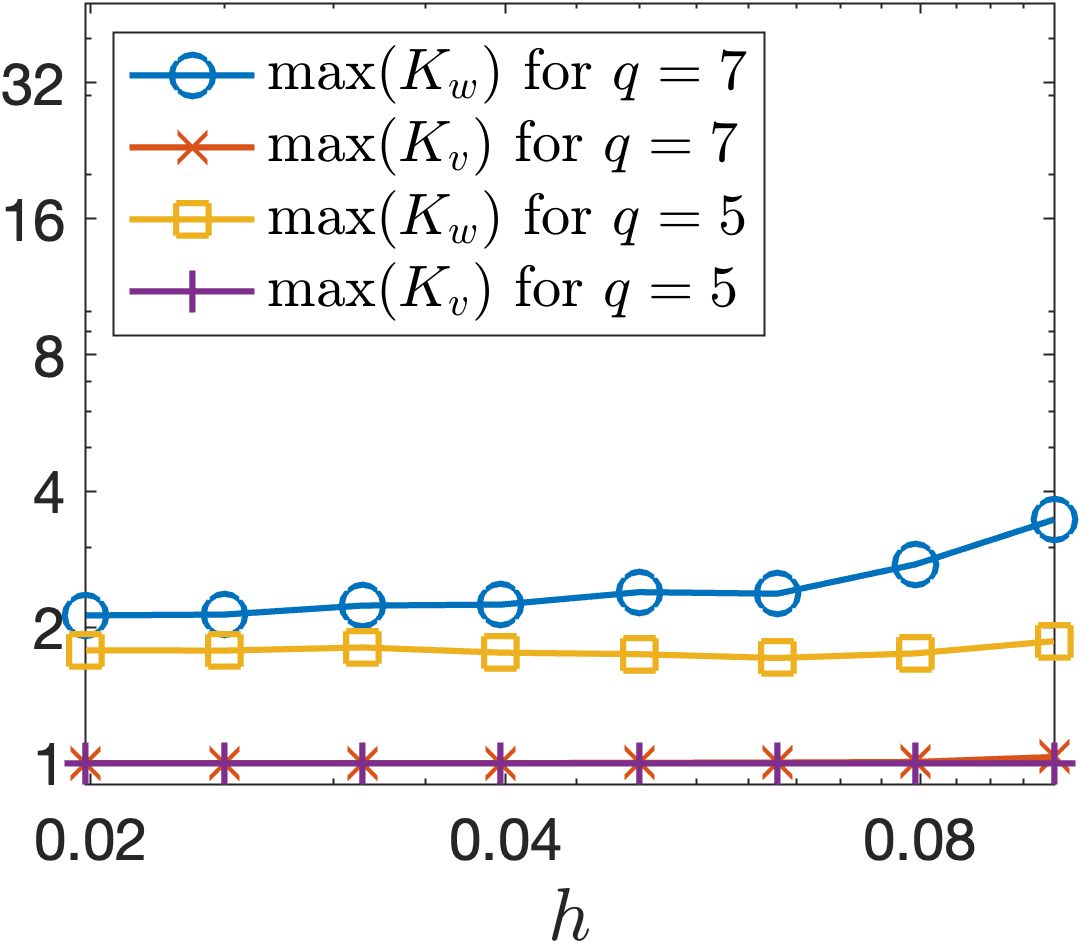}
} \hfill
\subcaptionbox{Maximum BSP stability constants on 2D domains.}{
	\includegraphics[width=0.23\textwidth]{
	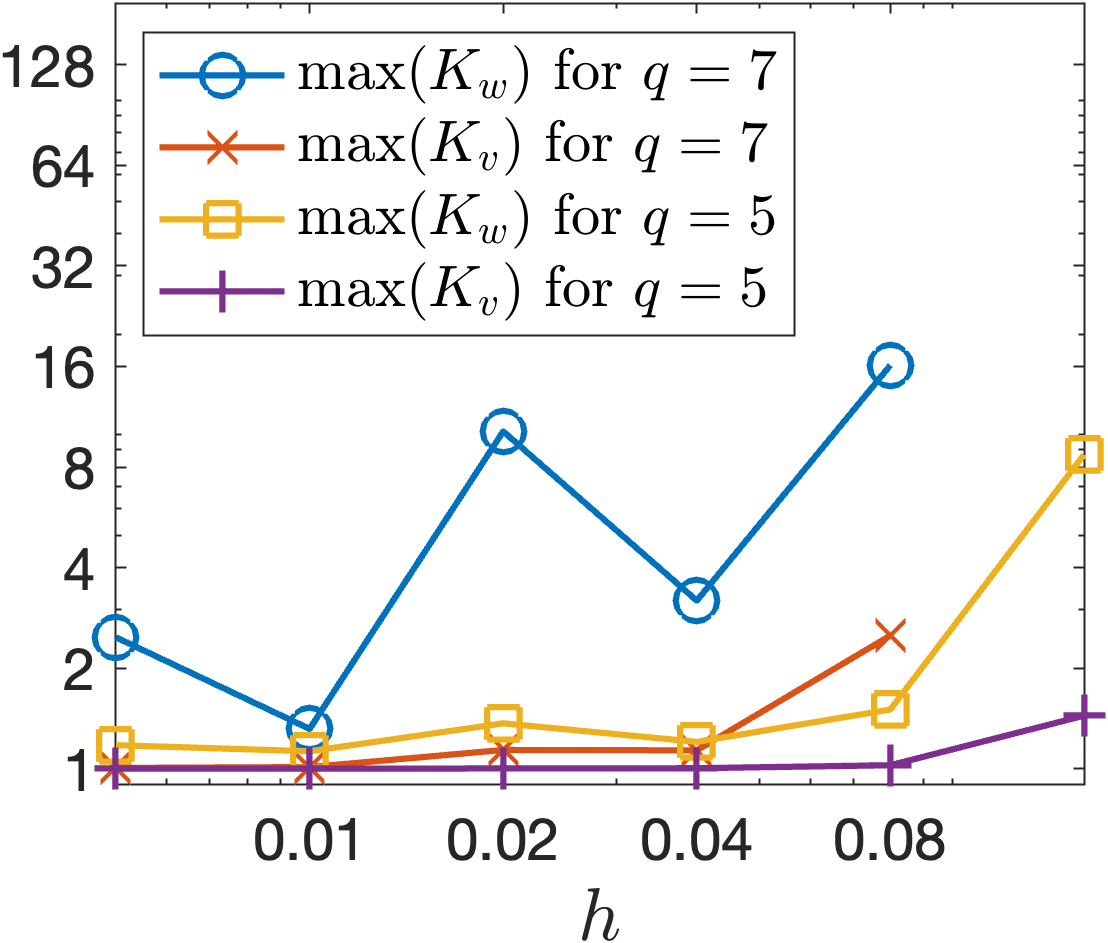}
} \hfill
\subcaptionbox{Maximum BSP stability constants on 3D domains.}{
	\includegraphics[width=0.23\textwidth]{
	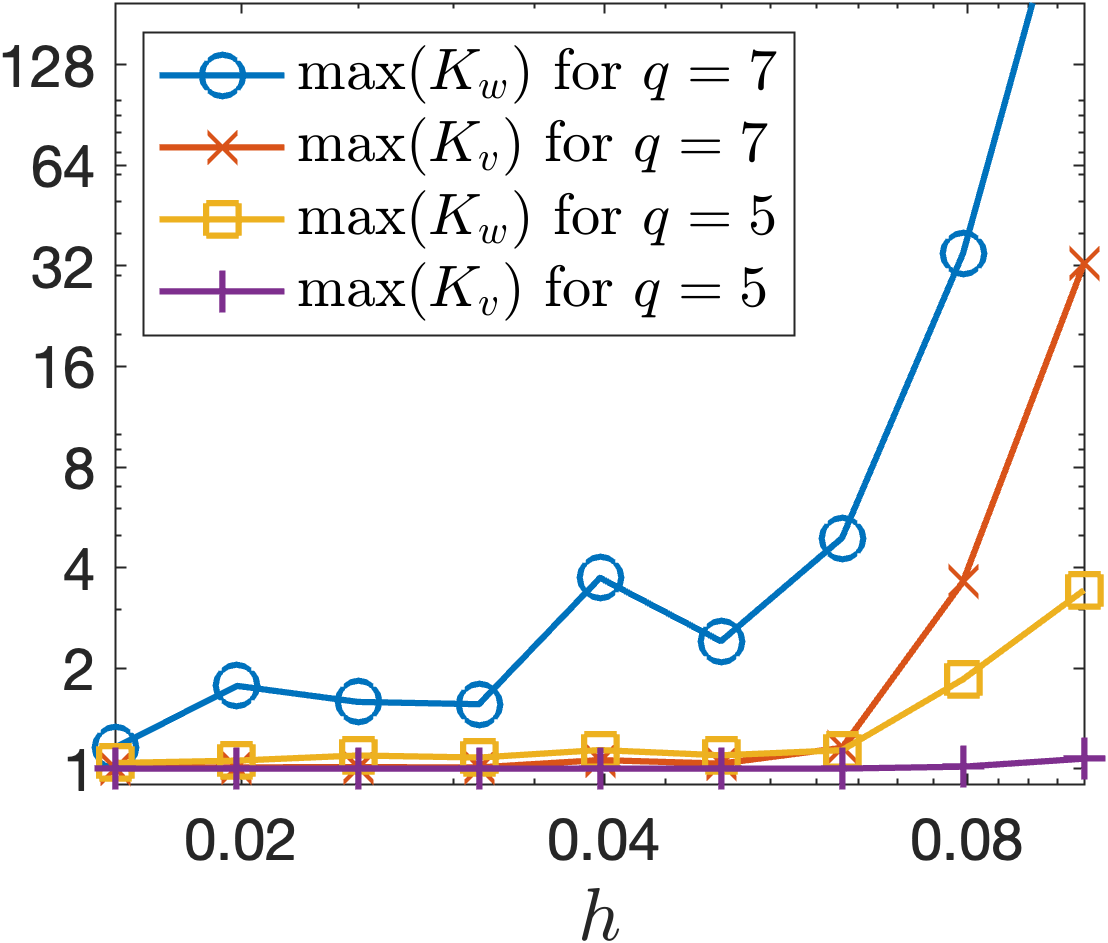}
}
\caption{Maximum values of the normalized
stability constants $K_w = \norm{w}_1/\abs{\Omega}$
and $K_v = \norm{v}_1/\abs{\partial\Omega}$ for the MFD
and BSP convergence tests with orders $q=5$ and $q=7$,
plotted as a function of $h$. The maximum is computed
across all domains $\Omega_2,\Omega_3,\Omega_5,\Omega_6$,
and all seeds $n_s = 1,\dots,8$.}
\label{fig:maxstab}
\end{figure}

\subsection{Comparison with other quadrature formulas}
\label{ssec:comp}

Now that the formulas produced
by Algorithms~\ref{algo:mfd} and \ref{algo:spline}
have been shown to be stable and convergent
with high order with respect to $h$,
we compare the errors of our meshless moment-free
quadrature to some well-established
methods  for numerical integration on two domains in 3D:
the solid torus $\Omega_6$ defined previously,
and a deco-tetrahedron
\begin{equation} \label{eq:scaledls}
\Omega_7=\Bigl\{
(x_1,x_2,x_3)\in \R^3 : \varphi(ax_1,ax_2,ax_3)+15<0\},\qquad a=40^{1/3},
\end{equation}
where
\begin{align}\label{eq:dtls}
\begin{split} 
\varphi(x_1,x_2,x_3)=
(x_1-2)^2 (x_1+2)^2 + (x_2-2)^2 (x_2+2)^2 + (x_3-2)^2 (x_3+2)^2 \\
+ 3(x_1^2 x_2^2 + x_2^2 x_3^2 + x_3^2 x_1^2) + 6 x_1 x_2 x_3
-10(x_1^2 + x_2^2 + x_3^2). 
\end{split}
\end{align}
The value of $a$ was chosen to make the volumes
of $\Omega_6$ and $\Omega_7$ approximately equal.
Note that the inclusion of $a$ in \eqref{eq:scaledls}
has the effect of scaling the domain
$\{\varphi(x_1,x_2,x_3)+15<0\}$
by a factor of $a^{-1}$ in every direction.
The domain $\Omega_7$ is shown in Figure~\ref{subfig:decotet}.
It is bounded by the closed quartic surface of genus 3
given by a level set of $\varphi(ax_1,ax_2,ax_3)$. The name
``deco-tetrahedron''
stems from \cite{PalaisKarcher2023},
and the particular quartic level set function~\eqref{eq:dtls}
has been shared by L.~Lampret on Mathematics Stack
Exchange after appearing in~\cite{Lampret}.

The boundary of the torus $\Omega_6$ is known in parametric form,
and it also admits the closed-form implicit representation
\[
\Bigl\{
(x_1,x_2,x_3) \in \R^3 :
\Bigl(\sqrt{x_1^2+x_2^2}-1\Bigr)^2 + x_3^2 - (0.32)^2 < 0
\Bigr\}.
\]

We compare the accuracy of our quadrature formulas for the
integral of the Renka function $f_2$ over $\Omega_6$
and $\Omega_7$ with three methods that also generate
quadrature formulas for 3D domains and are
available as open-source software.

\begin{itemize}
\item \textbf{Gmsh} is the quadrature method supplied with the
popular mesh generation software Gmsh \cite{geuzaine2009gmsh}. We have chosen it among
 publicly available meshing software because it supports high-order
meshing beyond degree two for 3D domains with parametric boundary.
 
\item \textbf{MFEM-MKO} is provided by MFEM,
an open-source C++ library for finite element
methods~\cite{anderson2021mfem}. This method 
is available for domains defined implicitly via a level set function. 
It is a reimplementation
of the scheme introduced in \cite{muller2013highly}.

\item \textbf{MFEM-Algoim} is another quadrature scheme
provided by MFEM for implicit 3D domains.
It relies on Algoim, the author's implementation
of the method presented in \cite{saye2015high}.

\end{itemize}

Figure~\ref{fig:comp} shows a comparison of relative
interior quadrature errors for the Renka function $f_2$
on the torus and on the deco-tetrahedron across four different
quadrature schemes: the BSP approach in Algorithm~\ref{algo:spline},
compound quadrature formulas on the curved
elements generated by Gmsh,
and the implicit methods MFEM-MKO and MFEM-Algoim.
To make the comparison meaningful, we plot the errors as a
function of $N_Y$, the total number of interior quadrature nodes.
Since the integral of $f_2$
cannot be computed on $\Omega_7$ with the approach
described at the beginning of Section~\ref{sec:numerical-tests},
we have used as a reference in Figure~\ref{subfig:decotet}
the value produced by
MFEM-Algoim with degree 15 and $N_Y \approx 5 \cdot 10^8$,
a choice clearly justified a posteriori by looking at the decay
in the errors for this method.

\begin{figure}[htbp!]
\centering
\subcaptionbox{Errors for $f_2$ on torus.}{
	\includegraphics[width=0.47\textwidth]{
	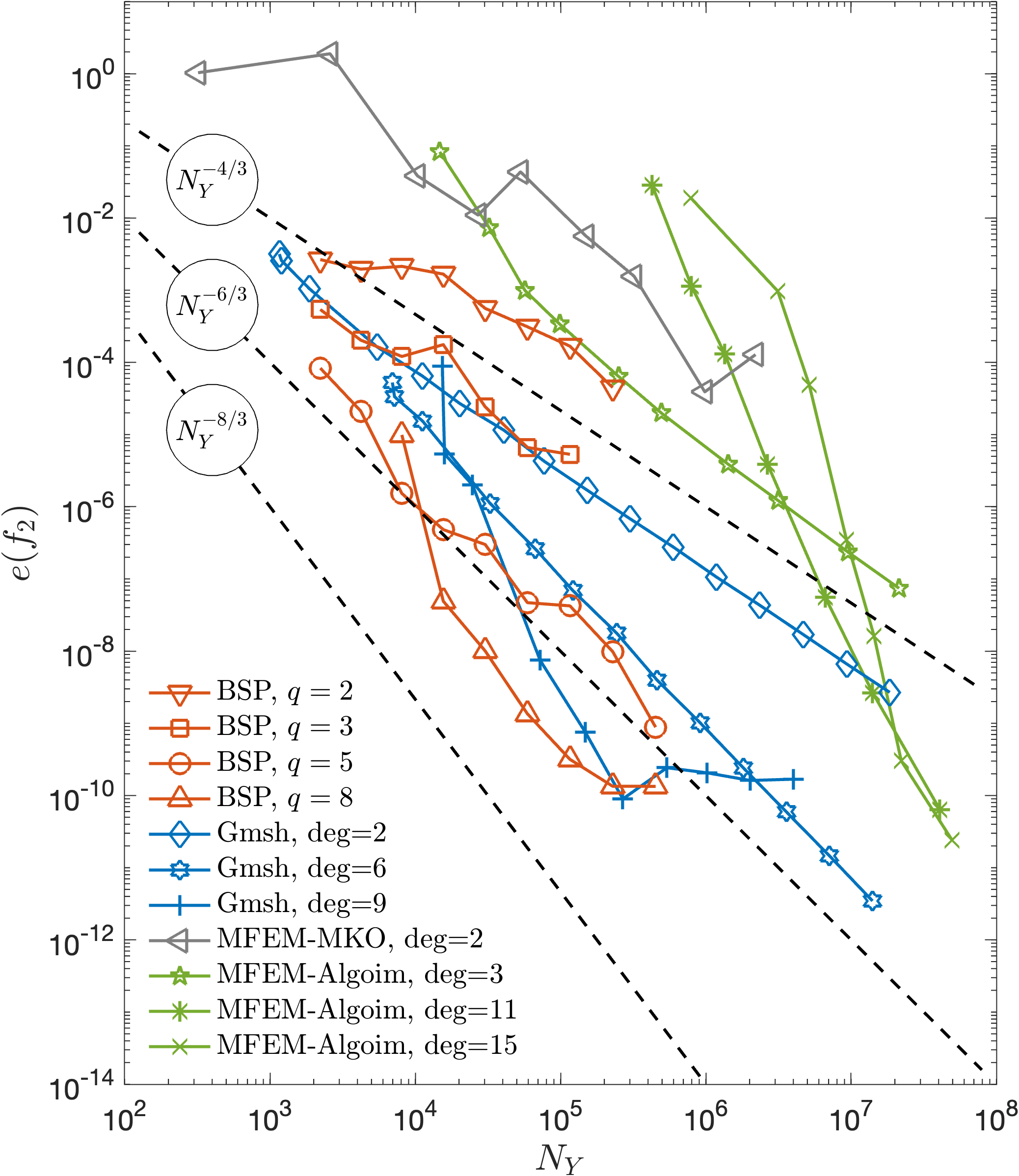}
} \hfill
\subcaptionbox{Errors for $f_2$ on deco-tetrahedron.}{
	\includegraphics[width=0.47\textwidth]{
	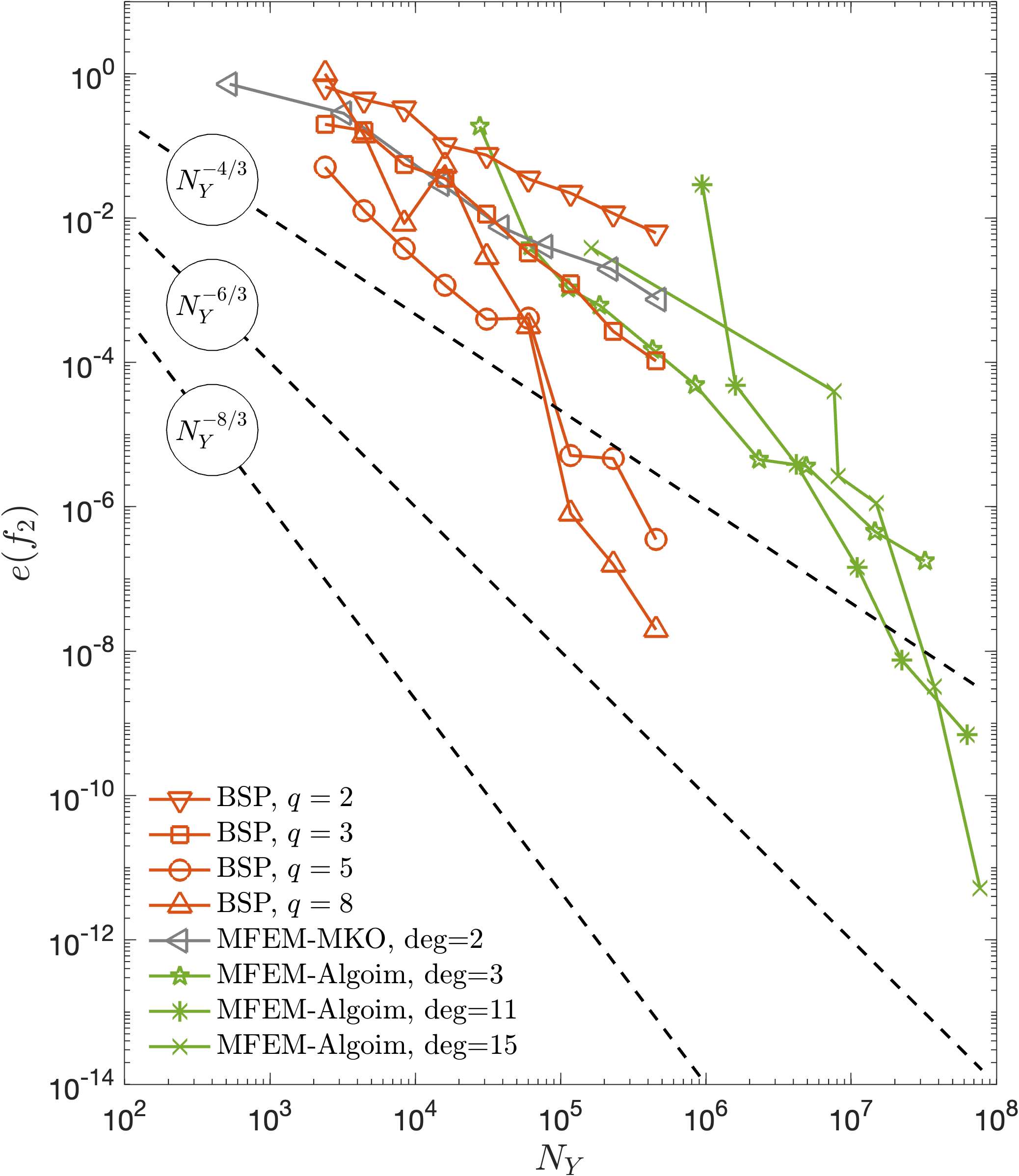}
} \\[1em]
\caption{Comparison of BSP algorithm with
state-of-the-art quadrature schemes on 3D domains
$\Omega_6$ (torus) and $\Omega_7$ (deco-tetrahedron).
Errors for the Renka function $f_2$
are relative, plotted as a function
of quadrature nodes $N_Y$, and averaged over 8 seeds
(BSP on nodes generated by an advancing
front method), or 8 shifts of the ambient
Cartesian grid (MFEM).
Reference slopes proportional to
powers 4,6,8 of the packing spacing $h_{ps}$
defined in \eqref{packsp} are shown
using dotted lines.}
\label{fig:comp}
\end{figure}

The settings and the errors for the BSP approach
on the torus are the same as in Figure~\ref{fig:torus}.
In particular, the same values of $h$ are used,
and errors are averaged over 8 seeds.
For the sake of comparison with lower-order
methods, we have added the results for $q=2$ and $q=3$.
On the deco-tetrahedron, all parameters
are also the same as in the tests of Section~\ref{ssec:conv},
with just two significant changes. First, only the set
of interior nodes $Y$ is generated by an advancing front
method, because a parametric description of the boundary
is not available. The set of boundary nodes $Z$ is
generated by a combination of rejection sampling of Halton
nodes, projection and thinning, as described in
Subsection~\ref{sssec:choice-of-nodes}.
As usual, the average spacing of nodes
in the two sets are matched, and $Z \subset Y$.
The second difference is that neither the volume nor
the surface area of $\Omega_7$ are known ahead of time.
For this reason, we use a purely moment-free
approach based on the fundamental solution, see
sections~\ref{ssec:on-the-choice-of-fhat-and-ghat}
and \ref{sssec:choice-fghat}. The fundamental solution
is centered at the point $(a^{-1},-2.5a^{-1},a^{-1})
\in \Omega_7$.

Gmsh only supports domains whose boundary is given
in parametric form, such as boundary representations
commonly used in computer-aided design.
For this reason, it is not included in the comparison
for the deco-tetrahedron.
Gmsh represents high-order tetrahedral meshes
by composing linear elements with a polynomial mapping
of arbitrary degree.
In our experiments, the mappings remained
injective even for high degree, presumably because
of the simple shape of $\partial\Omega_6$.
High-order mesh optimization was therefore not needed.
All the options in Gmsh were left at their default values,
except for \texttt{Mesh.ElementOrder}, which was set equal
to the considered polynomial degree, and \texttt{Mesh.MeshSizeMax},
which was set to various spacing parameters $h$ used to
control the maximum diameter of the tetrahedra in the mesh.
For quadratic meshing, a symmetrical Gaussian formula
with four nodes and degree of exactness 2
is used on each element. For degrees 6 and 9,
we have used formulas with 24 and 53 nodes and degrees
of exactness 6 and 9, respectively, as returned by the function
\texttt{gmsh.model.mesh.getIntegrationPoints}
exposed by Gmsh's python API. No quadrature nodes
are shared by adjacent elements.

As for the algorithms MFEM-MKO and MFEM-Algoim,
both are immersed methods that take as input an
ambient mesh surrounding $\Omega$, and so their
outputs are somewhat dependent on this choice,
although not in a critical way.
For our numerical experiments, we have chosen
a uniform Cartesian grid around $\Omega$
with step size $h$, built by partitioning
the same bounding box used in the BSP algorithm.
As in~\cite{muller2013highly}, the errors were
averaged over several random shifts of the ambient mesh.
We have used 8 different displacements, each one defined
by a uniformly random vector in $[-h,h]^3$.
To ensure that $\Omega$ remains included in the
shifted bounding box, the grid is symmetrically
extended by two elements in each direction (for
a total extension of length $2h$).

In the case of MFEM-MKO,
the degree in the legend of Figure~\ref{fig:comp}
corresponds to the maximum degree of polynomials
used in the moment-fitting equations.
In the case of MFEM-Algoim, it corresponds to the degree
of exactness of the univariate Gaussian formulas
used internally in combination with a
dimension-reduction approach.
In either case, the level set function is first
approximated by MFEM with a polynomial
of the same degree on each cell of the grid.
Since the level set function for the deco-tetrahedron
is a quartic polynomial, the approximation would be exact
for degrees 6 and 9. To make the comparison more fair,
we have applied a non-polynomial
transformation to \eqref{eq:dtls}.
More precisely, the quartic equation
$\varphi(ax_1,ax_2,ax_3) + 15 = 0$
that defines the boundary of $\Omega_7$
has been transformed into
\[
\sqrt{\varphi(ax_1,ax_2,ax_3)+100} - \sqrt{85} = 0,
\]
which still defines the same domain $\Omega_7$ but the level set
function is not a polynomial. The function
$\varphi(ax_1,ax_2,ax_3)+100$ is strictly positive for
all $(x_1,x_2,x_3) \in \R^3$.

Comparing the quadrature errors in Figure~\ref{fig:comp},
we see that Algorithm~\ref{algo:spline} performs well:
the most accurate scheme in this tests for both torus and
deco-tetrahedron belongs to the
BSP family for all node counts from $3 \cdot 10^3$
to $5 \cdot 10^5$.
For the torus, the BSP method with $q=3$ has errors comparable
to quadratic meshing, whereas the plots for the higher order methods
with $q=5$ and $q=8$ have slopes comparable to
Gmsh of degrees 6 and 9, respectively, although the errors of
BSP remain in part significantly
smaller for the same number of nodes.
The accuracy of Gmsh starts degrading
for degrees larger than 9, so higher degrees
were not included in the plot. 

MFEM-MKO was not found to be competitive:
we only report the errors for degree 2 because higher
degrees do not show any improvement.

As for MFEM-Algoim, the Gaussian formulas combined
with a dimension reduction technique reliably deliver
methods of very high order, which is reflected in the steep
slopes of the error plots for deg=11 and 15. However, the recursive
subdivision process on which it is based seems to produce
extremely large numbers of quadrature nodes for a given
accuracy target, often 10 or 100 times larger than
the ones obtained with BSP or Gmsh.
It is only for $N_Y \approx 10^8$ that
MFEM-Algoim overtakes the other methods.

\section{Conclusion} \label{sec:concl}
In this paper we introduce a new method of generating high-order
quadrature formulas on user-supplied nodes for multivariate
domains and their boundaries, based on the numerical
approximation of the divergence operator. The key idea of the
method is to seek a simultaneous approximation of the integrals
over the domain $\Omega$ and its boundary $\partial\Omega$ by
exploiting the connection between these two integrals provided
by the divergence theorem. We develop an error analysis
in a rather general setting based on the discretization
of a boundary value problem for the divergence operator.
Numerical experiments demonstrate the
robustness and high accuracy of the method on smooth and
piecewise smooth domains in 2D and 3D.

A key improvement in comparison to existing approaches for
user-supplied nodes is that the method is \emph{moment-free},
that is, the precomputation of moments, i.e.\ integrals
of some basis functions over $\Omega$ or $\partial\Omega$,
is not required. This makes our scheme easy to implement even
for complicated domains, independently from any specific
representation, such as parametric or implicit forms. In addition
to the positions of the quadrature nodes, only the directions of
outward-pointing normals at each node on $\partial\Omega$ are
needed. Quadrature weights are computed as a minimum-norm
solution to an underdetermined sparse linear system, which allows
application of efficient sparse linear algebra routines.

In the setting where nodes are to be generated for a given domain,
our method benefits from its meshless character, because it avoids
problematic high-order mesh generation otherwise
needed when the standard approach of a compound quadrature
formula is applied to a complicated 3D domain.
Instead, we suggest node generation by the meshless
advancing front method, and successfully test it numerically.
Moreover, in Section~\ref{sssec:choice-of-nodes} we show that the
method is robust with respect to the placement of the nodes,
performing reasonably on random and quasi-random nodes.

In Section~\ref{ssec:palgos} we described two particular practical
algorithms to compute quadrature weights by our method.
Although the algorithms share a common framework using numerical 
differentiation, the first one (MFD) is based on meshless
finite differences formulas, while the second (BSP) on
collocation with unfitted tensor-product B-splines
in a bounding box around $\Omega$.
We demonstrated in Section~\ref{ssec:conv} the effectiveness
of both algorithms by extensive numerical tests on a number of
smooth and piecewise smooth domains in 2D and 3D, including some
with reentrant corners and edges.
Furthermore, we demonstrated in Section~\ref{ssec:comp}
that the accuracy of our methods can exceed that
of state-of-the-art quadrature schemes
provided by the packages Gmsh and MFEM,
if the errors are compared for the same
number of interior quadrature nodes.
Our quadrature formulas were found to be of high convergence order
(above $h^7$ in numerical tests), and to possess small
stability constants, although not positive in general.
Performance of the two algorithms, MFD and BSP,
was found to be comparable, with no clear winner overall.
In both cases the convergence order matches
that of numerical differentiation,
as expected from our error analysis.
Variants of these algorithms, including those minimizing
the 1-norm instead of the 2-norm, or discretizing the
Laplacian instead of the divergence operator,
were also tested, but found to be unsuitable for either large sets
of quadrature nodes or nonsmooth domains, respectively.

While the results of this paper highlight the effectiveness of
the new approach for numerical integration on 2D and 3D domains, 
there are important directions for further research. 

In particular,
the computation of the quadrature weights may be made even more 
efficient by developing an effective preconditioning
strategy for the iterative solution of the linear system,
or by partitioning $\Omega$, so that computations can be localized 
and carried out in parallel, without challenging high-order 
reparametrization of the subdomains needed in the mesh-based 
compound quadrature formulas.
Advanced methods for the selection of sets
of influence in the MFD approach also have the potential
to improve performance by enhancing conditioning and
sparsity of the resulting linear system.

Even if the normalized stability constants
$\norm{w}_1/\abs{\Omega}$ and $\norm{v}_1/\abs{\partial\Omega}$ 
are small in our experiments,
stability of $w$ and $v$ remains an empirical observation,
and so an enhanced error analysis that guarantees small
bounds on $\norm{w}_1$ and $\norm{v}_1$ under appropriate
restrictions on the node sets $Y$ and $Z$, possibly after
some modification of the current algorithms, is highly desirable.
Modification of the algorithms
in order to generate positive quadrature weights would also be of
high interest for certain applications, provided that
the methods remain efficient.
Further analysis is needed in order to relax the suboptimal
assumptions on $f$ and $g$ in the error bounds  \eqref{eq:fb},
\eqref{eq:gb}, \eqref{eq:fbBS} and \eqref{eq:gbBS}, and extend
them to the divergence approach on piecewise smooth domains, 
where a major obstacle is the lack of research on high order
regularity of the boundary value problem
\eqref{eq:bvp-divergence} in this setting.

The flexibility of the proposed method makes it a promising tool
for the development of adaptive algorithms for the
integration of functions with sharp features or singularities,
and integration on non-Lipschitz domains, by
using variable density and local refinement of nodes.

Finally, investigations on the choice of node generation, numerical
differentiation methods and further ingredients of the general 
framework presented in Section~\ref{ssec:algorithm}
are needed in order to develop effective setups for the
integration over domains on surfaces and manifolds,
3D CAD domains in B-rep form, and for specific tasks
such as integration of cut elements in immersed methods
or implementation of meshless methods for PDEs in weak form.

\section*{Acknowledgments}
Bruno Degli Esposti is member
of the Gruppo Nazionale Calcolo Scientifico - Istituto
Nazionale di Alta Matematica (GNCS-INdAM).
The INdAM support through GNCS (CUP E53C23001670001)
is gratefully acknowledged.

\bibliographystyle{abbrv} 
\bibliography{bibliography,bibliography_oleg}

\end{document}